\documentclass[11pt,letterpaper]{amsart}

\usepackage[english]{babel}
\usepackage[utf8]{inputenc}
\usepackage{amsmath}
\usepackage{graphicx}
\usepackage[colorinlistoftodos]{todonotes}
\usepackage{amssymb,amscd,latexsym,epsfig,xy}
\usepackage[mathscr]{euscript}
\usepackage{amsthm}
\usepackage{tikz}
\usepackage{MnSymbol}
\usetikzlibrary{matrix,arrows,decorations.pathmorphing}
\usepackage{hyperref}
\usepackage{enumerate}
\usepackage{float}
\usepackage{tikz-cd}
\usepackage{relsize}
\usepackage{mathtools}
\usepackage{dsfont}
\usepackage{cleveref}
\usepackage{cite}

\crefformat{section}{\S#2#1#3} 
\crefformat{subsection}{\S#2#1#3}
\crefformat{subsubsection}{\S#2#1#3}

\newtheorem{defn}{Definition}[section]
\newtheorem{thm}[defn]{Theorem}
\newtheorem{lem}[defn]{Lemma}

\newtheorem{conj}[defn]{Conjecture}
\newtheorem{prop}[defn]{Proposition}

\newcommand\A{\mathbb A}
\newcommand\C{\mathbb C}
\newcommand\E{\mathbb E}

\newcommand\PP{\mathbb P}
\newcommand\R{\mathbb R}
\newcommand\Q{\mathbb Q}

\newcommand\Z{\mathbb Z}

\newcommand\cC{\mathcal C}

\newcommand\cM{\mathcal M}

\newcommand\cO{\mathcal O}
\newcommand\cP{\mathcal P}

\newcommand\cX{\mathcal X}

\newcommand\fd{\mathfrak d}
\newcommand\fD{\mathfrak D}

\newcommand\fM{\mathfrak M}

\newcommand\Arg{\operatorname{Arg}}

\newcommand\Coh{\operatorname{Coh}}

\newcommand\Ext{\operatorname{Ext}}

\newcommand\Hom{\operatorname{Hom}}

\newcommand\Rea{\operatorname{Re}\,}

\newcommand\Ima{\operatorname{Im}\,}

\newcommand\ch{\operatorname{ch}}
\newcommand\Pic{\operatorname{Pic}}
\newcommand\cl{\operatorname{cl}}

\newcommand\ev{\operatorname{ev}}

\newcommand\Stab{\operatorname{Stab}}
\newcommand\D{\operatorname{D}}

\newcommand\ord{\operatorname{ord}}

\newcommand\virt{\mathrm{virt}}

\newcommand\iin{\mathrm{in}}

\title[A proof of N.\ Takahashi's conjecture]{A proof of N.\ Takahashi's conjecture for 
$(\PP^2,E)$ and a refined sheaves/Gromov-Witten correspondence }

\author{Pierrick Bousseau}

\date{}

\begin{document}

\begin{abstract}
We prove N.\ Takahashi's conjecture determining the contribution of each contact point in genus-$0$ maximal contact Gromov--Witten theory of $\PP^2$ relative to a smooth cubic $E$. 
This is a new example of a question in Gromov--Witten theory that can be fully solved despite the presence of
contracted components and multiple covers.
The proof relies on a tropical computation of the Gromov--Witten invariants and on the interpretation of the tropical picture as describing wall-crossing in the derived category of coherent sheaves on $\PP^2$.
 
The same techniques allow us to prove a new sheaves/Gromov--Witten correspondence, relating Betti numbers of moduli spaces of one-dimensional Gieseker semistable sheaves on $\PP^2$, or equivalently refined genus-$0$ Gopakumar--Vafa invariants of local $\PP^2$,
with higher-genus maximal contact Gromov--Witten theory of $(\PP^2,E)$. The correspondence involves the non-trivial change of variables $y=e^{i \hbar}$, where $y$ is the refined/cohomological variable on the sheaf side, and $\hbar$ is the genus variable on the Gromov--Witten side. We explain how this correspondence can be heuristically motivated by a combination of mirror symmetry and hyperkähler rotation.
\end{abstract}

\maketitle

\setcounter{tocdepth}{1}


\tableofcontents

\thispagestyle{empty}

\newpage

\begin{figure}[h!]\label{figure_scattering}
\centering
\resizebox{0.90\textwidth}{0.90\textheight}{
\rotatebox{90}{
\begin{tikzpicture}[xscale=0.6,yscale=2.7,font=\fontsize{6}{6},define rgb/.code={\definecolor{mycolor}{RGB}{#1}}, rgb color/.style={define rgb={#1},mycolor}]
\draw[->,rgb color={255,132,0}] (-30.0,-4.00) -- (-51.0,-6.00);
\draw[->,rgb color={255,132,0}] (-30.0,-4.00) -- (-48.0,-6.00);
\draw[->,rgb color={255,132,0}] (-30.0,-4.00) -- (-48.0,-6.00);
\draw[->,rgb color={255,132,0}] (-30.0,-4.00) -- (-48.0,-6.00);
\draw[->,rgb color={255,132,0}] (-30.0,-4.00) -- (-42.0,-6.00);
\draw[->,rgb color={255,132,0}] (-30.0,-4.00) -- (-36.0,-6.00);
\draw[->,rgb color={255,132,0}] (-30.0,-4.00) -- (7.00,-4.00) node[right]{};
\draw[->,rgb color={255,132,0}] (-30.0,-4.00) -- (7.00,-4.00) node[right]{};
\draw[->,rgb color={255,132,0}] (-30.0,-4.00) -- (7.00,-4.00) node[right]{};
\draw[->,rgb color={255,132,0}] (-30.0,-4.00) -- (7.00,-4.00) node[right]{};
\draw[->,rgb color={255,132,0}] (-30.0,-4.00) -- (7.00,-4.00) node[right]{};
\draw[->,rgb color={255,132,0}] (-30.0,-4.00) -- (7.00,-4.00) node[right]{};
\draw[->,rgb color={255,132,0}] (-27.0,-4.00) -- (7.00,-4.00) node[right]{};
\draw[->,rgb color={255,132,0}] (-27.0,-4.00) -- (7.00,-4.00) node[right]{};
\draw[->,rgb color={255,132,0}] (-30.0,-4.00) -- (7.00,-4.00) node[right]{};
\draw[->,rgb color={255,132,0}] (-30.0,-4.00) -- (7.00,-1.94);
\draw[->,rgb color={255,132,0}] (-30.0,-4.00) -- (7.00,-1.94);
\draw[->,rgb color={255,132,0}] (-30.0,-4.00) -- (7.00,-1.94);
\draw[->,rgb color={255,132,0}] (-30.0,-4.00) -- (7.00,-2.24);
\draw[->,rgb color={255,132,0}] (-30.0,-4.00) -- (7.00,-2.46);
\draw[->,rgb color={255,132,0}] (-30.0,-4.00) -- (7.00,-1.76);
\draw[->,rgb color={255,132,0}] (-30.0,5.00) -- (-51.0,7.00);
\draw[->,rgb color={255,132,0}] (-30.0,5.00) -- (-48.0,7.00);
\draw[->,rgb color={255,132,0}] (-30.0,5.00) -- (-48.0,7.00);
\draw[->,rgb color={255,132,0}] (-30.0,5.00) -- (-48.0,7.00);
\draw[->,rgb color={255,132,0}] (-30.0,5.00) -- (-42.0,7.00);
\draw[->,rgb color={255,132,0}] (-30.0,5.00) -- (-36.0,7.00);
\draw[->,rgb color={255,132,0}] (-30.0,5.00) -- (7.00,5.00) node[right]{};
\draw[->,rgb color={255,132,0}] (-30.0,5.00) -- (7.00,5.00) node[right]{};
\draw[->,rgb color={255,132,0}] (-30.0,5.00) -- (7.00,5.00) node[right]{};
\draw[->,rgb color={255,132,0}] (-30.0,5.00) -- (7.00,5.00) node[right]{};
\draw[->,rgb color={255,132,0}] (-30.0,5.00) -- (7.00,5.00) node[right]{};
\draw[->,rgb color={255,132,0}] (-30.0,5.00) -- (7.00,5.00) node[right]{};
\draw[->,rgb color={255,132,0}] (-27.0,5.00) -- (7.00,5.00) node[right]{};
\draw[->,rgb color={255,132,0}] (-27.0,5.00) -- (7.00,5.00) node[right]{};
\draw[->,rgb color={255,132,0}] (-30.0,5.00) -- (7.00,5.00) node[right]{};
\draw[->,rgb color={255,132,0}] (-30.0,5.00) -- (7.00,2.94);
\draw[->,rgb color={255,132,0}] (-30.0,5.00) -- (7.00,2.94);
\draw[->,rgb color={255,132,0}] (-30.0,5.00) -- (7.00,2.94);
\draw[->,rgb color={255,132,0}] (-30.0,5.00) -- (7.00,3.24);
\draw[->,rgb color={255,132,0}] (-30.0,5.00) -- (7.00,3.46);
\draw[->,rgb color={255,132,0}] (-30.0,5.00) -- (7.00,2.76);
\draw[->,rgb color={255,132,0}] (-27.0,-4.00) -- (-39.0,-6.00);
\draw[->,rgb color={255,132,0}] (-27.0,-4.00) -- (7.00,-2.38);
\draw[->,rgb color={255,132,0}] (-27.0,5.00) -- (-39.0,7.00);
\draw[->,rgb color={255,132,0}] (-27.0,5.00) -- (7.00,3.38);
\draw[->,rgb color={255,132,0}] (-22.5,-3.75) -- (7.00,-3.75) node[right]{};
\draw[->,rgb color={255,132,0}] (-24.8,-3.75) -- (7.00,-3.75) node[right]{};
\draw[->,rgb color={255,107,0}] (-22.5,4.75) -- (7.00,4.75) node[right]{};
\draw[->,rgb color={255,107,0}] (-24.8,4.75) -- (7.00,4.75) node[right]{};
\draw[->,rgb color={255,132,0}] (-24.0,-3.67) -- (7.00,-3.67) node[right]{};
\draw[->,rgb color={255,132,0}] (-24.0,-3.67) -- (7.00,-3.67) node[right]{};
\draw[->,rgb color={255,0,0}] (-24.0,-3.50) -- (7.00,-0.917);
\draw[->,rgb color={255,0,0}] (-24.0,4.50) -- (7.00,1.92);
\draw[->,rgb color={255,107,0}] (-24.0,4.67) -- (7.00,4.67) node[right]{};
\draw[->,rgb color={255,132,0}] (-22.5,-3.50) -- (-30.0,-6.00);
\draw[->,rgb color={255,132,0}] (-22.5,-3.50) -- (7.00,-3.50) node[right]{};
\draw[->,rgb color={255,132,0}] (-22.5,-3.50) -- (7.00,-3.50) node[right]{};
\draw[->,rgb color={255,132,0}] (-21.0,-3.50) -- (7.00,-3.50) node[right]{};
\draw[->,rgb color={255,132,0}] (-18.0,-3.50) -- (7.00,-3.50) node[right]{};
\draw[->,rgb color={255,132,0}] (-22.5,-3.50) -- (7.00,-3.50) node[right]{};
\draw[->,rgb color={255,132,0}] (-22.5,-3.50) -- (7.00,-2.10);
\draw[->,rgb color={255,107,0}] (-22.5,4.50) -- (-30.0,7.00);
\draw[->,rgb color={255,107,0}] (-22.5,4.50) -- (7.00,4.50) node[right]{};
\draw[->,rgb color={255,107,0}] (-22.5,4.50) -- (7.00,4.50) node[right]{};
\draw[->,rgb color={255,107,0}] (-21.0,4.50) -- (7.00,4.50) node[right]{};
\draw[->,rgb color={255,107,0}] (-18.0,4.50) -- (7.00,4.50) node[right]{};
\draw[->,rgb color={255,107,0}] (-22.5,4.50) -- (7.00,4.50) node[right]{};
\draw[->,rgb color={255,107,0}] (-22.5,4.50) -- (7.00,3.10);
\draw[->,rgb color={255,132,0}] (-20.0,-3.33) -- (7.00,-3.33) node[right]{};
\draw[->,rgb color={255,107,0}] (-20.0,4.33) -- (7.00,4.33) node[right]{};
\draw[->,rgb color={255,107,0}] (-20.0,4.33) -- (7.00,4.33) node[right]{};
\draw[->,rgb color={255,132,0}] (-16.5,-3.25) -- (7.00,-3.25) node[right]{};
\draw[->,rgb color={255,132,0}] (-18.8,-3.25) -- (7.00,-3.25) node[right]{};
\draw[->,rgb color={255,107,0}] (-16.5,4.25) -- (7.00,4.25) node[right]{};
\draw[->,rgb color={255,107,0}] (-18.8,4.25) -- (7.00,4.25) node[right]{};
\draw[->,rgb color={255,107,0}] (-18.0,-3.00) -- (-40.5,-6.00);
\draw[->,rgb color={255,107,0}] (-18.0,-3.00) -- (-36.0,-6.00);
\draw[->,rgb color={255,107,0}] (-18.0,-3.00) -- (-36.0,-6.00);
\draw[->,rgb color={255,107,0}] (-18.0,-3.00) -- (-36.0,-6.00);
\draw[->,rgb color={255,107,0}] (-18.0,-3.00) -- (-27.0,-6.00);
\draw[->,rgb color={255,107,0}] (-18.0,-3.00) -- (-18.0,-6.00);
\draw[->,rgb color={255,107,0}] (-18.0,-3.00) -- (7.00,-3.00) node[right]{};
\draw[->,rgb color={255,107,0}] (-18.0,-3.00) -- (7.00,-3.00) node[right]{};
\draw[->,rgb color={255,107,0}] (-18.0,-3.00) -- (7.00,-3.00) node[right]{};
\draw[->,rgb color={255,107,0}] (-18.0,-3.00) -- (7.00,-3.00) node[right]{};
\draw[->,rgb color={255,107,0}] (-18.0,-3.00) -- (7.00,-3.00) node[right]{};
\draw[->,rgb color={255,107,0}] (-18.0,-3.00) -- (7.00,-3.00) node[right]{};
\draw[->,rgb color={255,107,0}] (-15.0,-3.00) -- (7.00,-3.00) node[right]{};
\draw[->,rgb color={255,107,0}] (-15.0,-3.00) -- (7.00,-3.00) node[right]{};
\draw[->,rgb color={255,107,0}] (-18.0,-3.00) -- (7.00,-3.00) node[right]{};
\draw[->,rgb color={255,107,0}] (-18.0,-3.00) -- (7.00,-1.33);
\draw[->,rgb color={255,107,0}] (-18.0,-3.00) -- (7.00,-1.33);
\draw[->,rgb color={255,107,0}] (-18.0,-3.00) -- (7.00,-1.33);
\draw[->,rgb color={255,107,0}] (-18.0,-3.00) -- (7.00,-1.61);
\draw[->,rgb color={255,107,0}] (-18.0,-3.00) -- (7.00,-1.81);
\draw[->,rgb color={255,107,0}] (-18.0,-3.00) -- (7.00,-1.15);
\draw[->,rgb color={255,107,0}] (-18.0,4.00) -- (-40.5,7.00);
\draw[->,rgb color={255,107,0}] (-18.0,4.00) -- (-36.0,7.00);
\draw[->,rgb color={255,107,0}] (-18.0,4.00) -- (-36.0,7.00);
\draw[->,rgb color={255,107,0}] (-18.0,4.00) -- (-36.0,7.00);
\draw[->,rgb color={255,107,0}] (-18.0,4.00) -- (-27.0,7.00);
\draw[->,rgb color={255,107,0}] (-18.0,4.00) -- (-18.0,7.00);
\draw[->,rgb color={255,107,0}] (-18.0,4.00) -- (7.00,4.00) node[right]{};
\draw[->,rgb color={255,107,0}] (-18.0,4.00) -- (7.00,4.00) node[right]{};
\draw[->,rgb color={255,107,0}] (-18.0,4.00) -- (7.00,4.00) node[right]{};
\draw[->,rgb color={255,107,0}] (-18.0,4.00) -- (7.00,4.00) node[right]{};
\draw[->,rgb color={255,107,0}] (-18.0,4.00) -- (7.00,4.00) node[right]{};
\draw[->,rgb color={255,107,0}] (-18.0,4.00) -- (7.00,4.00) node[right]{};
\draw[->,rgb color={255,107,0}] (-15.0,4.00) -- (7.00,4.00) node[right]{};
\draw[->,rgb color={255,107,0}] (-15.0,4.00) -- (7.00,4.00) node[right]{};
\draw[->,rgb color={255,107,0}] (-18.0,4.00) -- (7.00,4.00) node[right]{};
\draw[->,rgb color={255,107,0}] (-18.0,4.00) -- (7.00,2.33);
\draw[->,rgb color={255,107,0}] (-18.0,4.00) -- (7.00,2.33);
\draw[->,rgb color={255,107,0}] (-18.0,4.00) -- (7.00,2.33);
\draw[->,rgb color={255,107,0}] (-18.0,4.00) -- (7.00,2.61);
\draw[->,rgb color={255,107,0}] (-18.0,4.00) -- (7.00,2.81);
\draw[->,rgb color={255,107,0}] (-18.0,4.00) -- (7.00,2.15);
\draw[->,rgb color={255,107,0}] (-15.0,-3.00) -- (-24.0,-6.00);
\draw[->,rgb color={255,107,0}] (-15.0,-3.00) -- (7.00,-1.78);
\draw[->,rgb color={255,107,0}] (-15.0,4.00) -- (-24.0,7.00);
\draw[->,rgb color={255,107,0}] (-15.0,4.00) -- (7.00,2.78);
\draw[->,rgb color={255,107,0}] (-11.2,-2.75) -- (7.00,-2.75) node[right]{};
\draw[->,rgb color={255,107,0}] (-13.5,-2.75) -- (7.00,-2.75) node[right]{};
\draw[->,rgb color={255,0,0}] (-13.5,-2.50) -- (-45.0,-6.00);
\draw[->,rgb color={255,0,0}] (-13.5,-2.50) -- (7.00,-0.222);
\draw[->,rgb color={255,0,0}] (-13.5,3.50) -- (-45.0,7.00);
\draw[->,rgb color={255,0,0}] (-13.5,3.50) -- (7.00,1.22);
\draw[->,rgb color={255,77,0}] (-11.2,3.75) -- (7.00,3.75) node[right]{};
\draw[->,rgb color={255,77,0}] (-13.5,3.75) -- (7.00,3.75) node[right]{};
\draw[->,rgb color={255,107,0}] (-13.0,-2.67) -- (7.00,-2.67) node[right]{};
\draw[->,rgb color={255,107,0}] (-13.0,-2.67) -- (7.00,-2.67) node[right]{};
\draw[->,rgb color={255,77,0}] (-13.0,3.67) -- (7.00,3.67) node[right]{};
\draw[->,rgb color={255,107,0}] (-12.0,-2.50) -- (-12.0,-6.00);
\draw[->,rgb color={255,107,0}] (-12.0,-2.50) -- (7.00,-2.50) node[right]{};
\draw[->,rgb color={255,107,0}] (-12.0,-2.50) -- (7.00,-2.50) node[right]{};
\draw[->,rgb color={255,107,0}] (-10.5,-2.50) -- (7.00,-2.50) node[right]{};
\draw[->,rgb color={255,107,0}] (-7.50,-2.50) -- (7.00,-2.50) node[right]{};
\draw[->,rgb color={255,107,0}] (-12.0,-2.50) -- (7.00,-2.50) node[right]{};
\draw[->,rgb color={255,107,0}] (-12.0,-2.50) -- (7.00,-1.44);
\draw[->,rgb color={255,77,0}] (-12.0,3.50) -- (-12.0,7.00);
\draw[->,rgb color={255,77,0}] (-12.0,3.50) -- (7.00,3.50) node[right]{};
\draw[->,rgb color={255,77,0}] (-12.0,3.50) -- (7.00,3.50) node[right]{};
\draw[->,rgb color={255,77,0}] (-10.5,3.50) -- (7.00,3.50) node[right]{};
\draw[->,rgb color={255,77,0}] (-7.50,3.50) -- (7.00,3.50) node[right]{};
\draw[->,rgb color={255,77,0}] (-12.0,3.50) -- (7.00,3.50) node[right]{};
\draw[->,rgb color={255,77,0}] (-12.0,3.50) -- (7.00,2.44);
\draw[->,rgb color={255,107,0}] (-10.0,-2.33) -- (7.00,-2.33) node[right]{};
\draw[->,rgb color={255,77,0}] (-10.0,3.33) -- (7.00,3.33) node[right]{};
\draw[->,rgb color={255,77,0}] (-10.0,3.33) -- (7.00,3.33) node[right]{};
\draw[->,rgb color={255,107,0}] (-6.75,-2.25) -- (7.00,-2.25) node[right]{};
\draw[->,rgb color={255,107,0}] (-9.00,-2.25) -- (7.00,-2.25) node[right]{};
\draw[->,rgb color={255,77,0}] (-9.00,-2.00) -- (-27.0,-6.00);
\draw[->,rgb color={255,77,0}] (-9.00,-2.00) -- (-21.0,-6.00);
\draw[->,rgb color={255,77,0}] (-9.00,-2.00) -- (-21.0,-6.00);
\draw[->,rgb color={255,77,0}] (-9.00,-2.00) -- (-21.0,-6.00);
\draw[->,rgb color={255,77,0}] (-9.00,-2.00) -- (-9.00,-6.00);
\draw[->,rgb color={255,77,0}] (-9.00,-2.00) -- (7.00,-2.00) node[right]{};
\draw[->,rgb color={255,77,0}] (-9.00,-2.00) -- (7.00,-2.00) node[right]{};
\draw[->,rgb color={255,77,0}] (-9.00,-2.00) -- (7.00,-2.00) node[right]{};
\draw[->,rgb color={255,77,0}] (-9.00,-2.00) -- (7.00,-2.00) node[right]{};
\draw[->,rgb color={255,77,0}] (-9.00,-2.00) -- (7.00,-2.00) node[right]{};
\draw[->,rgb color={255,77,0}] (-9.00,-2.00) -- (7.00,-2.00) node[right]{};
\draw[->,rgb color={255,77,0}] (-6.00,-2.00) -- (7.00,-2.00) node[right]{};
\draw[->,rgb color={255,77,0}] (-6.00,-2.00) -- (7.00,-2.00) node[right]{};
\draw[->,rgb color={255,77,0}] (-9.00,-2.00) -- (7.00,-2.00) node[right]{};
\draw[->,rgb color={255,77,0}] (-9.00,-2.00) -- (3.00,-6.00);
\draw[->,rgb color={255,77,0}] (-9.00,-2.00) -- (7.00,-0.667);
\draw[->,rgb color={255,77,0}] (-9.00,-2.00) -- (7.00,-0.667);
\draw[->,rgb color={255,77,0}] (-9.00,-2.00) -- (7.00,-0.667);
\draw[->,rgb color={255,77,0}] (-9.00,-2.00) -- (7.00,-0.933);
\draw[->,rgb color={255,77,0}] (-9.00,-2.00) -- (7.00,-1.11);
\draw[->,rgb color={255,77,0}] (-9.00,-2.00) -- (7.00,-0.476);
\draw[->,rgb color={255,77,0}] (-9.00,3.00) -- (-27.0,7.00);
\draw[->,rgb color={255,77,0}] (-9.00,3.00) -- (-21.0,7.00);
\draw[->,rgb color={255,77,0}] (-9.00,3.00) -- (-21.0,7.00);
\draw[->,rgb color={255,77,0}] (-9.00,3.00) -- (-21.0,7.00);
\draw[->,rgb color={255,77,0}] (-9.00,3.00) -- (-9.00,7.00);
\draw[->,rgb color={255,77,0}] (-9.00,3.00) -- (7.00,3.00) node[right]{};
\draw[->,rgb color={255,77,0}] (-9.00,3.00) -- (7.00,3.00) node[right]{};
\draw[->,rgb color={255,77,0}] (-9.00,3.00) -- (7.00,3.00) node[right]{};
\draw[->,rgb color={255,77,0}] (-9.00,3.00) -- (7.00,3.00) node[right]{};
\draw[->,rgb color={255,77,0}] (-9.00,3.00) -- (7.00,3.00) node[right]{};
\draw[->,rgb color={255,77,0}] (-9.00,3.00) -- (7.00,3.00) node[right]{};
\draw[->,rgb color={255,77,0}] (-6.00,3.00) -- (7.00,3.00) node[right]{};
\draw[->,rgb color={255,77,0}] (-6.00,3.00) -- (7.00,3.00) node[right]{};
\draw[->,rgb color={255,77,0}] (-9.00,3.00) -- (7.00,3.00) node[right]{};
\draw[->,rgb color={255,77,0}] (-9.00,3.00) -- (3.00,7.00);
\draw[->,rgb color={255,77,0}] (-9.00,3.00) -- (7.00,1.67);
\draw[->,rgb color={255,77,0}] (-9.00,3.00) -- (7.00,1.67);
\draw[->,rgb color={255,77,0}] (-9.00,3.00) -- (7.00,1.67);
\draw[->,rgb color={255,77,0}] (-9.00,3.00) -- (7.00,1.93);
\draw[->,rgb color={255,77,0}] (-9.00,3.00) -- (7.00,2.11);
\draw[->,rgb color={255,77,0}] (-9.00,3.00) -- (7.00,1.48);
\draw[->,rgb color={255,77,0}] (-6.75,3.25) -- (7.00,3.25) node[right]{};
\draw[->,rgb color={255,77,0}] (-9.00,3.25) -- (7.00,3.25) node[right]{};
\draw[->,rgb color={255,77,0}] (-6.00,-2.00) -- (-6.00,-6.00);
\draw[->,rgb color={255,77,0}] (-6.00,-2.00) -- (7.00,-1.13);
\draw[->,rgb color={255,0,0}] (-6.00,-1.50) -- (-33.0,-6.00);
\draw[->,rgb color={255,0,0}] (-6.00,-1.50) -- (7.00,0.667);
\draw[->,rgb color={255,0,0}] (-6.00,2.50) -- (-33.0,7.00);
\draw[->,rgb color={255,0,0}] (-6.00,2.50) -- (7.00,0.333);
\draw[->,rgb color={255,77,0}] (-6.00,3.00) -- (-6.00,7.00);
\draw[->,rgb color={255,77,0}] (-6.00,3.00) -- (7.00,2.13);
\draw[->,rgb color={255,77,0}] (-3.00,-1.75) -- (7.00,-1.75) node[right]{};
\draw[->,rgb color={255,77,0}] (-5.25,-1.75) -- (7.00,-1.75) node[right]{};
\draw[->,rgb color={255,42,0}] (-3.00,2.75) -- (7.00,2.75) node[right]{};
\draw[->,rgb color={255,42,0}] (-5.25,2.75) -- (7.00,2.75) node[right]{};
\draw[->,rgb color={255,77,0}] (-5.00,-1.67) -- (7.00,-1.67) node[right]{};
\draw[->,rgb color={255,42,0}] (-5.00,2.67) -- (7.00,2.67) node[right]{};
\draw[->,rgb color={255,77,0}] (-4.50,-1.50) -- (7.00,-1.50) node[right]{};
\draw[->,rgb color={255,77,0}] (-4.50,-1.50) -- (7.00,-1.50) node[right]{};
\draw[->,rgb color={255,77,0}] (-3.00,-1.50) -- (7.00,-1.50) node[right]{};
\draw[->,rgb color={255,77,0}] (0.000,-1.50) -- (7.00,-1.50) node[right]{};
\draw[->,rgb color={255,77,0}] (-4.50,-1.50) -- (7.00,-1.50) node[right]{};
\draw[->,rgb color={255,77,0}] (-4.50,-1.50) -- (7.00,-5.33);
\draw[->,rgb color={255,77,0}] (-4.50,-1.50) -- (7.00,-0.733);
\draw[->,rgb color={255,42,0}] (-4.50,2.50) -- (7.00,2.50) node[right]{};
\draw[->,rgb color={255,42,0}] (-4.50,2.50) -- (7.00,2.50) node[right]{};
\draw[->,rgb color={255,42,0}] (-3.00,2.50) -- (7.00,2.50) node[right]{};
\draw[->,rgb color={255,42,0}] (0.000,2.50) -- (7.00,2.50) node[right]{};
\draw[->,rgb color={255,42,0}] (-4.50,2.50) -- (7.00,2.50) node[right]{};
\draw[->,rgb color={255,42,0}] (-4.50,2.50) -- (7.00,6.33);
\draw[->,rgb color={255,42,0}] (-4.50,2.50) -- (7.00,1.73);
\draw[->,rgb color={255,77,0}] (-3.00,-1.33) -- (7.00,-1.33) node[right]{};
\draw[->,rgb color={255,42,0}] (-3.00,-1.00) -- (-10.5,-6.00);
\draw[->,rgb color={255,42,0}] (-3.00,-1.00) -- (-3.00,-6.00);
\draw[->,rgb color={255,42,0}] (-3.00,-1.00) -- (-3.00,-6.00);
\draw[->,rgb color={255,42,0}] (-3.00,-1.00) -- (-3.00,-6.00);
\draw[->,rgb color={255,42,0}] (-3.00,-1.00) -- (7.00,-1.00) node[right]{};
\draw[->,rgb color={255,42,0}] (-3.00,-1.00) -- (7.00,-1.00) node[right]{};
\draw[->,rgb color={255,42,0}] (-3.00,-1.00) -- (7.00,-1.00) node[right]{};
\draw[->,rgb color={255,42,0}] (-3.00,-1.00) -- (7.00,-1.00) node[right]{};
\draw[->,rgb color={255,42,0}] (-3.00,-1.00) -- (7.00,-1.00) node[right]{};
\draw[->,rgb color={255,42,0}] (-3.00,-1.00) -- (7.00,-1.00) node[right]{};
\draw[->,rgb color={255,42,0}] (0.000,-1.00) -- (7.00,-1.00) node[right]{};
\draw[->,rgb color={255,42,0}] (0.000,-1.00) -- (7.00,-1.00) node[right]{};
\draw[->,rgb color={255,42,0}] (-3.00,-1.00) -- (7.00,-1.00) node[right]{};
\draw[->,rgb color={255,42,0}] (-3.00,-1.00) -- (7.00,-4.33);
\draw[->,rgb color={255,42,0}] (-3.00,-1.00) -- (7.00,-2.67);
\draw[->,rgb color={255,42,0}] (-3.00,-1.00) -- (7.00,0.111);
\draw[->,rgb color={255,42,0}] (-3.00,-1.00) -- (7.00,0.111);
\draw[->,rgb color={255,42,0}] (-3.00,-1.00) -- (7.00,0.111);
\draw[->,rgb color={255,42,0}] (-3.00,-1.00) -- (7.00,-0.167);
\draw[->,rgb color={255,42,0}] (-3.00,-1.00) -- (7.00,-0.333);
\draw[->,rgb color={255,42,0}] (-3.00,-1.00) -- (7.00,0.333);
\draw[->,rgb color={255,42,0}] (-3.00,2.00) -- (-10.5,7.00);
\draw[->,rgb color={255,42,0}] (-3.00,2.00) -- (-3.00,7.00);
\draw[->,rgb color={255,42,0}] (-3.00,2.00) -- (-3.00,7.00);
\draw[->,rgb color={255,42,0}] (-3.00,2.00) -- (-3.00,7.00);
\draw[->,rgb color={255,42,0}] (-3.00,2.00) -- (7.00,2.00) node[right]{};
\draw[->,rgb color={255,42,0}] (-3.00,2.00) -- (7.00,2.00) node[right]{};
\draw[->,rgb color={255,42,0}] (-3.00,2.00) -- (7.00,2.00) node[right]{};
\draw[->,rgb color={255,42,0}] (-3.00,2.00) -- (7.00,2.00) node[right]{};
\draw[->,rgb color={255,42,0}] (-3.00,2.00) -- (7.00,2.00) node[right]{};
\draw[->,rgb color={255,42,0}] (-3.00,2.00) -- (7.00,2.00) node[right]{};
\draw[->,rgb color={255,42,0}] (0.000,2.00) -- (7.00,2.00) node[right]{};
\draw[->,rgb color={255,42,0}] (0.000,2.00) -- (7.00,2.00) node[right]{};
\draw[->,rgb color={255,42,0}] (-3.00,2.00) -- (7.00,2.00) node[right]{};
\draw[->,rgb color={255,42,0}] (-3.00,2.00) -- (7.00,5.33);
\draw[->,rgb color={255,42,0}] (-3.00,2.00) -- (7.00,3.67);
\draw[->,rgb color={255,42,0}] (-3.00,2.00) -- (7.00,0.889);
\draw[->,rgb color={255,42,0}] (-3.00,2.00) -- (7.00,0.889);
\draw[->,rgb color={255,42,0}] (-3.00,2.00) -- (7.00,0.889);
\draw[->,rgb color={255,42,0}] (-3.00,2.00) -- (7.00,1.17);
\draw[->,rgb color={255,42,0}] (-3.00,2.00) -- (7.00,0.667);
\draw[->,rgb color={255,42,0}] (-3.00,2.00) -- (7.00,1.33);
\draw[->,rgb color={255,42,0}] (-3.00,2.33) -- (7.00,2.33) node[right]{};
\draw[->,rgb color={255,42,0}] (-3.00,2.33) -- (7.00,2.33) node[right]{};
\draw[->,rgb color={255,77,0}] (0.000,-1.25) -- (7.00,-1.25) node[right]{};
\draw[->,rgb color={255,77,0}] (-2.25,-1.25) -- (7.00,-1.25) node[right]{};
\draw[->,rgb color={255,42,0}] (0.000,2.25) -- (7.00,2.25) node[right]{};
\draw[->,rgb color={255,42,0}] (-2.25,2.25) -- (7.00,2.25) node[right]{};
\draw[->,rgb color={255,0,0}] (-1.50,-0.500) -- (-18.0,-6.00);
\draw[->,rgb color={255,0,0}] (-1.50,-0.500) -- (7.00,2.33);
\draw[->,rgb color={255,0,0}] (-1.50,1.50) -- (-18.0,7.00);
\draw[->,rgb color={255,0,0}] (-1.50,1.50) -- (7.00,-1.33);
\draw[->,rgb color={255,42,0}] (0.000,-1.00) -- (7.00,-3.33);
\draw[->,rgb color={255,42,0}] (0.000,-1.00) -- (7.00,-0.417);
\draw[->,rgb color={255,42,0}] (2.25,-0.750) -- (7.00,-0.750) node[right]{};
\draw[->,rgb color={255,42,0}] (0.000,-0.750) -- (7.00,-0.750) node[right]{};
\draw[->,rgb color={255,42,0}] (0.000,-0.667) -- (7.00,-0.667) node[right]{};
\draw[->,rgb color={255,42,0}] (0.000,-0.500) -- (7.00,-0.500) node[right]{};
\draw[->,rgb color={255,42,0}] (0.000,-0.500) -- (7.00,-0.500) node[right]{};
\draw[->,rgb color={255,42,0}] (1.50,-0.500) -- (7.00,-0.500) node[right]{};
\draw[->,rgb color={255,42,0}] (4.50,-0.500) -- (7.00,-0.500) node[right]{};
\draw[->,rgb color={255,42,0}] (0.000,-0.500) -- (7.00,-0.500) node[right]{};
\draw[->,rgb color={255,42,0}] (0.000,-0.500) -- (7.00,-1.67);
\draw[->,rgb color={255,42,0}] (0.000,-0.500) -- (7.00,0.0833);
\draw[->,rgb color={255,0,0}] (0.000,0.000) -- (7.00,0.000) node[right]{};
\draw[->,rgb color={255,0,0}] (0.000,0.000) -- (7.00,0.000) node[right]{};
\draw[->,rgb color={255,0,0}] (0.000,0.000) -- (7.00,0.000) node[right]{};
\draw[->,rgb color={255,0,0}] (0.000,0.000) -- (7.00,0.000) node[right]{};
\draw[->,rgb color={255,0,0}] (0.000,0.000) -- (7.00,0.000) node[right]{};
\draw[->,rgb color={255,0,0}] (0.000,0.000) -- (7.00,0.000) node[right]{};
\draw[->,rgb color={255,0,0}] (3.00,0.000) -- (7.00,0.000) node[right]{};
\draw[->,rgb color={255,0,0}] (3.00,0.000) -- (7.00,0.000) node[right]{};
\draw[->,rgb color={255,0,0}] (0.000,0.000) -- (7.00,0.000) node[right]{};
\draw[->,rgb color={255,0,0}] (0.000,0.000) -- (7.00,-4.67);
\draw[->,rgb color={255,0,0}] (0.000,0.000) -- (7.00,-2.33);
\draw[->,rgb color={255,0,0}] (0.000,0.000) -- (7.00,-2.33);
\draw[->,rgb color={255,0,0}] (0.000,0.000) -- (7.00,-2.33);
\draw[->,rgb color={255,0,0}] (0.000,0.000) -- (7.00,-1.17);
\draw[->,rgb color={255,0,0}] (0.000,0.000) -- (7.00,1.17);
\draw[->,rgb color={255,0,0}] (0.000,0.000) -- (7.00,1.17);
\draw[->,rgb color={255,0,0}] (0.000,0.000) -- (7.00,1.17);
\draw[->,rgb color={255,0,0}] (0.000,0.000) -- (7.00,-0.778);
\draw[->,rgb color={255,0,0}] (0.000,0.000) -- (7.00,0.778);
\draw[->,rgb color={255,0,0}] (0.000,0.000) -- (7.00,1.56);
\draw[->,rgb color={255,0,0}] (0.000,0.000) -- (7.00,0.583);
\draw[->,rgb color={255,0,0}] (0.000,0.500) -- (0.000,-6.00);
\draw[->,rgb color={255,0,0}] (0.000,0.500) -- (0.000,7.00);
\draw[->,rgb color={255,0,0}] (0.000,1.00) -- (7.00,1.00) node[right]{};
\draw[->,rgb color={255,0,0}] (0.000,1.00) -- (7.00,1.00) node[right]{};
\draw[->,rgb color={255,0,0}] (0.000,1.00) -- (7.00,1.00) node[right]{};
\draw[->,rgb color={255,0,0}] (0.000,1.00) -- (7.00,1.00) node[right]{};
\draw[->,rgb color={255,0,0}] (0.000,1.00) -- (7.00,1.00) node[right]{};
\draw[->,rgb color={255,0,0}] (0.000,1.00) -- (7.00,1.00) node[right]{};
\draw[->,rgb color={255,0,0}] (3.00,1.00) -- (7.00,1.00) node[right]{};
\draw[->,rgb color={255,0,0}] (3.00,1.00) -- (7.00,1.00) node[right]{};
\draw[->,rgb color={255,0,0}] (0.000,1.00) -- (7.00,1.00) node[right]{};
\draw[->,rgb color={255,0,0}] (0.000,1.00) -- (7.00,3.33);
\draw[->,rgb color={255,0,0}] (0.000,1.00) -- (7.00,3.33);
\draw[->,rgb color={255,0,0}] (0.000,1.00) -- (7.00,3.33);
\draw[->,rgb color={255,0,0}] (0.000,1.00) -- (7.00,5.67);
\draw[->,rgb color={255,0,0}] (0.000,1.00) -- (7.00,-0.167);
\draw[->,rgb color={255,0,0}] (0.000,1.00) -- (7.00,-0.167);
\draw[->,rgb color={255,0,0}] (0.000,1.00) -- (7.00,-0.167);
\draw[->,rgb color={255,0,0}] (0.000,1.00) -- (7.00,2.17);
\draw[->,rgb color={255,0,0}] (0.000,1.00) -- (7.00,-0.556);
\draw[->,rgb color={255,0,0}] (0.000,1.00) -- (7.00,0.222);
\draw[->,rgb color={255,0,0}] (0.000,1.00) -- (7.00,1.78);
\draw[->,rgb color={255,0,0}] (0.000,1.00) -- (7.00,0.417);
\draw[->,rgb color={255,0,0}] (0.000,1.50) -- (7.00,1.50) node[right]{};
\draw[->,rgb color={255,0,0}] (0.000,1.50) -- (7.00,1.50) node[right]{};
\draw[->,rgb color={255,0,0}] (1.50,1.50) -- (7.00,1.50) node[right]{};
\draw[->,rgb color={255,0,0}] (4.50,1.50) -- (7.00,1.50) node[right]{};
\draw[->,rgb color={255,0,0}] (0.000,1.50) -- (7.00,1.50) node[right]{};
\draw[->,rgb color={255,0,0}] (0.000,1.50) -- (7.00,2.67);
\draw[->,rgb color={255,0,0}] (0.000,1.50) -- (7.00,0.917);
\draw[->,rgb color={255,0,0}] (0.000,1.67) -- (7.00,1.67) node[right]{};
\draw[->,rgb color={255,0,0}] (2.25,1.75) -- (7.00,1.75) node[right]{};
\draw[->,rgb color={255,0,0}] (0.000,1.75) -- (7.00,1.75) node[right]{};
\draw[->,rgb color={255,42,0}] (0.000,2.00) -- (7.00,4.33);
\draw[->,rgb color={255,42,0}] (0.000,2.00) -- (7.00,1.42);
\draw[->,rgb color={255,42,0}] (1.00,-0.333) -- (7.00,-0.333) node[right]{};
\draw[->,rgb color={255,0,0}] (1.00,1.33) -- (7.00,1.33) node[right]{};
\draw[->,rgb color={255,0,0}] (1.00,1.33) -- (7.00,1.33) node[right]{};
\draw[->,rgb color={255,42,0}] (3.75,-0.250) -- (7.00,-0.250) node[right]{};
\draw[->,rgb color={255,42,0}] (1.50,-0.250) -- (7.00,-0.250) node[right]{};
\draw[->,rgb color={255,0,0}] (1.50,0.500) -- (7.00,0.500) node[right]{};
\draw[->,rgb color={255,0,0}] (1.50,0.500) -- (7.00,0.500) node[right]{};
\draw[->,rgb color={255,0,0}] (3.00,0.500) -- (7.00,0.500) node[right]{};
\draw[->,rgb color={255,0,0}] (6.00,0.500) -- (7.00,0.500) node[right]{};
\draw[->,rgb color={255,0,0}] (1.50,0.500) -- (7.00,0.500) node[right]{};
\draw[->,rgb color={255,0,0}] (1.50,0.500) -- (7.00,-0.111);
\draw[->,rgb color={255,0,0}] (1.50,0.500) -- (7.00,1.11);
\draw[->,rgb color={255,0,0}] (3.75,1.25) -- (7.00,1.25) node[right]{};
\draw[->,rgb color={255,0,0}] (1.50,1.25) -- (7.00,1.25) node[right]{};
\draw[->,rgb color={255,0,0}] (2.00,0.333) -- (7.00,0.333) node[right]{};
\draw[->,rgb color={255,0,0}] (2.00,0.333) -- (7.00,0.333) node[right]{};
\draw[->,rgb color={255,0,0}] (2.00,0.667) -- (7.00,0.667) node[right]{};
\draw[->,rgb color={255,0,0}] (4.50,0.250) -- (7.00,0.250) node[right]{};
\draw[->,rgb color={255,0,0}] (2.25,0.250) -- (7.00,0.250) node[right]{};
\draw[->,rgb color={255,0,0}] (4.50,0.750) -- (7.00,0.750) node[right]{};
\draw[->,rgb color={255,0,0}] (2.25,0.750) -- (7.00,0.750) node[right]{};
\draw[->,rgb color={255,0,0}] (3.00,0.000) -- (7.00,-0.667);
\draw[->,rgb color={255,0,0}] (3.00,0.000) -- (7.00,0.444);
\draw[->,rgb color={255,0,0}] (3.00,1.00) -- (7.00,1.67);
\draw[->,rgb color={255,0,0}] (3.00,1.00) -- (7.00,0.556);
\end{tikzpicture}
}
}
\caption{First steps of the scattering diagram $S(\fD^\iin_{\cl^+})$. Figure due to Tim Gr\"afnitz \cite{gabele2019tropical}.}
\end{figure}

\newpage

\section{Introduction}

\subsection{Genus-$0$ Gromov--Witten theory of $(\PP^2,E)$}
\label{section_relative_GW}

Let $E$ be a smooth cubic curve in the complex projective plane 
$\PP^2$. For every $d \in \Z_{>0}$, we consider 
degree-$d$ rational curves in $\PP^2$ intersecting 
$E$ at a single point. The space of degree-$d$ rational curves in $\PP^2$ is of dimension $3d-1$. 
On the other hand, a generic degree-$d$ curve intersects $E$ in $3d$ points, and so a single intersection point with $E$ should define a constraint of codimension $3d-1$. 
Therefore, we expect that counting degree-$d$ rational curves in $\PP^2$ intersecting $E$ at a single point should be a well-posed enumerative problem. 

We formulate a precise version of this enumerative problem using Gromov--Witten theory. Let 
$\overline{M}_{0}(\PP^2/E,d)$ be the moduli space of 
genus-$0$ degree-$d$ stable maps to $\PP^2$ relative to $E$,
with maximal contact order with $E$ at a single point.
According to general relative Gromov--Witten theory
\cite{MR1882667, MR1938113}, the moduli space $\overline{M}_{0}(\PP^2/E,d)$ is a proper Deligne--Mumford stack, admitting a zero-dimensional
virtual fundamental class $[\overline{M}_{0}(\PP^2/E,d)]^{\virt}$.
The corresponding genus-$0$ maximal contact Gromov--Witten invariant of $(\PP^2,E)$ is defined by
\[N_{0,d}^{\PP^2/E} \coloneq \int_{[\overline{M}_{0}(\PP^2/E,d)]^{\virt}} 1 \in \Q \,.\]

The main advantage of the Gromov--Witten definition over a more naive enumerative definition is that $N_{0,d}^{\PP^2/E}$ is independent of the specific choice of $E$. This follows from the general deformation invariance property of Gromov--Witten theory. The main drawback of the Gromov--Witten definition is that the moduli space $\overline{M}_{0}(\PP^2/E,d)$ is in general of positive dimension and contains stable maps which are very far from being immersions. 
In particular, the invariants 
$N_{0,d}^{\PP^2/E}$ are in general non-integers and their direct geometric meaning is unclear. Nevertheless, a general theme in Gromov-Witten theory is that it is often possible to reorganize Gromov-Witten invariants to form so-called BPS counts which are integers, sometimes more geometrically meaningful, and often have better properties.

Let $K_{\PP^2}$ be the noncompact Calabi-Yau 3-fold defined as 
the total space of the canonical line bundle $\cO(-3)$ of $\PP^2$, also known as local $\PP^2$.
For every $d \in \Z_{>0}$, let $N_{0,d}^{K_{\PP^2}}$ be the genus-$0$ degree-$d$ Gromov-Witten invariant of 
$K_{\PP^2}$. According to the local-relative correspondence of van Garrel-Graber-Ruddat \cite{MR3948687}, we have, for every $d \in \Z_{>0}$, the relation
\[ N_{0,d}^{\PP^2/E}=(-1)^{d-1}3d N_{0,d}^{K_{\PP^2}}\]
between maximal contact Gromov-Witten invariants of $(\PP^2,E)$ and Gromov-Witten invariants of 
$K_{\PP^2}$.
The invariants $N_{0,d}^{K_{\PP^2}}$
can be computed by torus localization and the answer can be expressed in the framework of local mirror symmetry \cite{MR1797015}. 
Therefore, the relative Gromov-Witten invariants $N_{0,d}^{\PP^2/E}$ can be effectively computed.

We define relative BPS counts $n_{0,d}^{\PP^2/E}$ by the formula
\begin{equation} 
(-1)^{d-1} N_{0,d}^{\PP^2/E} 
=\sum_{\ell|d} \frac{1}{\ell^2} n_{0,\frac{d}{\ell}}^{\PP^2/E} \,,
\end{equation}
where the sum is over the positive divisors of $d$. The BPS counts underlying the 
invariants $N_{0,d}^{K_{\PP^2}}$ are the genus-$0$ Gopakumar-Vafa invariants $n_{0,d}^{K_{\PP^2}}$
defined by the formula
\begin{equation}
N_{0,d}^{K_{\PP^2}}=\sum_{\ell|d} \frac{1}{\ell^3}
n_{0,\frac{d}{\ell}}^{K_{\PP^2}}\,,
\end{equation}
where the sum is again over the positive divisors of $d$.
The local-relative correspondence can be rewritten at the BPS level as 
$n_{0,d}^{\PP^2/E}=3d n_{0,d}^{K_{\PP^2}}$. It is known
(see Proposition \ref{prop_katz_conj}) that the BPS invariants $n_{0,d}^{K_{\PP^2}}$ are integers of sign $(-1)^{d-1}$. 
It follows that the relative BPS invariants  $n_{0,d}^{\PP^2/E}$ are also integers of sign $(-1)^{d-1}$. Therefore, the BPS structure underlying the invariants $N_{0,d}^{\PP^2/E}$ can be considered as understood.

\subsection{Contributions of the various contact points}
\label{section_takahashi_intro}

\subsubsection{Statements}

The present paper concerns the BPS structure underlying a
more precise version of the invariants $N_{0,d}^{\PP^2/E}$ taking into account the position of the contact point with $E$.

We fix $p_0$
one of the $9$ flex points of $E$ and we denote by $L_{p_0}$ the line tangent to $E$ at $p_0$. If $C$ is a degree $d$ 
rational curve in 
$\PP^2$ intersecting $E$ at a single point $p$, then the cycle
$C-dL_{p_0}$ intersects $E$ in the cycle $3d p-3dp_0=3d(p-p_0)$.
As the cycle $C-dL_{p_0}$ has degree $0$, it is linearly equivalent to $0$, and so the cycle $3d (p-p_0)$ is linearly equivalent to $0$ in $E$.

Thus\footnote{One should note that in general an element of $\overline{M}_0(\PP^2/E,d)$ is a map to an expanded degeneration of $\PP^2$ along $E$, and so an additional argument is required. In genus-$0$, every component mapping inside the expansion of $\PP^2$ along $E$ projects to a point in $E$, and so the result follows indeed from the previous paragraph. A more general argument working in arbitrary genus is given in Lemma 
\ref{lem_contact_torsion}}, if $p$ is in the image of the evaluation map $\ev \colon \overline{M}_0(\PP^2/E,d) \rightarrow E$ at the contact point with $E$, then $p-p_0$ is necessarily a $(3d)$-torsion point of
the group
$\Pic^0(E)$ of degree-$0$ cycles on $E$ up to linear equivalence. Therefore, we have a decomposition 
\[\overline{M}_0(\PP^2/E,d) = \coprod_{p \in P_d} \overline{M}_0(\PP^2/E,d)^p \,,\]
where the disjoint union is over the set $P_d$ of $(3d)^2$ points $p$ of $E$ such that $p-p_0$ is a $(3d)$-torsion point in $\Pic^0(E)$. 
The set $P_d$ is independent of the choice of the flex point $p_0$. Indeed, if $p_0'$ is another flex point, then $p_0-p_0'$ is $3$-torsion and in particular $(3d)$-torsion in $\Pic^0(E)$.

For every $p \in P_d$, let
$[\overline{M}_{0}(\PP^2/E,d)^p]^{\virt}$ be
the restriction of 
$[\overline{M}_{0}(\PP^2/E,d)]^{\virt}$
to $\overline{M}_{0}(\PP^2/E,d)^p$.
We define the contribution of $p$ to the relative Gromov-Witten invariant $N_{0,d}^{\PP^2/E}$ by
\[N_{0,d}^{\PP^2/E,p} \coloneq \int_{[\overline{M}_{0}(\PP^2/E,d)^p]^{\virt}} 1 \in \Q \,.\]

By construction, the relative Gromov-Witten invariant 
$N_{0,d}^{\PP^2/E}$ is the sum of the contributions of the points $p \in P_d$:
\begin{equation}\label{eq_total_gw}
N_{0,d}^{\PP^2/E} = \sum_{p \in P_d} N_{0,d}^{\PP^2/E,p} \,.\end{equation}

The main question we wish to address is how the numbers
$N_{0,d}^{\PP^2/E,p}$ depend on the point $p$ in $P_d$.
For different points $p \in P_d$, the geometry of multiple cover contributions to $N_{0,d}^{\PP^2/E,p}$ can be quite different. 

For example, for $d=2$ and $p \in P_2$ with $p \notin P_1$, then $\overline{M}_{0}(\PP^2/E,d)^p$ is a single point and $N_{0,2}^{\PP^2/E,p}=1$. However, for $p \in P_2$ with $p \in P_1$, that is, if $p$ is a flex point, then
the contributions come from double covers of the tangent line to $E$ at $p$, so 
$\overline{M}_{0}(\PP^2/E,d)^p$ is positive-dimensional , and one finds after a computation of the virtual class that
$N_{0,2}^{\PP^2/E,p}=\frac{3}{4}$,
see \cite[Proposition 6.1]{MR2667135}.

We introduce some notation in order to make a systematic study of this phenomenon.
For $p \in \bigcup_{d \geqslant 1} P_d$, we denote by 
$d(p)$ the smallest positive integer $d$ such that 
$p \in P_d$. 
The points $p \in P_d$ with $d(p)=d$ are 
primitive in the sense that they do not belong to any $P_k$ with $k<d$ and so are the simplest
from the point of view of multiple covers in Gromov-Witten theory. 
By contrast, the points $p \in P_d$ with $d(p)=1$ are exactly the $9$ flex points of $E$, which contribute to Gromov-Witten invariants in every degree, and so are the most complicated from the point of view of multiple covers.

By a monodromy argument reviewed in Lemma
\ref{lem_monodromy}, we show that
$N_{0,d}^{\PP^2/E,p}$ only depends on $p \in P_d$ through $d(p)=k$. Hence, 
for every $d \in \Z_{>0}$ and 
$k \in \Z_{>0}$ dividing $d$, we write
$N_{0,d}^{\PP^2/E,k}$ for the common value  of the invariants $N_{0,d}^{\PP^2/E,p}$ with
$p \in P_d$ such that $d(p)=k$. 

Therefore, the question on the dependence of the point of contact is reduced to the question of the dependence on $k$ of the invariants $N_{0,d}^{\PP^2/E, k}$. 
This question is 
completely solved by the following result, expressing the general invariants 
$N_{0,d}^{\PP^2/E,k}$ in terms of the primitive ones 
$N_{0,d'}^{\PP^2/E,d'}$.

\begin{thm} \label{thm_takahashi_precise}
For every $d \in \Z_{>0}$ and $k
\in \Z_{>0}$ dividing $d$, we have 
\begin{equation}(-1)^{d-1} N_{0,d}^{\PP^2/E,k}=\sum_{\substack{d'\in \Z_{>0} \\ k|d'|d}} \frac{1}{(d/d')^2}(-1)^{d'-1}
N_{0,d'}^{\PP^2/E,d'} \,.\end{equation}
\end{thm}

According to point (1) of 
\cite[Proposition 4.21]{choi2018log}, the primitive invariants $N_{0,d}^{\PP^2/E,d}$
are positive integers. Indeed, the corresponding moduli space is the union of finitely many possibly nonreduced points, 
each contributing its length to the Gromov-Witten count. Thus, Theorem \ref{thm_takahashi_precise} also expresses the BPS structure underlying the Gromov-Witten invariants $N_{0,d}^{\PP^2/E,k}$.

In order to get rid of signs, it is useful to define, for every 
$d \in \Z_{>0}$ and $k \in \Z_{>0}$ dividing $d$:
\[\overline{\Omega}_{d, k}^{\PP^2/E} \coloneq (-1)^{d-1}
N_{0,d}^{\PP^2/E, k} \in \Q\,.\]
There exists a unique collection of $\Omega^{\PP^2/E}_{d,k} \in \Q$, indexed by $d \in \Z_{>0}$ and $k \in \Z_{>0}$ dividing $d$, such that
\begin{equation}\label{equation_bps}
\overline{\Omega}_{d, k}^{\PP^2/E}=\sum_{\substack{d' \in \Z_{\geqslant 1}\\ k|d'|d}} \frac{1}{(d/d')^2} \Omega^{\PP^2/E}_{d', k}\,.\end{equation}
Indeed, this relation can be inverted by the M\"obius inversion formula. We call $\Omega_{d,k}^{\PP^2/E}$
the degree-$d$ relative BPS invariant of 
$(\PP^2,E)$ attached to the point $p$
with $d(p)=k$.
The following result
rephrases Theorem \ref{thm_takahashi_precise}
in terms of the relative BPS invariants 
$\Omega_{d,k}^{\PP^2/E}$.

\begin{thm}\label{thm_takahashi_bps}
For every $d \in \Z_{>0}$, the relative 
BPS invariant $\Omega_{d,k}^{\PP^2/E}$ is independent of $k$, 
that is, we have
\[ \Omega_{d,k}^{\PP^2/E}=\Omega_{d,k'}^{\PP^2/E}\]
for every $k, k' \in \Z_{>0}$ dividing $d$.
\end{thm}

By Theorem \ref{thm_takahashi_bps}, it makes sense to write
$\Omega_d^{\PP^2/E}$ for the common value of the invariants
$\Omega_{d,k}^{\PP^2/E}$. As we have $\Omega_d^{\PP^2/E}=(-1)^{d-1} N_{0,d}^{\PP^2/E,d}$ and we know that $N_{0,d}^{\PP^2/E,d}$ is a positive integer, it follows that $\Omega_d^{\PP^2/E}$ is an integer of sign $(-1)^{d-1}$. As the $(3d)^2$ points $p \in P_d$ all have the same BPS contribution
$\Omega_d^{\PP^2/E}$, it follows from (\ref{eq_total_gw}) that  
\begin{equation} n_{0,d}^{\PP^2/E}=(3d)^2 \Omega_d^{\PP^2/E}\,.\end{equation}

Table \ref{table_bps} gives the sequence of the relative BPS numbers
$\Omega_d^{\PP^2/E}$ for $d \leqslant 10$.

\begin{table}[H]
\centering
\begin{tabular}{l|r}
$d$ & $\Omega_d^{\PP^2/E}$ \\\hline
1 & 1 \\
2 & -1 \\
3 & 3 \\
4 & -16 \\
5 & 113 \\
6 & -948 \\
7 & 8974 \\
8 & -92840 \\
9 & 1027737 \\
10 & -12000405
\end{tabular}
\caption{\label{tab:widgets}
Invariants $\Omega_d^{\PP^2/E}$ for $d \leqslant 10$}
\label{table_bps}
\end{table}

\subsubsection{Example} The simplest non-trivial case of Theorem \ref{thm_takahashi_precise}
is obtained for
$d=2$ and $k=1$. In such case, Theorem 
\ref{thm_takahashi_precise} states that 
\[- N_{0,2}^{\PP^2/E,1}
=-N_{0,2}^{\PP^2/E,2}+\frac{1}{4}N_{0,1}^{\PP^2/E,1}\,.\]
If one knows that 
$N_{0,1}^{\PP^2/E,1}=1$
and $N_{0,2}^{\PP^2/E,2}
=1$, we obtain that one should have 
\[ N_{0,2}^{\PP^2/E,1}=1-\frac{1}{4}=\frac{3}{4}\,\,,\]
which is indeed correct by 
\cite[Proposition 6.1]{MR2667135}.

\subsubsection{Comments on Theorem \ref{thm_takahashi_precise}}

Theorem \ref{thm_takahashi_precise} is an addition to the relatively short list of questions in Gromov-Witten theory which can be fully solved despite the presence of
contracted components and multiple covers.

The study of the numbers
$N_{0,d}^{\PP^2/E,k}$ was initiated by 
N.\ Takahashi  \cite{takahashi9605007curves, MR1844627} and the content of Theorem \ref{thm_takahashi_precise} was known since then as the N.\ Takahashi's conjecture. 
In the general form of \cite[Conjecture 1.3]{choi2018log}, the pair $(\PP^2,E)$ is replaced by a pair $(Y,D)$, 
where $Y$ is a del Pezzo surface and $D$ a smooth anticanonical divisor on $Y$. One might expect to generalize our proof of 
Theorem \ref{thm_takahashi_precise} to treat that case. 
On the Gromov--Witten side, the work \cite{gabele2019tropical} already considers the general case of del Pezzo surfaces. What is missing is to generalize on the sheaf side the work \cite{bousseau2019scattering} from $(\PP^2,E)$ to del Pezzo surfaces.
An obstruction is that the space of Bridgeland stability conditions for local del Pezzo surfaces has not been studied as much as for local $\PP^2$. For example, the work of Bayer-Macri \cite{MR2852118} for local $\PP^2$ has not been generalized yet to the case of local del Pezzo surfaces.
We leave the question open for the present paper.

Theorems \ref{thm_takahashi_precise}-\ref{thm_takahashi_bps} can be viewed as an analogue for the log K3 surface $(\PP^2,E)$ of
a previously known result for K3 surfaces.
This analogy was already clear in \cite{takahashi9605007curves}.
Let $S$ be a projective K3 surface and let $\beta$ be an effective curve class on $S$. Let 
$N_{0,\beta}^S \in \Q$ be the reduced Gromov--Witten count of rational curves in $S$ of class $\beta$, invariant under deformation of $S$ keeping $\beta$ effective \cite{MR2746343}. 
Using the monodromy in the moduli space of K3 surfaces, one can show that $N_{0,\beta}$ only depends on 
$\beta^2$ and on the divisibility of 
$\beta$ in the lattice 
$H_2(S,\Z)$.
The divisibility of $\beta$ for K3 surfaces is analogous to the choice of the point $p \in P_d$ for $(\PP^2,E)$.
If the divisibility of $\beta$ is $1$, that is, 
if $\beta$ is primitive, then 
$N_{0,\beta}^S$ is a positive integer, counting rational curves with integer multiplicities, 
whereas if $\beta$ is non-primitive, $N_{0,\beta}^S$ is only a rational number, receiving complicated contributions from multiple covers.
One defines BPS invariants $n_{0,\beta}^S$
by the formula 
\[ N_{0,\beta}^S=\sum_{\beta=\ell \beta'}
\frac{1}{\ell^3}n_{0,\beta'}^S\,.\]

The following result, imprimitive case of the Yau-Zaslow conjecture, 
is due to
Klemm-Maulik-Pandharipande-Scheidegger
\cite{MR2669707} and  is the analogue of Theorems \ref{thm_takahashi_precise}-\ref{thm_takahashi_bps} for K3 surfaces.

\begin{thm}[\cite{MR2669707}]
For every effective curve class $\beta$ on $S$, the BPS invariant $n_{0,\beta}$ is independent of the divisibility of $\beta$, that is, depends on $\beta$
only through $\beta^2$.
\end{thm}

In the Noether-Lefschetz approach used in \cite{MR2669707}, 
families of K3 surfaces with nodal fibers can be replaced by their small resolutions. For a family of pairs $(\PP^2,E)$,
it is not possible to similarly get rid of the fibers with $E$ nodal, in which rational curves that we are trying to count can fall. 
This seems to make the study of $(\PP^2,E)$
more difficult than the study of K3 surfaces. Our proof of Theorems \ref{thm_takahashi_precise}-\ref{thm_takahashi_bps} for
$(\PP^2,E)$ uses an approach different from \cite{MR2669707}.

\subsubsection{Remark on the definition of relative BPS invariants}
Our definition of relative BPS invariants
by formula (\ref{equation_bps})
matches the definition of log BPS invariants used in 
\cite{bousseau2018quantum_tropical} and 
(up to a sign 
$(-1)^{d-1}$)
in \cite{choi2018log}.
However, given the multiple cover formula of \cite[Proposition 6.1]{MR2667135}, 
one might be tempted, as done in \cite{MR2667135, MR3228298, gabele2019tropical}, to call relative 
BPS invariants the numbers 
$m_{d,k}^{\PP^2/E}$ defined by 
\[ N_{0,d}^{\PP^2/E,k}
=\sum_{\substack{d'\in \Z_{>0}\\
k|d'|d}}
\frac{1}{(d/d')^2} \binom{(d/d')(3d'-1)-1}{(d/d')-1} m_{d',k}^{\PP^2/E}\,.\]
It is proved in \cite{MR3228298} that 
the integrality of the invariants 
$\Omega_{d,k}^{\PP^2/E}$ is equivalent to the integrality of the invariants 
$m_{d,k}^{\PP^2/E}$, and that the matrix comparing the two set of invariants has a natural interpretation in terms of Donaldson-Thomas invariants of loop quivers. 

Theorem
\ref{thm_takahashi_bps} takes such simple form only when phrased in terms of the invariants 
$\Omega_{d,k}^{\PP^2/E}$. In general, the invariants 
$m_{d,k}^{\PP^2/E}$ depend on $k$.
For example, we have $m_{2,1}^{\PP^2/E}=0$ and $m_{2,2}^{\PP^2/E}=1$, see \cite[Table 7.1]{gabele2019tropical} for more examples. 
One can of course express Theorem \ref{thm_takahashi_bps} in terms of 
the invariants $m_{d,k}^{\PP^2/E}$, but the resulting formula is more complicated than the simple statement that the invariants $\Omega_{d,k}^{\PP^2/E}$ do not depend on $k$.

\subsection{Structure of the proof of Theorems \ref{thm_takahashi_precise}-\ref{thm_takahashi_bps}}
\label{section_structure_proof_intro}

Our proof of Theorems \ref{thm_takahashi_precise}-\ref{thm_takahashi_bps}
relies on a connection with moduli spaces of one-dimensional Gieseker semistable sheaves on
$\PP^2$. 

\subsubsection{Coherent sheaves}
We refer to \cite{huybrechts2010geometry} as general reference on Gieseker semistable sheaves. 
For every $d \in \Z_{>0}$ and $\chi \in \Z$, we denote by $M_{d,\chi}$ the moduli space of S-equivalence classes of Gieseker semistable sheaves on $\PP^2$, supported on curves of degree $d$ and of Euler characteristic $\chi$.
It is proved in \cite{MR1263210} that, for every $d \in \Z_{>0}$ and $\chi \in \Z$, $M_{d,\chi}$ is a
nonempty integral normal projective variety of dimension $d^2+1$. If $d$ and $\chi$ are coprime, then $M_{d,\chi}$ is smooth.
However, $M_{d,\chi}$ is generally singular if $d$ and $\chi$ are not coprime.

Nevertheless, the intersection cohomology groups $IH^j(M_{d,\chi}, \Q)$ behave as well as cohomology of a smooth projective variety
\cite{MR572580,MR696691,MR751966}.
We denote by $Ie(M_{d,\chi})$ the corresponding intersection topological Euler characteristic.
According to \cite[Corollary 6.1.3]{bousseau2019scattering}, the intersection Euler characteristic is positive:
$Ie(M_{d,\chi}) \in \Z_{>0}$.
For every $d \in \Z_{>0}$ and $\chi \in \Z$, we define 
\[ \Omega_{d,\chi}^{\PP^2}
\coloneq 
(-1)^{\dim M_{d,\chi}}
Ie(M_{d,\chi})\in \Z\,.\]

We will prove that the invariants $\Omega_{d,\chi}^{\PP^2}$ are the  Donaldson-Thomas invariants for
one-dimensional sheaves on the Calabi-Yau 3-fold $K_{\PP^2}$  \cite{MR2951762}. 
The conjecture
\cite[Conjecture 6.20]{MR2951762}
on the $\chi$-independence of Donaldson-Thomas invariants of one-dimensional sheaves was proved for $K_{\PP^2}$, 
by combination of \cite{MR2264664, MR2892766, MR2215440, MR2250076} (see also \cite[Appendix A]{MR3861701}). 
Therefore, we have the following result.

\begin{thm}[=Theorem \ref{thm_joyce_conj}] \label{thm_joyce_conj_intro}
The intersection Euler characteristics 
$Ie(M_{d,\chi})$ is independent of $\chi$, that is, for every $d \in \Z_{>0}$
and $\chi, \chi' \in \Z$, we have 
\[ \Omega_{d,\chi}^{\PP^2}=\Omega_{d,\chi'}^{\PP^2}\,.\]
\end{thm}

\subsubsection{Gromov-Witten/sheaves correspondence}
Theorem \ref{thm_joyce_conj_intro} is formally similar to the rephrasing of Theorem \ref{thm_takahashi_precise} given by 
Theorem \ref{thm_takahashi_bps}. It is not a coincidence: Theorem \ref{thm_takahashi_bps} is a corollary of Theorem 
\ref{thm_joyce_conj_intro} and of
the following Theorem \ref{thm_local_relative_intro}.
We introduce some notation in order to state Theorem \ref{thm_log_bps_local_bps_intro}.
For every $d \in \Z_{>0}$ and $\chi \in \Z$, we define
\[ \ell_{d,\chi} \coloneq \frac{d}{\gcd(d,\chi)} \in  \Z_{>0}\,.\]
For every $G$ an abelian group and $x$ an element of $G$ of finite order divisible by $3$, we denote by $d(x)$ the smallest positive integer such that 
$(3d(x))x=0$ in $G$.
For every $\ell \in \Z_{>0}$, we denote by $r_\ell$ the number of elements $x \in \Z/(3\ell)$ such that $d(x)=\ell$.
For every $k, \ell \in \Z_{>0}$, we denote by
$s_{k,\ell}$ the number of $x=(a,b)
\in \Z/(3k) \times \Z/(3k)$
such that 
$d(x)=k$ and $d(a)=\ell$. The first values of $r_{\ell}$ and $s_{k,\ell}$ for small values of $k$ and $\ell$ are 
listed in Table \ref{table_r} and Table \ref{table_skl} respectively.

\begin{center}
\begin{table}[H]
\begin{tabular}{ c|c }
$\ell$ & $r_\ell$ \\
 \hline
 1 & 3  \\ 
 2 & 3   \\ 
 3 & 6  \\ 
\end{tabular}
\vspace{0.5cm}
\caption{First values of $r_\ell$ \label{table_r}}
\end{table}
\end{center}

\begin{center}
\begin{table}[H]
\begin{tabular}{ |c c c| }
\hline
$s_{1,1}=9$ & & \\
 $s_{2,1}=9$ & $s_{2,2}=18$&  \\ 
 $s_{3,1}=18$ &   & $s_{3,3}=54$ \\
 \hline
\end{tabular}
\vspace{0.5cm}
\caption{First values of $s_{k,\ell}$ \label{table_skl}}
\end{table}
\end{center}

\begin{thm} \label{thm_log_bps_local_bps_intro}
For every $d\in \Z_{>0}$ and $\chi \in \Z$, we have 
\[ 
\Omega_{d,\chi}^{\PP^2}=\sum_{\ell_{d,\chi}|k|d}
\frac{s_{k,\ell_{d,\chi}}}{r_{\ell_{d,\chi}}} \Omega_{d,k}^{\PP^2/E}\,.\]
\end{thm}

Let $\Omega_d^{\PP^2}$ be the common value of the 
$\Omega_{d,\chi}^{\PP^2}$, which makes sense by Theorem 
\ref{thm_joyce_conj_intro}. We will show that Theorem \ref{thm_log_bps_local_bps_intro}
can be rewritten as the following Theorem \ref{thm_local_relative_intro}, which directly implies 
 Theorem 
\ref{thm_takahashi_bps}.

\begin{thm} \label{thm_local_relative_intro}
For every $d \in \Z_{>0}$ and $k \in \Z_{>0}$
dividing $d$, we have 
\[ \Omega_{d,k}^{\PP^2/E}=\frac{1}{3d} \Omega_d^{\PP^2}\,.\]
\end{thm}

Theorem \ref{thm_log_bps_local_bps_intro} gives a connection between two distinct worlds: relative Gromov-Witten theory of $(\PP^2,E)$ and moduli spaces of semistable one-dimensional sheaves on $\PP^2$. 

We will prove
Theorem \ref{thm_log_bps_local_bps_intro} by combination of the main result of \cite{gabele2019tropical} with the main result of \cite{bousseau2019scattering}.
In \cite{gabele2019tropical}, Gr\"afnitz considers a normal crossing degeneration of 
$(\PP^2,E)$ and applies the formalism of log Gromov-Witten theory 
\cite{MR3011419, MR3224717, MR3257836, abramovich2017decomposition} to compute the invariants $N_{0,d}^{\PP^2/E,k}$ in terms of the special fiber, and then in terms of the tropicalization $B$ of the special fiber. 
The main result of \cite{gabele2019tropical} computes the Gromov-Witten invariants $N_{0,d}^{\PP^2/E,k}$ in terms of a wall-structure $\mathcal{S}$ on $B$. This wall-structure previously appeared in \cite[Example 2.4]{cps} in the context of mirror symmetry for $(\PP^2,E)$. The tropical correspondence result of \cite{gabele2019tropical} was expected from the point of view of mirror symmetry.

On the other hand, we proved  in \cite{bousseau2019scattering} that a scattering diagram $S(\fD^\iin_{\cl^+})$ can be used to algorithmically compute the intersection Euler characteristic of the moduli spaces of Gieseker semistable sheaves on $\PP^2$.
We will prove Theorem \ref{thm_log_bps_local_bps_intro} by
first showing that the wall-structure $\mathcal{S}$ of 
\cite{gabele2019tropical} is related to the scattering diagram 
$S(\fD^\iin_{\cl^+})$ of \cite{bousseau2019scattering}, and then applying Theorem \ref{thm_joyce_conj_intro}
on the sheaf side.

A natural question is to try to extend the sheaf/ Gromov--Witten correspondence of Theorem \ref{thm_log_bps_local_bps_intro} to other types of Gromov--Witten invariants of $(\PP^2,E)$ such as the ones recently considered in \cite{grafnitz2022proper} in relation with the broken lines of the wall-structure $\mathcal{S}$. 
We expect the sheaf side to involve an appropriate notion of framing, but we leave the question open for the present paper.

\subsubsection{Analogy with \cite{MR2667135, MR2662867, bousseau2018quantum_tropical}.}
Gross-Pandharipande-Siebert \cite{MR2667135} proved a tropical correspondence theorem
 between local scattering diagrams in $\R^2$ and log Gromov-Witten invariants 
 of log Calabi-Yau surfaces $(Y,D)$ with $D$ a cycle of rational curves.
By combination with the work of Reineke \cite{MR2801406, MR2650811} on wall-crossing for Donaldson-Thomas invariants of quivers,
they obtained a correspondence between log Gromov-Witten invariants 
of log Calabi-Yau surfaces and Donaldson-Thomas invariants of some acyclic quivers.
This correspondence has been extended and generalized in various directions
\cite{MR2662867, MR3004575, MR3033514, MR3383167}, notably including a refined Donaldson-Thomas/higher-genus Gromov-Witten version \cite{MR3904449, bousseau2018quantum_tropical}.

The present paper describes an analogous story: the pair 
$(Y,D)$ with $D$ a cycle of rational curves being replaced by the pair 
$(\PP^2,E)$ with $E$ a smooth genus-$1$ curve, 
the representations of the acyclic quivers being replaced by the objects of
$\D^b(\PP^2)$, the tropical correspondence theorem of \cite{MR2667135} 
being replaced by the tropical correspondence theorem of 
\cite{gabele2019tropical}, and the work of Reineke \cite{MR2801406, MR2650811} being replaced by 
\cite{bousseau2019scattering}. In fact, it is similar but with one more level of 
difficulty:
the scattering diagram $S(\fD^\iin_{\cl^+})$ 
consists in infinitely many local scatterings, 
and the category $\D^b(\PP^2)$ is much richer 
than the category of representations of an acyclic quiver. 
In some sense, our story is built from infinitely many local pieces, 
each one being exactly of the same nature as 
\cite{MR2667135, MR2662867}.

\subsection{Higher-genus/refined correspondence}

We will in fact prove a more general version of Theorem \ref{thm_log_bps_local_bps_intro}, involving higher-genus Gromov-Witten invariants of $(\PP^2,E)$ and Betti numbers of the moduli spaces $M_{d,\chi}$.

\subsubsection{Higher-genus Gromov-Witten theory of $(\PP^2,E)$}
\label{section_higher_genus_intro}
For every $g \in \Z_{\geq 0}$, let
$\overline{M}_{g}(\PP^2/E,d)$ be the moduli space of 
genus-$g$ degree-$d$ stable maps to $\PP^2$ relative to $E$,
with maximal contact order with $E$ at a single point.
The moduli space $\overline{M}_{g}(\PP^2/E,d)$ is a proper Deligne-Mumford stack, admitting a $g$-dimensional
virtual fundamental class
$[\overline{M}_{g}(\PP^2/E,d)]^{\virt}$.

If $\pi \colon \cC \rightarrow 
\overline{M}_{g}(\PP^2/E,d)$ is the universal curve,
of relative dualizing sheaf 
$\omega_\pi$, then the Hodge bundle
$E\coloneqq \pi_{*}\omega_\pi$
is a rank-$g$ vector bundle over 
$\overline{M}_{g}(\PP^2/E,d)$. Its Chern classes 
are classically called the lambda classes
\cite{MR717614},
$\lambda_j \coloneqq c_j(\E)$
for $0 \leqslant j \leqslant g$.
We denote by
\[N_{g,d}^{\PP^2/E} \coloneq \int_{[\overline{M}_{g}(\PP^2/E,d)]^{\virt}} (-1)^g \lambda_g \in \Q \,,\]
the corresponding relative genus-$g$ Gromov-Witten invariant of $(\PP^2,E)$.
The invariants 
$N_{g,d}^{\PP^2/E}$ are analogous to the higher-genus Gromov-Witten invariants of K3 surfaces involved in the Katz-Klemm-Vafa conjecture \cite{MR2746343}.

Let $\ev \colon \overline{M}_g(\PP^2/E,d) \rightarrow E$ be the evaluation at the contact point with $E$. We show in Lemma \ref{lem_contact_torsion}, that if $p$ is in the image of $\ev$, then $p-p_0$ is necessarily a $(3d)$-torsion point of
the group
$\Pic^0(E)$ of degree-$0$ cycles on $E$ up to linear equivalence, as in genus $0$. Thus, we can define higher-genus versions $N_{g,d}^{\PP^2/E,p}$
of the genus-$0$ invariants $N_{0,d}^{\PP^2/E,p}$, $p\in P_d$.
We show in Lemma \ref{lem_monodromy} that, as in genus $0$, the invariants $N_{g,d}^{\PP^2/E,p}$ only depend on $p$ through $d(p)$, and so we can define higher-genus versions $N_{g,d}^{\PP^2/E,k}$ of the 
genus-$0$ invariants $N_{0,d}^{\PP^2/E,k}$.

We state a conjectural higher-genus version of Theorem \ref{thm_takahashi_precise}.

\begin{conj}\label{conj_takahashi_higher_genus_intro}
For every $d \in \Z_{>0}$ and $k \in \Z_{>0}$ dividing $d$,
we have the equality
\[ (-1)^{d-1}
\sum_{g \geqslant 0}N_{g,d}^{\PP^2/E,k}
\hbar^{2g-1} 
=\sum_{\substack{d'\in \Z_{\geqslant 1} 
\\k|d'|d}} \frac{1}{(d/d')}
(-1)^{d'-1}
\sum_{g \geqslant 0} N_{g,d'}^{\PP^2/E,d'}((d/d')\hbar)^{2g-1}\]
in $\Q[\![\hbar]\!]$.
\end{conj}

We define, for every 
$d \in \Z_{>0}$ and $k \in \Z_{>0}$
dividing $d$:
\[ \overline{\Omega}^{\PP^2/E}_{d,k}
(\hbar)
\coloneq (-1)^{d-1}\left(2 \sin (\hbar/2) \right)
\left( \sum_{g\geqslant 0} N_{g,d}^{\PP^2/E,k} \hbar^{2g-1} \right)\in \Q[\![\hbar]\!]\,.\]
There exists a unique collection of higher-genus relative BPS invariants
$\Omega_{d,k}^{\PP^2/E}(\hbar) \in \Q[\![\hbar]\!]$, indexed by $d\in \Z_{>0}$ and $k\in \Z_{>0}$ dividing $d$, such that
\[ \overline{\Omega}^{\PP^2/E}_{d,k}
(\hbar)
=\sum_{\substack{d'\in \Z_{\geqslant 1}
\\ k|d'|d}}
\frac{1}{(d/d')}\frac{2\sin(\hbar/2)}{
2\sin((d/d')\hbar/2)} 
\Omega^{\PP^2/E}_{d',k}((d/d')\hbar)\,.\]
Indeed, this relation can be inverted by the M\"obius inversion formula.

Thus, we can rephrase Conjecture \ref{conj_takahashi_higher_genus_intro} as
follows.

\begin{conj} \label{conj_tak_higher_genus_bps_intro}
For every $d \in \Z_{>0}$, the higher-genus relative BPS invariant
$\Omega_{d,k}^{\PP^2/E}(\hbar)$
is independent of $k$, 
that is, for every 
$k, k'\in \Z_{>0}$ divisors of $d$,
we have 
\[ \Omega_{d,k}^{\PP^2/E}(\hbar)
=\Omega_{d,k'}^{\PP^2/E}(\hbar)\,.\]
\end{conj}

\subsubsection{Refined sheaf counting}
We denote by $Ib_{j}(M_{d,\chi})$ the  intersection Betti numbers of the moduli spaces $M_{d,\chi}$ of one-dimensional Gieseker semistable sheaves.
According to \cite[Corollary 6.1.3]{bousseau2019scattering}, the odd-degree part of intersection cohomology of $M_{d,\chi}$ vanishes: 
$Ib_{2k+1}(M_{d,\chi})=0$, and so
$Ie(M_{d,\chi}) \in \Z_{>0}$.
For every $d \in \Z_{>0}$ and $\chi \in \Z$, we define 
\[ \Omega_{d,\chi}^{\PP^2}(y^{\frac{1}{2}}) \coloneq 
(-y^{\frac{1}{2}})^{-\dim M_{d,\chi}}
\sum_{j=0}^{\dim M_{d,\chi}}
Ib_{2j}(M_{d,\chi}) y^{j} \in \Z[y^{\pm \frac{1}{2}}]\,.\]
We will prove in Proposition \ref{prop_joyce_song} that the invariants $\Omega_{d,\chi}^{\PP^2}(y^{\frac{1}{2}})$ are the
refined Donaldson--Thomas invariants for
one-dimensional sheaves on $K_{\PP^2}$ \cite{MR2951762}.
Therefore, the following Conjecture \ref{conj_refined_joyce_conj_intro}
is a refined version of \cite[Conjecture 6.20]{MR2951762}\footnote{\cite[Conjecture 6.20]{MR2951762} is stated for Donaldson--Thomas counts of one-dimensional Gieseker semistable sheaves on arbitrary 
Calabi-Yau 3-folds. One can similarly state a version of Conjecture \ref{conj_refined_joyce_conj_intro} for refined Donaldson-Thomas counts of one-dimensional Gieseker semistable sheaves on arbitrary Calabi-Yau 3-folds. The specialization of this conjecture to the case of flopping curves is discussed in \cite{davison2019refined}. }.

\begin{conj} \label{conj_refined_joyce_conj_intro}
The intersection Betti numbers $Ib_j(M_{d,\chi})$ are independent 
of $\chi$, that is, for every 
$d\in \Z_{>0}$, $\chi, \chi' \in \Z$, we have 
\[ \Omega_{d,\chi}^{\PP^2}
(y^{\frac{1}{2}})=
\Omega_{d,\chi'}^{\PP^2}(y^{\frac{1}{2}})\,.\]
\end{conj}

Conjecture \ref{conj_refined_joyce_conj_intro} was already
stated as \cite[Conjecture 0.4.3]{bousseau2019scattering} and it was proved in 
\cite[Theorem 0.4.5]{bousseau2019scattering} that Conjecture \ref{conj_refined_joyce_conj_intro} holds for $d \leq 4$.

\subsubsection{Higher-genus/refined correspondence}
The following Theorem \ref{thm_log_bps_local_bps_refined_intro} is a generalization of Theorem \ref{thm_log_bps_local_bps_intro}
and establishes a correspondence between higher-genus 
relative Gromov-Witten theory of $(\PP^2,E)$ and refined sheaf counting on local $\PP^2$. 

We stress that it is not an example of the Gromov-Witten/stable pairs correspondence of 
\cite{MR2264664}. 
The Gromov-Witten/stable pairs correspondence involves a similar change of variables $q=e^{i\hbar}$, where $\hbar$ is the genus variable on the Gromov-Witten side. 
On the stable pair side, we have unrefined invariants and $q$ is the variables keeping track of the holomorphic Euler characteristic. 
By contrast, in Theorem \ref{thm_log_bps_local_bps_refined_intro}, we have refined invariants on the sheaf side and $y$ is the refined/cohomological variable. 
Furthermore, in the Gromov-Witten/stable pairs correspondence, both sides are attached to the same geometry, whereas Theorem \ref{thm_log_bps_local_bps_refined_intro} involves a relative geometry on the Gromov-Witten side but not on the sheaf side.
In fact, one could combine Theorem \ref{thm_log_bps_local_bps_refined_intro} with the Gromov-Witten/stable pairs correspondence applied to $(\PP^2 \times \A^1, E \times \A^1)$. The result is a connection between unrefined sheaf counting on $(\PP^2 \times \A^1, E \times \A^1)$ and refined one-dimensional sheaf counting on $K_{\PP^2}$, the variable $q$ on the unrefined side becoming the variable $y$ on the refined side.

\begin{thm}\label{thm_log_bps_local_bps_refined_intro}
For every $d\in \Z_{>0}$ and $\chi \in \Z$, we have the equality
\[ 
\Omega_{d,\chi}^{\PP^2}(y^{\frac{1}{2}})=\sum_{\ell_{d,\chi}|k|d}
\frac{s_{k,\ell_{d,\chi}}}{r_{\ell_{d,\chi}}} \Omega_{d,k}^{\PP^2/E}(\hbar)\,\]
after the change of variables 
\[y=e^{i\hbar}=\sum_{n\geqslant 1}\frac{(i\hbar)^n}{n!}\,.\]
\end{thm}

The proof of Theorem \ref{thm_log_bps_local_bps_refined_intro} is a generalization of the proof of Theorem \ref{thm_log_bps_local_bps_intro}.
By \cite{bousseau2019scattering}, the scattering diagram $S(\fD^\iin_{\cl^+})$
is the classical limit of a quantum scattering diagram\footnote{In  \cite{bousseau2019scattering}, the quantum scattering diagram $S(\fD^\iin_{y^+})$ is denoted $S(\fD^\iin_{q^+})$. More generally, the refined variable is denoted by $q$ in \cite{bousseau2019scattering} and by $y$ in the present paper.}
$S(\fD^\iin_{y^+})$
computing the refined invariants 
$\Omega_{d,\chi}^{\PP^2}(y^{\frac{1}{2}})$.
Therefore, it remains to show that  $S(\fD^\iin_{y^+})$ also computes the higher-genus
Gromov-Witten invariants $N_{g,d}^{\PP^2/E,k}$, that is, we have to prove a higher-genus version of the genus-$0$ tropical correspondence of \cite{gabele2019tropical}.
This will be done following the degeneration argument of \cite{gabele2019tropical}
and using the works \cite{MR3904449, bousseau2018quantum_tropical} which establish the connection between quantum scattering diagrams and higher-genus Gromov-Witten invariants of surfaces with insertion of $\lambda_g$.

We will prove that Theorem \ref{thm_log_bps_local_bps_refined_intro} implies the following result.

\begin{thm}\label{thm_equiv_conj_intro}
Conjectures \ref{conj_tak_higher_genus_bps_intro} and \ref{conj_refined_joyce_conj_intro} are equivalent. Furthermore, assuming that these conjectures hold and denoting by 
$\Omega_d^{\PP^2}(y^{\frac{1}{2}})$
(resp.\ $\Omega_d^{\PP^2/E}(\hbar)$)
the common value of the $\Omega_{d,\chi}^{\PP^2}(y^{\frac{1}{2}})$
(resp.\ $\Omega_{d,k}^{\PP^2/E}(\hbar)$), we have
\[ \Omega_d^{\PP^2/E}(\hbar)
=\frac{1}{3d} \Omega_d^{\PP^2}(y^{\frac{1}{2}}) \]
after the change of variables $y=e^{i\hbar}$.
\end{thm}

We reformulate the result of Theorem \ref{thm_log_bps_local_bps_refined_intro}
after summation over all the possible values of $\chi$ on the sheaf side, that is, over all possible contact points on the Gromov-Witten side.
Tensoring by $\cO(1)$ gives an isomorphism between 
$M_{d,\chi}$ and $M_{d,\chi+d}$. Thus, $\Omega_{d,\chi}(y^{\frac{1}{2}})$ can only depend on $\chi$
through $\chi\!\!\!\!\mod d$. 
Define 
\[ \Omega_d^{\PP^2}(y^{\frac{1}{2}})
\coloneq \frac{1}{d} \sum_{\chi\!\!\!\!\mod d} \Omega_{d,\chi}^{\PP^2}(y^{\frac{1}{2}})\,\]
by averaging over the $d$ possible values of 
$\chi\!\!\!\mod d$.
If Conjecture \ref{conj_refined_joyce_conj_intro} is true, then $\Omega_{d,\chi}^{\PP^2}(y^{\frac{1}{2}})$ is in fact independent of $\chi$, and so $\Omega_d^{\PP^2}(y^{\frac{1}{2}})$ is the common value of the $\Omega_{d,\chi}^{\PP^2}(y^{\frac{1}{2}})$.
We define 
\[ \bar{F}^{NS}(y^{\frac{1}{2}},Q) 
\coloneq i \sum_{d \in \Z_{>0}}
\sum_{\ell \in \Z_{>0}} \frac{1}{\ell} \frac{\Omega_d^{\PP^2}(y^{\frac{\ell}{2}})}{
y^{\frac{\ell}{2}}-y^{-\frac{\ell}{2}}} Q^{\ell d} \in \Q(y^{\pm \frac{1}{2}})[\![Q]\!]\,.\]
Using the change of variables $y=e^{i \hbar}$, we define series $\bar{F}_g^{NS}(Q) \in \Q[\![Q]\!]$ by the expansion
\[ \bar{F}^{NS}(y^{\frac{1}{2}},Q) =\sum_{g \in \Z_{\geqslant 0}} \bar{F}_g^{NS} (Q) 
(-1)^g \hbar^{2g-1}\,.\]
On the other hand, for every $g \in \Z_{\geqslant 0}$, we define
\[ \bar{F}_g^{\PP^2/E} (Q) 
\coloneq \sum_{d \in \Z_{>0}}
\frac{(-1)^{g+d-1}}{3d} N_{g,d}^{\PP^2/E} Q^d \in \Q[\![Q]\!]\,.\]
Then, Theorem \ref{thm_log_bps_local_bps_refined_intro}
implies the following result.

\begin{thm} \label{thm_NS_intro}
For every $g \in \Z_{\geqslant 0}$, we have 
\[  \bar{F}^{NS}_g(Q) =\bar{F}_g^{\PP^2/E} (Q) \,.\]
\end{thm}

Assuming Conjecture \ref{conj_refined_joyce_conj_intro}, 
Theorem \ref{thm_NS_intro} was previously conjectured in 
\cite[\S 8.6]{bousseau2018quantum_tropical}.

The series $\bar{F}^{NS}(y^{\frac{1}{2}},Q)$ can be viewed as a mathematically precise definition of the Nekrasov-Shatashvili limit of the refined 
topological string free energy of local $\PP^2$
\cite{MR3024275}. As such, the series $\bar{F}^{NS}_g(Q)$ are expected to be expressible in terms of quasimodular forms and to satisfy a version of the holomorphic anomaly equation. Proving such results starting from the definition of  $\bar{F}^{NS}(y^{\frac{1}{2}})$ in terms 
of Betti numbers of moduli of one-dimensional sheaves seems hopeless without a geometric understanding of the change of variables $y=e^{i\hbar}$ involved in the definition of the series $\bar{F}^{NS}_g(Q)$. Theorem 
\ref{thm_NS_intro} gives such geometric interpretation by identifying 
$\bar{F}^{NS}_g(Q)$ with a genus $g$ Gromov-Witten generating series.
The conjectures of \cite{MR3024275} about the series $\bar{F}_g^{NS}(Q)$ will be proved in \cite{HAE}
by working with the series $\bar{F}_g^{\PP^2/E}(Q)$ on the Gromov-Witten side, and then using Theorem \ref{thm_NS_intro} to conclude.

\subsection{Applications to moduli spaces of one-dimensional sheaves on $\PP^2$}

Theorems \ref{thm_log_bps_local_bps_intro} and 
\ref{thm_log_bps_local_bps_refined_intro} establish a  correspondence between relative Gromov-Witten theory of moduli spaces of $(\PP^2,E)$, and moduli spaces of one-dimensional Gieseker semistable sheaves on $\PP^2$.
Theorem \ref{thm_takahashi_precise} is a non-trivial application of this correspondence to the Gromov-Witten side, proved using knowledge on the sheaf side. 
In this section we collect results obtained using the correspondence the other way around, that is, using knowledge on the Gromov-Witten side to deduce non-trivial results on the sheaf side. 

\subsubsection{Divisibility of the topological Euler characteristics}
In order to state Theorem \ref{thm_divisibility_intro}
in the most elementary way, we focus on the smooth moduli spaces $M_{d,1}$ of one-dimensional Gieseker semistable sheaves 
of degree $d$ and with holomorphic Euler characteristic one, and we consider its topological Euler characteristic 
$e(M_{d,1}) \in \Z$. According to Theorem \ref{thm_joyce_conj_intro}, we have $e(M_{d,1})=(-1)^{d-1}
\Omega_d^{\PP^2}$. Thus, by Theorem \ref{thm_local_relative_intro}, we have 
$\frac{(-1)^{d-1}}{3d} \Omega_d^{\PP^2}= \Omega_{g,k}^{\PP^2/E}$ for every divisor 
$k$ of $d$. As reviewed after the statement of 
Theorem \ref{thm_takahashi_precise},
the primitive invariants $N_{0,d}^{\PP^2/E,d}=(-1)^{d-1}\Omega_{d,d}^{\PP^2/E}$
are positive integers. Therefore, we obtain the following result.

\begin{thm} \label{thm_divisibility_intro}
For every positive integer $d$, the topological Euler characteristic
$e(M_{d,1})$ is divisible by $3d$. 
\end{thm}

The divisibility given by Theorem \ref{thm_divisibility_intro} was previously experimentally observed for small values of $d$, and conjectured to hold in general in \cite{choi2018local}. 
It can be viewed as a quite elementary looking statement on the topology of a classical family of algebraic varieties \cite{MR1263210}. However, our proof is not so elementary, relying in an essential way on some back and forth interaction with Gromov-Witten theory. We don't know a different proof.
As we will review in Proposition \ref{prop_katz_conj}, $(-1)^{d-1}e(M_{d,1})$ is equal to the genus-$0$ Gopakumar-Vafa invariant $n_{0,d}^{K_{\PP^2}}$ defined by genus-$0$ Gromov-Witten theory of local $\PP^2$. In particular, these invariants can be extracted from solutions to the mirror Picard-Fuchs 
differential equation. To give a direct proof of Theorem  
\ref{thm_divisibility_intro} starting from this definition
seems to be a non-obvious number-theoretic question\footnote{In fact, even the integrality of $n_{0,d}^{K_{\PP^2}}$ is unclear from the point of view of the Picard-Fuchs equation.}.

A more general divisibility statement, at the level of Poincar\'e polynomials, has been observed for small values of $d$, and conjectured to hold in general in \cite[Conjecture 4.15]{choi2018local}: for every $d \in \Z_{>0}$, the Poincar\'e polynomial $P_{M_{d,1}}(y) \coloneq \sum_{j=0}^{\dim M_{d,1}} b_{2j}(M_{d,1}) y^j$ of $M_{d,1}$ should be divisible by 
$[3d]_y \coloneq \sum_{j=0}^{3d-1}y^j$. For $y=1$, this conjecture specializes to Theorem \ref{thm_divisibility_intro}. Our proof of Theorem 
\ref{thm_divisibility_intro} relies on connecting $e(M_{d,1})$ with some question in 
genus-$0$ Gromov-Witten theory and the integrality comes from the essentially enumerative nature of the Gromov-Witten 
invariants. Theorem
\ref{thm_log_bps_local_bps_refined_intro}
gives a connection of $P_{M_{d,1}}(y)$
to Gromov-Witten theory, but to an all genus question, far from being enumerative. 
Assuming Conjectures \ref{conj_refined_joyce_conj_intro}-\ref{conj_tak_higher_genus_bps_intro}, one should have 
\[ \Omega_d^{\PP^2/E}(\hbar)=
\frac{1}{3d} \Omega_d^{\PP^2}(y^{\frac{1}{2}}) \]
by Theorem \ref{thm_equiv_conj_intro}. Nevertheless, 
the expected conjectural multicovering structure of higher-genus open Gromov-Witten invariants in Calabi-Yau 3-folds\footnote{The invariants $N_{g,d}^{\PP^2/E}$ should be viewed as providing an algebro-geometric definition of open higher genus Gromov-Witten invariants of the noncompact Calabi-Yau 3-fold $(\PP^2-E) \times \A^1$.}, see \cite[formula 3.31]{MR1806596}, predicts that $\Omega_{d}^{\PP^2/E}(\hbar)$ should be the product of $\frac{[3d]_y}{3d}$ by a Laurent polynomial in $y^{\frac{1}{2}}$ with integer coefficients. Therefore, the Laurent polynomial
$\Omega_{d}(y^{\frac{1}{2}})$ should be divisible by $[3d]_y$, which is exactly  \cite[Conjecture 4.15]{choi2018local}. Previous evidence for \cite[Conjecture 4.15]{choi2018local} was limited to explicit checks in low degree. Theorems \ref{thm_log_bps_local_bps_refined_intro}-
\ref{thm_equiv_conj_intro} relate this conjecture to
the expected general multicovering structure of higher-genus Gromov-Witten theory.

\subsubsection{Towards $\chi$-independence of the Betti numbers}

According to Conjecture \ref{conj_refined_joyce_conj_intro}, the intersection Betti numbers $M_{d,\chi}$ of the moduli 
spaces $M_{d,\chi}$ should be independent of $\chi$.
Tensoring by $\cO(1)$ and Serre duality imply that the moduli spaces $M_{d,\chi}$ and 
$M_{d,\chi'}$ are isomorphic if $\chi=\pm \chi' \mod d$. 
However, if $d \geqslant 3$ and $\chi \neq \pm \chi' \mod d$, then the algebraic varieties 
$M_{d,\chi}$ and $M_{d,\chi'}$ 
are not isomorphic by \cite[Theorem 8.1]{woolf2013nef}, which is why Conjecture \ref{conj_refined_joyce_conj_intro} is not obvious. We do not know how to prove Conjecture 
\ref{conj_refined_joyce_conj_intro} but we remark that the connection with Gromov-Witten theory given by Theorem 
\ref{thm_log_bps_local_bps_refined_intro} provides a non-trivial result. Indeed, the Gromov-Witten side of 
Theorem \ref{thm_log_bps_local_bps_refined_intro}
only depends on $\chi$ through $\gcd(d,\chi)$, and so
we obtain the following special case of Conjecture 
\ref{conj_refined_joyce_conj_intro}.

\begin{thm} \label{thm_chi_independence_intro}
The intersection Betti numbers $Ib_j(M_{d,\chi})$ depend on 
$\chi$ only through $\gcd(d,\chi)$, that is, for every $d\in \Z_{>0}$ and $\chi,\chi' \in \Z$ such that 
$\gcd(d,\chi)=\gcd(d,\chi')$, we have 
\[ Ib_j(M_{d,\chi})=Ib_j(M_{d,\chi'})\,.\]
\end{thm}

In order to better appreciate Theorem \ref{thm_chi_independence_intro}, we mention a closely related result. 
Let $C$ be a smooth projective curve of genus $g$ and let $M_{d,\chi}^{H}(C)$ be the moduli space of semistable Higgs bundles on $C$ of rank $d$ and degree $d(g-1)+\chi$ \cite{MR887284}. 
Through the spectral cover construction, one can view $M_{d,\chi}^{H}(C)$ as a moduli space of one-dimensional Gieseker semistable sheaves on the total space $K_C$ of the canonical line bundle of $C$, and so $M_{d,\chi}^{H}(C)$ is a close cousin of $M_{d,\chi}$.

\begin{thm}[\cite{MR2453601}\footnote{\cite{MR2453601}
only treats the case $\gcd(d,\chi)=1$. We remark that the same argument works in general.}]\label{thm_higgs_moduli}
The intersection Betti numbers $Ib_j(M_{d,\chi}^H(C))$ depend on 
$\chi$ only through $\gcd(d,\chi)$, that is, for every $d\in \Z_{>0}$ and $\chi,\chi' \in \Z$ such that 
$\gcd(d,\chi)=\gcd(d,\chi')$, we have 
\[ Ib_j(M_{d,\chi}^H(C))=Ib_j(M_{d,\chi'}^H(C))\,.\]
\end{thm}

By non-abelian Hodge theory,
$M_{d,\chi}^H(C)$ is homeomorphic to the character variety $M_{d,\chi}^{C_g}$ of $GL_d$ local systems on a once-punctured genus $g$ compact topological surface, with diagonal monodromy $e^{\frac{2i\pi \chi}{d}}$ around the puncture.
For $\chi$ and $\chi'$ such that $\gcd(d,\chi)=\gcd(d,\chi')$, $M_{d,\chi}^{C_g}$
and $M_{d,\chi'}^{C_g}$ are Galois conjugates of an algebraic variety defined over $\Q(\zeta)$, where $\zeta$ is a primitive $d$-th root of unity. Theorem \ref{thm_higgs_moduli} then follows from the facts that intersection Betti numbers have a purely algebraic definition (through their \'etale description \cite{MR751966}), and are a topological invariant \cite{MR696691}. 

It is amusing to compare the proofs of Theorem \ref{thm_chi_independence_intro} and Theorem 
\ref{thm_higgs_moduli}. In both cases, it is completely unclear how to find a direct connection between the moduli spaces for distinct values of $\chi$. In order to prove Theorem 
\ref{thm_higgs_moduli},  one uses non-abelian Hodge theory and the different values of $\chi$ become connected by the action of a Galois group. 
In our proof of Theorem \ref{thm_chi_independence_intro}, 
we translate the problem into a completely different looking question in Gromov-Witten theory involving contact points with an elliptic curve, 
and the different values of $\chi$ become connected by the action of the monodromy group of the elliptic curve.

\subsection{Local-relative correspondence and real algebraic geometry}
In \cref{section_local_relative}, we explain how Theorem \ref{thm_local_relative_intro} is compatible with the specialization to 
$\PP^2$ of the local-relative correspondence of \cite{MR3948687}.
In \cref{section_real_gw}, we present a conjectural connection between genus-$0$ real Gromov-Witten invariants of $K_{\PP^2}$ and Euler characteristics of the real loci of moduli spaces of semistable one-dimensional stable sheaves on $\PP^2$. 
We refer to
\cref{section_more_gw} for precise statements.

\subsection{Plan of the paper}
In \cref{section_preliminary} we prove preliminary results on the Gromov-Witten side: contact points are torsion points and by a monodromy argument $N_{d,g}^{\PP^2/E,p}$
only depends on $p$ through $d(p)$.
In \cref{section_dimension_one_sheaves}
we show that the invariants $\Omega_{d,\chi}^{\PP^2}(y^{\frac{1}{2}})$ are the
refined Donaldson--Thomas invariants for
one-dimensional sheaves on $K_{\PP^2}$ and 
we prove Theorem 
\ref{thm_joyce_conj_intro} on the 
$\chi$-independence of the invariants 
$\Omega_{d,\chi}^{\PP^2}$.
In \cref{section_comparison_cps_wall_structure}, we compare the scattering diagram $S(\fD^\iin_{\cl^+})$ of 
\cite{bousseau2019scattering} with the wall-structure 
$\mathcal{S}$ of \cite{cps}.
In \cref{section_takahashi_proof}, we combine the previous results with
\cite{bousseau2019scattering} and \cite{gabele2019tropical} to prove Theorems \ref{thm_log_bps_local_bps_intro}-\ref{thm_local_relative_intro}.
In \cref{section_higher_genus_refinement},
we prove Theorems 
\ref{thm_log_bps_local_bps_refined_intro},
\ref{thm_equiv_conj_intro}, \ref{thm_NS_intro}
on the correspondence between the refined 
invariants $\Omega_{d,\chi}^{\PP^2}(y^{\frac{1}{2}})$ and the higher-genus Gromov-Witten invariants $N_{g,d}^{\PP^2/E}$.
In \cref{section_more_gw}, we discuss further results, including the compatibility with the local-relative correspondence
of \cite{MR3948687}, and a conjectural sheaves/Gromov-Witten correspondence in real algebraic geometry of 
$K_{\PP^2}$.
In \cref{section_heuristic_explanation}, 
we end with a heuristic explanation, based on a combination of hyperkähler rotation
and mirror symmetry, for the main connection unraveled in this paper between sheaf counting on $\PP^2$ and relative Gromov-Witten theory of $(\PP^2,E)$.

\subsection{Acknowledgments}
I thank Jinwon Choi, Michel van Garrel, Rahul Pandharipande and Juliang Shen for conversations stimulating my interest in N.\ Takahashi's conjecture. I thank Tim Gr\"afnitz for discussions on his work \cite{gabele2019tropical}. 
I thank Younghan Bae for a clarifying discussion related to Lemma 
\ref{lem_contact_torsion}. I thank Michel van Garrel, Penka Georgieva, Rahul Pandharipande, Vivek Shende and Bernd Siebert for invitations to conferences and seminars where this work has been presented and for related discussions.

I acknowledge the support of Dr.\ Max R\"ossler, the Walter Haefner Foundation and the ETH Z\"urich
Foundation.

\section{Preliminary results on the Gromov-Witten side}
\label{section_preliminary}

In this section, we prove two elementary results on maximal contact
Gromov-Witten theory of $(\PP^2,E)$. In \S \ref{section_contact_torsion}, we show that contact points with $E$ are torsion points. In \S \ref{section_monodromy}, we describe the constraints imposed by monodromy on the dependence on the contact point.

\subsection{Contact points are torsion points}
\label{section_contact_torsion}
Let $E$ be a smooth cubic curve in the complex projective plane $\PP^2$.
We fix $p_0$ as one of the $9$ flex points of $E$.
For every $d \in \Z_{>0}$ and $g \in \Z_{\geqslant 0}$, let
$\overline{M}_{g}(\PP^2/E,d)$ be the moduli space of 
genus $g$ degree $d$ stable maps to $\PP^2$ relative to $E$,
with maximal contact order with $E$ at a single point.
The moduli space $\overline{M}_{g}(\PP^2/E,d)$ is a proper Deligne-Mumford stack
\cite{MR1882667, MR1938113}.

\begin{lem}\label{lem_contact_torsion}
Let $p$ be a point of $E$ in the image of the evaluation morphism 
\[\ev \colon \overline{M}_{g}(\PP^2/E,d) 
\rightarrow E\,\]
at the contact point with $E$.
Then $p-p_0$ is a 
$(3d)$-torsion point of $\Pic^0(E)$.
\end{lem}

\begin{proof}
Let $N_{E|\PP^2}$ be the normal line bundle to $E$ in $\PP^2$. We have 
$N_{E|\PP^2}=\cO_{\PP^2}(3)|_E=\cO_E(9p_0)$.
By definition of relative stable maps \cite{MR1882667, MR1938113},
a point of 
$\overline{M}_{g}(\PP^2/E,d)$
is a map $f \colon C \rightarrow 
B_0 \cup \dots \cup B_n$, where 
the target is obtained from $\PP^2$
by $n$ successive degenerations to the normal cone of $E$.
We have $B_0 =\PP^2$ containing the cubic curve $E=E_0$, and, for 
$1\leqslant j \leqslant n$, each irreducible component $B_j$ is isomorphic to the 
$\PP^1$-bundle $\PP(N_{E|\PP^2} \oplus \cO_E)$ over $E$, with two sections 
$E_{j-1}$ and $E_j$. The components $B_j$ and $B_{j+1}$ are glued together along $E_j$.
We define $p_{0j}$ by induction on $j$:
$p_{00} \coloneq p_0 \in E_0=E$, and for $j
\geqslant 1$, $p_{0j} \in E_j$ is the intersection with $E_j$ of the $\PP^1$-fiber of 
$B_j$ passing through $p_{0(j-1)}$.

The maximal tangency condition for $f$ is imposed along the divisor $E_n$ in $B_n$:
$f(C)$ intersects $E_n$ in a unique point 
$q_n \in E_n$ with multiplicity $3d$.
We have to show that $3d(q_n-p_{0n})
=0 \in \Pic^0(E_n)$.

For every $0 \leqslant j \leqslant n$, 
let $C_j$ be the union of irreducible components of $C$ mapping to $B_j$, and let $f_j \colon C_j \rightarrow B_j$ be the 
restriction of $f$ to $C_j$.
For every $0 \leqslant j \leqslant n$, $f(C)$ intersects $E_j$ in finitely many points 
$q_{kj}\in E_j$ with multiplicities $m_{kj}$. We show by induction on $j$ that
$\sum_k m_{kj}(q_{kj}-p_{0j})
=0$ in $\Pic^0(E_j)$.

We first treat the case $j=0$.
As $\Pic(\PP^2)=\Z$, the curve $f_0(C_0)$ is linearly equivalent to a 
multiple of the tangent line to $E$ at the flex point $p_0$. Restriction of this linear equivalence to $E_0$ gives
$\sum_k m_{k0}(q_{k0}-p_{00})
=0$ in $\Pic^0(E_0)$.

Assume by induction that 
\[\sum_k m_{k(j-1)}(q_{k(j-1)}-p_{0(j-1)})
=0\] 
in $\Pic^0(E_{j-1})$. 
As $B_j=\PP(N_{E|\PP^2} \oplus \cO_E)$
and $N_{E|\PP^2}=\cO_E(9p_0)$, 
$\sum_k m_{kj} q_{kj}$ differs from
$\sum_k m_{k(j-1)} q_{k(j-1)}$
by a multiple of $p_{0j}$
in $\Pic(E_j) \simeq \Pic(E_{j-1})$.
It follows that
\[\sum_k m_{kj}(q_{kj}-p_{0j})
=0\] 
in $\Pic^0(E_{j})$.

\end{proof}

\subsection{Monodromy constraint}
\label{section_monodromy}

In this section we prove Lemma \ref{lem_monodromy}, which was used in the introduction to reduce the dependence on $p$ of the invariants $N_{g,d}^{\PP^2/E,p}$ to a dependence on $d(p)$.

Recall that we denote by $P_d$ the set of $(3d)^2$ points $p$ of $E$ such that $p-p_0$ is a $(3d)$-torsion point in $\Pic^0(E)$, and  that, 
for $p \in \bigcup_{d \geqslant 1} P_d$, we denote by
$d(p)$ the smallest $d \in \Z_{>0}$ such that 
$p \in P_d$.
For every $p \in P_d$ and $p_0$ a flex point of $E$, we denote $\ord(p-p_0)$ the order of $p-p_0$ in the group 
$\Pic^0(E)[3d]$ of $(3d)$-torsion 
points of $\Pic^0(E)$

\begin{lem}\label{lem_divisibility}
Let $d \in \Z_{>0}$, $p \in P_d$, and let $p_0$ be a flex point of $E$.
\begin{enumerate}
    \item If $3$ divides $d(p)$, then $\ord(p-p_0)=3d(p)$.
    \item If $3$ does not divide $d(p)$, then 
    $\ord(p-p_0)=3d(p)$ or $\ord(p-p_0)=d(p)$.
\end{enumerate}
\end{lem}

\begin{proof}
We have $p \in P_{d(p)}$, so $p-p_0$ is $(3d(p))$-torsion, and so $\ord(p-p_0)$ divides $3d(p)$. If 
$\ord(p-p_0)=3d(p)$, then the proof is finished. 

Therefore, we assume that 
$\ord(p-p_0)<3d(p)$. By minimality of $d(p)$, $\ord(p-p_0)$ is not divisible by $3$ (else one would have $p \in P_{\ord(p-p_0)/3}$). 
It follows that $\ord(p-p_0)$ divides 
$d(p)$, and so, using $p \in P_{\ord(p-p_0)}$ and  the minimality of $d(p)$ again, that
$\ord(p-p_0)=d(p)$.
\end{proof}

The following result is a variant of \cite[Lemma 4.13]{choi2018log}. In particular, the idea of the proof to use the freedom in the choice of the flex point comes from there.

\begin{lem} \label{lem_monodromy}
For every $d \in \Z_{>0}$
and $g\in \Z_{\geqslant 0}$, the 
Gromov-Witten invariants $N_{g,d}^{\PP^2/E,p}$ only depend on the point $p \in P_d$ through the integer $d(p)$. 
\end{lem}

\begin{proof}
We fix a $d \in \Z_{>0}$ and two points $p, p' \in P_d$
such that $d(p)=d(p')$. We have to show that 
$N_{g,d}^{\PP^2/E,p}=N_{g,d}^{\PP^2/E,p'}$.

We show below that there exists a flex point $p_0$ of $E$ such that $\ord(p-p_0)=\ord(p'-p_0)$.
This will be enough. Indeed, the monodromy of the family of smooth cubics $E$ in $\PP^2$ maps surjectively to $SL(2,\Z)$
acting on $\Pic^0(E)[3d] \simeq (\Z/(3d))^2$, and so if 
$\ord(p-p_0)=\ord(p'-p_0)$, then $p$ and $p'$ can be related by monodromy and so the result follows from deformation invariance in relative Gromov-Witten theory.

We first choose $p_0$ any flex point of $E$. If $3$ divides 
$d(p)=d(p')$, then by Lemma 
\ref{lem_divisibility}, $\ord(p-p_0)=3d(p)=3d(p')=\ord(p-p_0)$ and we are done.

If $3$ does not divide $d(p)=d(p')$ and $\ord(p-p_0)=\ord(p'-p_0)$, then we are also done.
Therefore, by Lemma \ref{lem_divisibility}, and up to exchanging $p$ and $p'$, we can assume that $3$ does not divide $d(p)=d(p')$,
$\ord(p-p_0)=3d(p)$ and $\ord(p'-p_0)=d(p)$.
Then $d(p)(p-p_0)$ is nonzero $3$-torsion.
If $t$ is nonzero $3$-torsion, then 
$d(p)t$ is also nonzero $3$-torsion as $3$ does not divide 
$d(p)$. So if $t_1$ and $t_2$ are both nonzero $3$-torsion,
with $d(p)t_1=d(p)t_2$, then $t_1=t_2$. It follows that there exists $t$ nonzero $3$-torsion such that $d(p)t \neq -d(p)(p-p_0)$, and so $\ord(p-p_0+t)=3d(p)$ and 
$\ord(p'-p_0+t)=3d(p)$. So it is enough to replace 
$p_0$ by $p_0-t$.
\end{proof}

\section{One-dimensional Gieseker semistable sheaves on $\PP^2$}
\label{section_dimension_one_sheaves}

\subsection{Statement of the $\chi$-independence of $\Omega_{d,\chi}^{\PP^2}$}
\label{chi_independ_statement}

We refer to \cite{huybrechts2010geometry} as general reference on Gieseker semistable sheaves. For every $d \in \Z_{>0}$ and $\chi \in \Z$, let $M_{d,\chi}$ be the moduli space of S-equivalence classes of Gieseker semistable sheaves on $\PP^2$, supported on curves of degree $d$ and of Euler characteristic $\chi$.
It is proved in \cite{MR1263210} that, for every $d \in \Z_{>0}$ and $\chi \in \Z$, $M_{d,\chi}$ is a
nonempty integral normal projective variety of dimension $d^2+1$. If $d$ and $\chi$ are coprime, then $M_{d,\chi}$ is smooth.
However, $M_{d,\chi}$ is generally singular if $d$ and $\chi$ are not coprime.

Nevertheless, the intersection cohomology groups $IH^j(M_{d,\chi}, \Q)$ behave as well as cohomology of a smooth projective variety
\cite{MR572580,MR696691,MR751966}.
We denote by $Ib_{2j}(M_{d,\chi})$ the corresponding intersection Betti numbers and by $Ie(M_{d,\chi})$ the corresponding intersection topological Euler characteristic\footnote{In \cite[\S 2.4]{bousseau2019scattering}, the intersection Euler characteristic is denoted by $Ie^+(M_{d,\chi})$.}.
According to \cite[Corollary 6.1.3]{bousseau2019scattering} the odd degree part of intersection cohomology of $M_{d,\chi}$ vanishes, so 
$Ib_{2k+1}(M_{d,\chi})=0$ for every $k \in \Z$, and
$Ie(M_{d,\chi}) \in \Z_{>0}$.

For every $d \in \Z_{>0}$ and $\chi \in \Z$, we define 
\[ \Omega_{d,\chi}^{\PP^2}(y^{\frac{1}{2}}) \coloneq 
(-y^{\frac{1}{2}})^{-\dim M_{d,\chi}}
\sum_{j=0}^{\dim M_{d,\chi}}
Ib_{2j}(M_{d,\chi}) y^{j} \in \Z[y^{\pm \frac{1}{2}}]\,,\]
and 
\[ \Omega_{d,\chi}^{\PP^2}
\coloneq 
\Omega_{d,\chi}^{\PP^2}
(y^{\frac{1}{2}}=1)=(-1)^{\dim M_{d,\chi}}
Ie(M_{d,\chi})=(-1)^{d-1} Ie(M_{d,\chi}) \in \Z\,.\]

\begin{thm}[=Theorem \ref{thm_joyce_conj_intro}] \label{thm_joyce_conj}
The intersection Euler characteristics 
$Ie(M_{d,\chi})$ are independent of $\chi$, that is, for every $d \in \Z_{\geqslant 1}$
and $\chi, \chi' \in \Z$, we have 
\[ \Omega_{d,\chi}^{\PP^2}=\Omega_{d,\chi'}^{\PP^2}\,.\]
\end{thm}

The proof of Theorem \ref{thm_joyce_conj} is given in the remaining part of  \cref{section_dimension_one_sheaves}. 
Using the framework of \cite{meinhardt2015donaldson}, we first prove in Proposition \ref{prop_joyce_song}
that the signed intersection Euler characteristics $\Omega_{d,\chi}^{\PP^2}$
coincide with the Joyce-Song
\cite{MR2951762} Donaldson-Thomas invariants of the local Calabi-Yau 3-fold $K_{\PP^2}$ total space of the canonical line bundle $\cO(-3)$ of $\PP^2$. 
This reduces Theorem \ref{thm_joyce_conj} to a general conjecture of Joyce-Song (Conjecture 6.20 
of \cite{MR2951762}) about Donaldson-Thomas invariants for dimension $1$ sheaves on Calabi-Yau 3-folds. 
Then, it remains to explain that this conjecture is known in the special case of $K_{\PP^2}$. This is done in 
Proposition \ref{prop_joyce_song_conjecture} by combination of \cite{MR2264664, MR2892766, MR2215440, MR2250076}.

\subsection{Comparison with Donaldson-Thomas invariants of $K_{\PP^2}$.}
\label{section_dt}

Let $\Coh_{\leqslant 1}(K_{\PP^2})$ be the  abelian category of coherent sheaves on 
$K_{\PP^2}$ supported in dimension less or equal to $1$. We have 
\begin{align*}
K_0(\Coh_{\leqslant 1}(K_{\PP^2})) &\simeq \Z^2
\\ 
E \longmapsto & (d(E),\chi(E))\,.
\end{align*}

The map 
\begin{align*}
Z \colon K_0(\Coh_{\leqslant 1}(K_{\PP^2})) &\longrightarrow \C \\
(d,\chi) &\longmapsto Z_{d,\chi}
\coloneq -\chi+id 
\end{align*}
defines a stability condition on $\Coh_{\leqslant 1}(K_{\PP^2})$.
The notions of stability and semistability on $\Coh_{\leqslant 1}(K_{\PP^2})$ coincide with Gieseker stability and semistability. 
In this context, Joyce-Song (see \cite[\S 6.4]{MR2951762} and \cite[\S 4.4-4.5]{MR2892766}) define rational Donaldson-Thomas invariants 
\[\overline{\Omega}_{d,\chi}^{\PP^2} \in \Q\,,\] 
denoted by $\bar{DT}^{(0,0,\beta,k)}$
in \cite{MR2951762} and $N_{n,\beta}$ in 
\cite{MR2892766},
which appear as the classical limit $y^{\frac{1}{2}} \rightarrow 1$ of rational refined Donaldson-Thomas invariants
\[\overline{\Omega}_{d,\chi}^{\PP^2}(y^{\frac{1}{2}}) \in \Q(y^{\pm \frac{1}{2}})\,.\]

\begin{lem} \label{lem_zero_section}
Every semistable object $E$ in $\Coh_{\leqslant 1}(K_{\PP^2})$ not supported in dimension $0$, that is, with $d(E) \geq 1$, is scheme-theoretically supported on the zero section $\PP^2$.
\end{lem}

\begin{proof}
Through the classical spectral cover construction, as reviewed for example in \cite[Proposition 2.2]{MR4158461}, the category $\Coh_{\leqslant 1}(K_{\PP^2})$
is equivalent to the category of Higgs sheaves $(F,\phi)$ on $\PP^2$, with $F$ supported in dimension less or equal to $1$. Furthermore, under this equivalence,
semistability for $\Coh_{\leqslant 1}(K_{\PP^2})$ corresponds to semistability for Higgs sheaves \cite[Lemma 2.9]{MR4158461}.

Let $(F,\phi)$ be a semistable Higgs sheaf on $\PP^2$ with $d(F) \geq 1$. We have to show that $\phi=0$.  Remark that $(F,\phi)$
semistable implies that $(F \otimes K_{\PP^2}, \phi \otimes 1)$ is semistable. 
The sheaf $\Ima \phi$ is a $\phi$-invariant quotient of $F$
and a $\phi \otimes 1$-invariant subsheaf of $F \otimes K_{\PP^2}$.
Therefore, if $\Ima \phi \neq 0$, semistability implies
\[ \Arg(-\chi(F)+id(F)) \leqslant \Arg(-\chi(\Ima \phi)+id(\Ima \phi))
\leqslant \Arg(-\chi(F \otimes K_{\PP^2})+id(F))\,.\]
We have $\chi(F\otimes K_{\PP^2})=\chi(F)-3d(F)$, and so, if $d(F) \geq 1$, the above inequality cannot hold.
We conclude that $\Ima \phi=0$, that is, $\phi=0$.

\end{proof}

\begin{prop} \label{prop_joyce_song}
For every $d \in \Z_{\geqslant 1}$ and 
$\chi \in \Z$, the rational Donaldson-Thomas invariants $\overline{\Omega}_{d,\chi}^{\PP^2}$
and rational refined Donaldson-Thomas invariants $\overline{\Omega}_{d,\chi}^{\PP^2}(y^{\frac{1}{2}})$
of $K_{\PP^2}$
are expressed in terms of the invariants 
$\Omega_{d,\chi}^{\PP^2}$
and $\Omega_{d,\chi}^{\PP^2}(y^{\frac{1}{2}})$ of the moduli spaces of one-dimensional semistable sheaves on $\PP^2$ as
\[ \overline{\Omega}_{d,\chi}^{\PP^2}
=\sum_{\substack{\ell \in \Z_{\geqslant 1}
\\ \ell|(d,\chi)}}
\frac{1}{\ell^2}
\Omega_{\frac{d}{\ell},\frac{\chi}{\ell}}^{\PP^2} \,,\]
and 
\[ \overline{\Omega}_{d,\chi}^{\PP^2}(y^{\frac{1}{2}})=
\sum_{\substack{\ell\in \Z_{\geqslant 1} \\\ell|(d,\chi)}}
\frac{1}{\ell} \frac{y^{\frac{1}{2}}-y^{-\frac{1}{2}}}{y^{\frac{\ell}{2}}-y^{-\frac{\ell}{2}}} \Omega_{\frac{d}{\ell},\frac{\chi}{\ell}}^{\PP^2}(y^{\frac{\ell}{2}})\,.\]
\end{prop}

\begin{proof}
By Lemma \ref{lem_zero_section}, for every $d \in \Z_{\geqslant 1}$ and $\chi \in \Z$, the moduli spaces $M_{d,\chi}$ and moduli stacks $\fM_{d,\chi}$ of Gieseker semistable sheaves on $\PP^2$ are also the moduli spaces and moduli stacks of $Z$-semistable objects of $\Coh_{\leqslant 1}(K_{\PP^2})$.
As we have $\Ext^2_{\PP^2}(E,E)=0$ for every Gieseker semistable sheaf $E$ on $\PP^2$ not supported in dimension $0$, 
the moduli stack $\fM_{d,\chi}$ is smooth for every $d \in \Z_{\geqslant 1}$ and $\chi \in \Z$ and so its Behrend function is constant equal to $(-1)^{\dim \fM_{d,\chi}}$. Remark that the sheaves supported
in dimension $0$ are exactly those for which $Z_{d,\chi}\in \R_{<0}$, that is, 
$\frac{1}{\pi} \Arg Z_{d,\chi}=1$.

The definition of the rational Donaldson-Thomas invariants $\overline{\Omega}_{d,\chi}^{\PP^2}$ in \cite{MR2951762} uses the motivic Hall algebra of 
$\Coh_{\leqslant 1}(K_{\PP^2})$. 
Using a version of \cite[Lemma 3.5.5]{bousseau2019scattering} 
to go from the motivic Hall algebra
to numerical identities, we can rewrite this definition as follows. 
The symmetrized virtual Poincar\'e rational functions 
$\tilde{b}(\fM_{d,\chi})(y^{\frac{1}{2}})$ are defined according to
\cite[\S 3.3]{bousseau2019scattering}.
There exists a unique set of 
$\overline{\Omega}_{d,\chi}(y^{\frac{1}{2}}) \in \Z(y^{\frac{1}{2}})$,
$d \in \Z_{\geqslant 1}$, $\chi \in \Z$, such that, for every $\phi \in (0,1)$, we have the equality of power series
\[ 1+\sum_{\substack{d \in \Z_{\geqslant 1},\chi \in \Z \\ \frac{1}{\pi} \Arg Z_{d,\chi}=\phi} } \tilde{b}(\fM_{d,\chi})(y^{\frac{1}{2}})
z^{(d,\chi)}=\exp 
\left( -\sum_{\substack{d \in \Z_{\geqslant 1},\chi \in \Z \\ 
\frac{1}{\pi} \Arg Z_{d,\chi}=\phi}} 
\frac{\overline{\Omega}^{\PP^2}_{d,\chi}(y^{\frac{1}{2}})}{y^{\frac{1}{2}}
-y^{-\frac{1}{2}}}
z^{(d,\chi)}
\right)
\,.\]
Then the non-trivial result of \cite{MR2951762} is the existence of the limit
\[ \overline{\Omega}^{\PP^2}_{d,\chi} 
\coloneq \lim_{y^{\frac{1}{2}} \rightarrow 1} \overline{\Omega}^{\PP^2}_{d,\chi}
(y^{\frac{1}{2}}) \in \Q\,.\]

On the other hand, let $\Coh_{\leqslant 1}(K_{\PP^2})(\phi)$ be the abelian subcategory of
$\Coh_{\leqslant 1}(K_{\PP^2})$ whose objects are $0$ and the nonzero Gieseker semistable objects with $\frac{1}{\pi} \Arg Z=\phi$. We check that the abelian category $\Coh_{\leqslant 1}(K_{\PP^2})(\phi)$ satisfies the technical conditions (1)-(8)
required to apply 
\cite[Theorem 1.1]{meinhardt2015donaldson}. Conditions (1)-(6) are general requirements on deformation theory and moduli spaces of objects, and they are automatically satisfied for categories of coherent sheaves. The non-trivial conditions are (7) and (8). For (7), one needs to check that 
\[ (E,F):=\dim \Hom(E,F)-\dim \Ext^1(E,F)\] 
is locally constant as a function of the objects $E$ and $F$. This is clear if either $E$ or $F$ are zero. If both $E$ and $F$ are non-zero objects of $\Coh_{\leqslant 1}(K_{\PP^2})(\phi)$, then as $\phi \neq 1$, they are not supported in dimension $0$, and so by Lemma \ref{lem_zero_section}, they are scheme theoretically supported on the zero section $\PP^2$. Similarly, any non-zero object of $\Coh_{\leqslant 1}(K_{\PP^2})(\phi)$ is supported on the zero section, so $\Hom(E,F)\simeq \Hom_{\PP^2}(E,F)$, $\Ext^1(E,F) \simeq \Ext^1_{\PP^2}(E,F)$, and 
$(E,F)=\dim \Hom_{\PP^2}(E,F)-\dim \Ext^1_{\PP^2}(E,F)$, where $E$ and $F$ are viewed as sheaves on $\PP^2$. Furthermore, as $E$ and $F$ are semistable of the same phase $\phi$ on $\PP^2$ and not supported in dimension $0$, we have using Serre duality that \[ \dim \Ext^2_{\PP^2}(E,F)=\dim \Hom_{\PP^2}(F, E \otimes K_{\PP^2})=0\,.\] 
Hence, $(E,F)=\chi_{\PP^2}(E,F)$, where $\chi_{\PP^2}$ is the Euler form for sheaves on $\PP^2$. By the Riemann-Roch theorem on $\PP^2$, $\chi_{\PP^2}(E,F)$ only depends on the Chern characters $\ch(E)$ and $\ch(F)$, so (7) is satisfied.
For (8), one needs to check that $(E,F)=(F,E)$ for all $E$ and $F$ in $\Coh_{\leqslant 1}(K_{\PP^2})(\phi)$. By the previous discussion, it it enough to check that $\chi_{\PP^2}(E,F)=\chi_{\PP^2}(F,E)$. But as $E$ and $F$ are supported in dimension 1, we have 
\[ \chi_{\PP^2}(E,F)=-d(E)d(F)\] by Riemann-Roch theorem for 
$\PP^2$ (see for example \cite[Lemma 2.1.2]{bousseau2019scattering}), where $d(E)$ and $d(F)$ are the degrees of $E$ and $F$ respectively. As this expression is symmetric under the exchange of $E$ and $F$, (8) is also satisfied.

Applying \cite[Theorem 1.1]{meinhardt2015donaldson}, we obtain
\[ 1+\sum_{\substack{d \in \Z_{\geqslant 1},\chi \in \Z \\ \frac{1}{\pi} \Arg Z_{d,\chi}=\phi} } \tilde{b}(\fM_{d,\chi})(y^{\frac{1}{2}})
z^{(d,\chi)}=\exp 
\left( -\sum_{\substack{d \in \Z_{\geqslant 1},\chi \in \Z \\ 
\frac{1}{\pi} \Arg Z_{d,\chi}=\phi}} 
\sum_{\ell \geqslant 1}
\frac{1}{\ell}
\frac{\Omega^{\PP^2}_{d,\chi}(y^{\frac{\ell}{2}})}{y^{\frac{\ell}{2}}
-y^{-\frac{\ell}{2}}}
z^{(\ell d,\ell \chi)}
\right)
\,.\]
Comparing the previous equalities, we obtain that, for every $d\in \Z_{\geqslant 1}$ and $\chi \in \Z$, we have 
\[ \overline{\Omega}_{d,\chi}^{\PP^2}(y^{\frac{1}{2}})=
\sum_{\substack{\ell\in \Z_{\geqslant 1} \\\ell|(d,\chi)}}
\frac{1}{\ell} \frac{y^{\frac{1}{2}}-y^{-\frac{1}{2}}}{y^{\frac{\ell}{2}}-y^{-\frac{\ell}{2}}} \Omega_{\frac{d}{\ell},\frac{\chi}{\ell}}^{\PP^2}(y^{\frac{\ell}{2}})\,,\]
and so, taking the limit $y^{\frac{1}{2}}\rightarrow 1$, we obtain \[ \overline{\Omega}_{d,\chi}^{\PP^2}
=\sum_{\substack{\ell \in \Z_{\geqslant 1}
\\ \ell|(d,\chi)}}
\frac{1}{\ell^2}
\Omega_{\frac{d}{\ell},\frac{\chi}{\ell}}^{\PP^2} \,,\]
as desired.

Remark that the above use of \cite[Theorem 1.1]{meinhardt2015donaldson} gives a more direct proof of the existence of the limit 
\[\lim_{y^{\frac{1}{2}} \rightarrow 1} \overline{\Omega}^{\PP^2}_{d,\chi}
(y^{\frac{1}{2}})\,,\] 
and so our result is in fact logically independent of the hard 
content of \cite{MR2951762}.

\end{proof}

\subsection{Proof of the $\chi$-independence of $\Omega_{d,\chi}^{\PP^2}$}
\label{section_chi_independence}

\begin{prop}\label{prop_joyce_song_conjecture}
For every $d \in \Z_{\geqslant 1}$, $\chi, \chi' \in \Z$, we have 
\[ \overline{\Omega}^{\PP^2}_{d,\chi}
=\sum_{\substack{\ell\in \Z_{\geqslant 1}
\\\ell|(d,\chi)}} \frac{1}{\ell^2}
\overline{\Omega}^{\PP^2}_{\frac{d}{\ell},1}\,.\]
\end{prop}

\begin{proof}
The following  sequence of arguments is also reviewed in \cite[Appendix A]{MR3861701}.

Proposition \ref{prop_joyce_song_conjecture} is the specialization to $K_{\PP^2}$ of a general conjecture about one-dimensional sheaf counting of Calabi-Yau 3-folds: see \cite[Conjecture 6.20]{MR2951762}
and \cite[Conjecture 6.3]{MR2892766}).
Toda
\cite[Theorem 6.4]{MR2892766} proved by wall-crossing in the derived category that this conjecture is equivalent to the strong rationality conjecture for stable pairs Pandharipande-Thomas invariants, see \cite[Conjecture 3.14]{MR2545686} and
\cite[Conjecture 6.2]{MR2892766}). 

Given the known ideal sheaves Donaldson-Thomas/stable pairs Pandharipande-Thomas correspondence, proved by wall-crossing in the derived category
for general Calabi-Yau 3-folds \cite{MR2813335} or by computation of both sides in the toric case \cite[\S 5]{MR2746343}, the strong rationality conjecture for stable pair Pandharipande-Thomas invariants can be translated into a strong rationality statement for ideal sheaves Donaldson-Thomas invariants.

As $K_{\PP^2}$ is a toric Calabi-Yau 3-fold, 
ideal sheaves Donaldson-Thomas invariants can be computed by localization and organized using the topological vertex formalism \cite{MR2264664}.
By a study of the explicit formulas coming from the topological vertex formalism, Konishi \cite[Theorem 1.3]{MR2215440}
proved that the strong rationality statement holds for local toric surfaces, 
and so in particular for $K_{\PP^2}$
(the proof was later generalized to arbitrary toric Calabi-Yau 3-folds in \cite{MR2250076}).
\end{proof}

We can now end the proof of
Theorem \ref{thm_joyce_conj}. Comparing  Proposition \ref{prop_joyce_song} and Proposition 
\ref{prop_joyce_song_conjecture}, we obtain that, for every $d \in \Z_{\geqslant 1}$ and
$\chi \in \Z$, we have $\Omega_{d,\chi}^{\PP^2}=\Omega_{d,1}^{\PP^2}$.
In particular, $\Omega_{d,\chi}^{\PP^2}$ does not depend on $\chi$.

One might imagine that a proof of the refined version of the strong rationality conjecture for refined stable pair invariants of $K_{\PP^2}$ could follow from a direct study of an appropriate version of the refined topological vertex and that this would lead to a proof of
Conjecture \ref{conj_refined_joyce_conj_intro} on the 
$\chi$-independence of $\Omega_{d,\chi}^{\PP^2}(y^{\frac{1}{2}})$. We postpone such study to some future work.

Let $\Omega_d^{\PP^2}$ be the common value of the 
$\Omega_{d,\chi}^{\PP^2}$, which makes sense by Theorem 
\ref{thm_joyce_conj}. Let $n_{0,d}^{K_{\PP^2}}$ be the genus-$0$ Gopakumar-Vafa invariants of $K_{\PP^2}$ defined from the genus-$0$ Gromov-Witten invariants 
$N_{0,d}^{K_{\PP^2}}$ by the multicover formula
\[
N_{0,d}^{K_{\PP^2}}=\sum_{\ell|d} \frac{1}{\ell^3}
n_{0,\frac{d}{\ell}}^{K_{\PP^2}}\,.
\]

\begin{prop} \label{prop_katz_conj}
For every $d \in \Z_{\geqslant 1}$, we have 
\[ \Omega_d^{\PP^2}=n_{0,d}^{K_{\PP^2}}\,.\]
\end{prop}

\begin{proof}
More generally, for any Calabi-Yau 3-fold $X$, the Gopakumar-Vafa invariants $n_{r,\beta}^X$ of $X$ for $r\in \Z_{\geq 0}$ and $\beta \in H_2(X,\Z)$ are defined from the genus $g$ Gromov-Witten invariants $N_{g,\beta}^X$ by the formula
\[ \sum_{g \geq 0} \sum_{\beta \neq 0} N_{g,\beta}^X u^{2g-2} Q^\beta = \sum_{r \geq 0} \sum_{\beta \neq 0} n_{r,\beta}^X 
\sum_{k \geq 1} \frac{1}{k} \frac{Q^{k \beta}}{\left( 2 \sin (ku/2) \right)^{2-2g}}\]

On the other hand, by \cite[\S 3]{MR2545686}, we have BPS invariants $n_{r,\beta}^{X,PT}$ for $r \in \Z$ and $\beta \in H_2(X,\Z)$, such that for fixed $\beta$ we have $n_{r,\beta}^{X,PT}=0$ for $r>>0$ and $r<<0$, and which are defined in terms of the stable pairs Pandharipande-Thomas invariants $P_{n,\beta}^X$ by the formula
\[ \log(1+\sum_{n \in \Z} \sum_{\beta \neq 0}P_{n,\beta}^X q^n Q^\beta)= \sum_{r \in \Z} \sum_{\beta \neq 0} n_{r,\beta}^{X,PT} \sum_{k \geq 1} \frac{(-1)^{g-1}}{k}((-q)^k-2+(-q)^{-k})^{g-1} Q^{k \beta} \,.\]

The Gromov-Witten/stable pairs correspondence conjectures the equality
\[ \sum_{g \geq 0} \sum_{\beta \neq 0} N_{g,\beta}^X u^{2g-2} Q^\beta = \log(1+\sum_{n \in \Z} \sum_{\beta \neq 0}P_{n,\beta}^X q^n Q^\beta) \]
after the change of variables $q=-e^{iu}$. Moreover, according to the strong rationality conjecture for stable pairs \cite[Conjecture 3.14]{MR2545686}, one should have $n_{r,\beta}^{X,PT}=0$ for $r<0$. Hence, these two conjectures combined predict that 
\[ n_{r,\beta}^{X}=n_{r,\beta}^{X,PT}\]
for all $r$ and $\beta$.

For $X=K_{\PP^2}$, the Gromov-Witten/stable pairs correspondence follows from the combination of the Gromov--Witten/ideal sheaves correspondence proved for local toric surfaces in \cite{MR2264664} with the ideal sheaves/stable pairs
correspondence \cite{MR2813335}\cite[\S 5]{MR2746343}.
As reviewed in the proof of Proposition \ref{prop_joyce_song_conjecture}, the strong rationality conjecture for stable pairs is also proved for $K_{\PP^2}$. Hence, we deduce that
\[ n_{r,d}^{K_{\PP^2}}= n_{r,d}^{K_{\PP^2},PT} \]
for all $r$ and $d$.
Finally, by
\cite[Theorem 6.4]{MR2892766}, the strong rationality conjecture also implies the equality
\[ \Omega_d^{\PP^2} = n_{0,d}^{K_{\PP^2},PT} \,.\]
Therefore, we have 
\[\Omega_d^{\PP^2} = n_{0,d}^{K_{\PP^2}} \] 
and this concludes the proof of Proposition \ref{prop_katz_conj}.
\end{proof}

\section{Comparison of the scattering diagrams of \cite{bousseau2019scattering} and \cite{cps}}
\label{section_comparison_cps_wall_structure}
In this section we prove Proposition \ref{prop_comparing_scatterings} comparing the scattering diagram $S(\fD^\iin_{\cl^+})$ introduced in \cite{bousseau2019scattering}
with a specific consistent structure 
$\mathcal{S}$, in the sense of \cite{MR2846484}, introduced in\cite[Example 2.4]{cps}. We use freely the notions and notation introduced in 
\cite{bousseau2019scattering}.

\subsection{The integral affine manifold of \cite{cps}}
\label{section_cps}
In \cite[Example 2.4]{cps}, an integral 
affine manifold with singularities $B$, an integral polyhedral decomposition 
$\mathscr{P}$ of $B$, and a multivalued 
$\mathscr{P}$-piecewise affine function 
$\varphi_{CPS}$ on $B$ are introduced.
One can describe $B$ as the complement in 
$\R^2$ of the three cones 
\[ (1/2,1/2)+\R_{> 0} (1,0) \times \R_{> 0} (0,1)\,,\]
\[(-1/2,0)+\R_{> 0} (0,1)\times \R_{> 0} (-1,-1)\,,\]
\[(0,-1/2)+\R_{> 0} (-1,-1) \times \R_{> 0} (1,0)\,,\] 
with:
\begin{enumerate}
\item Identification of $(1/2,1/2)+\R_{\geqslant 0} (1,0)$
with $(1/2,1/2)+\R_{\geqslant 0} (0,1)$, and gluing of the integral affine structure from $(1/2,1/2)+\R_{\geqslant 0} (1,0)$ 
to $(1/2,1/2)+\R_{\geqslant 0}(0,1)$ by \[ \begin{pmatrix}
 0 & -1 \\
 1 & 2
\end{pmatrix} \in SL(2,\Z)\,.\]
\item Identification of $
(-1/2,0)+\R_{\geqslant 0}(1,0)$
with $(-1/2,0)+\R_{\geqslant 0} (-1,-1)$, and gluing of the integral affine structure from $(-1/2,0)+\R_{\geqslant 0}(1,0)$
to $(-1/2,0)+\R_{\geqslant 0} (-1,-1)$ by 
\[ \begin{pmatrix}
 3 & -1 \\
 4 & -1
\end{pmatrix} \in SL(2,\Z)\,.\]
\item 
Identification of $
(0,-1/2)+\R_{\geqslant 0}(-1,-1)$
with $(0,-1/2)+\R_{\geqslant 0} (1,0)$
and gluing of the integral affine structure from $
(0,-1/2)+\R_{\geqslant 0}(-1,-1)$
to $(0,-1/2)+\R_{\geqslant 0} (1,0)$
by \[ \begin{pmatrix}
 3 & -4 \\
 1 & -1
\end{pmatrix} \in SL(2,\Z)\,.\]
\end{enumerate}
The integral affine structure is smooth on the complement $B_0$ of the three singular points $(1/2,1/2)$, $(-1/2,0)$, $(0,-1/2)$.
Each singularity is simple in the sense that the monodromy around it is conjugated with \[\begin{pmatrix}
 1 & 1 \\
 0 & 1
\end{pmatrix}\] 
in $GL(2,\Z)$.  More precisely,
the monodromy takes this form when expressed in the basis consisting of two primitive vectors, one in the
direction of the monodromy-invariant line, and the other in the direction of
one of the faces of the cones that are removed.

The total monodromy around an anticlokwise loop encircling the three singularities is \[
\begin{pmatrix}
 3& -4 \\
 1 & -1
\end{pmatrix}
\begin{pmatrix}
 3& -1 \\
 4 & -1
\end{pmatrix}
\begin{pmatrix}
 0 & -1 \\
 1 & 2
\end{pmatrix}
= \begin{pmatrix}
 1& 9 \\
 0 & 1
\end{pmatrix}\,.
\]
The edges of the integral polyhedral decomposition $\mathscr{P}$ are:
\begin{enumerate}
\item The three unbounded edges:
 \[(1,0)+\R_{\geqslant 0}(1,0)\,, (0,1)+\R_{\geqslant 0}(0,1)\,, (-1,-1)+\R_{\geqslant 0} (-1,-1)\,.\]
\item The three line segments:
\[ [(1,0),(0,1)]\,, [(0,1),(-1,-1)]\,,
[(-1,-1),(1,0)]\,.\]
\end{enumerate}
The unique bounded two-dimensional face of $\mathscr{P}$ is the triangle $\bar{P}_0$ 
delimited by the three line segments
$[(1,0),(0,1)]$, $[(0,1),(-1,-1)]$,
$[(-1,-1),(1,0)]$. We denote by $\bar{P}_{(1/2,1/2)}$ (resp.\ $\bar{P}_{(-1/2,0)}$, $\bar{P}_{(0,-1/2)}$) the unbounded two-dimensional face of $\mathscr{P}$ 
containing $(1/2,1/2)$ (resp.\ 
$(-1/2,0)$, $(0,-1/2)$) in its boundary.
We denote by
$P_0$, $P_{(1/2,1/2)}$, $P_{(-1/2,0)}$,
$P_{(0,-1/2)}$ the interior of 
$\bar{P}_0$, $\bar{P}_{(1/2,1/2)}$, $\bar{P}_{(-1/2,0)}$,
$\bar{P}_{(0,-1/2)}$ respectively.
A representative of $\varphi_{CPS}$ can be defined by\footnote{Recall, see footnote in Definition 1.2.4 of \cite{bousseau2019scattering}, that there is a sign difference in our conventions for the function $\varphi$ with respect to \cite{MR2846484, MR2722115, cps, gabele2019tropical}.}
\begin{enumerate}
\item $\varphi_{CPS}=0$ in restriction to $\bar{P}_0$.
\item $\varphi_{CPS}=-(x+y-1)$ in restriction to
$\bar{P}_{(1/2,1/2)}$.
\item $\varphi_{CPS}=-(-2x+y-1)$ in restriction to $\bar{P}_{(-1/2,0)}$.
\item $\varphi_{CPS}=-(x-2y-1)$ in restriction to 
$\bar{P}_{(0,-1/2)}$.
\end{enumerate}
We fix this representative of $\varphi_{CPS}$ in what follows.

The triple $(B,\mathscr{P},\varphi_{CPS})$ naturally appears as tropicalization 
of a degeneration of the pair $(\PP^2,E)$, see \cite[\S 1]{gabele2019tropical}
and \cref{section_vertical_gw_refined}.

\begin{figure}[h!]
\centering
\setlength{\unitlength}{1.2cm}
\begin{picture}(10,5)
\thicklines
\put(4,1.5){\circle*{0.1}}
\put(5,3.5){\circle*{0.1}}
\put(6,2.5){\circle*{0.1}}
\put(5.5,3){\circle*{0.1}}
\put(5.5,3){\line(1,0){2.5}}
\put(5.5,3){\line(0,1){2}}
\put(5,2){\circle*{0.1}}
\put(5,2){\line(1,0){3}}
\put(5,2){\line(-1,-1){2}}
\put(4.5,2.5){\circle*{0.1}}
\put(4.5,2.5){\line(-1,-1){2}}
\put(4.5,2.5){\line(0,1){2.5}}
\put(4,1.5){\line(-1,-1){1.5}}
\put(5,3.5){\line(0,1){1.5}}
\put(6,2.5){\line(1,0){2}}
\put(4,1.5){\line(2,1){2}}
\put(4,1.5){\line(1,2){1}}
\put(5,3.5){\line(1,-1){1}}
\put(5,2.5){$P_0$}
\end{picture}
\caption{$(B, \mathscr{P})$.}
\end{figure}

\subsection{Comparison with the integral affine manifold of \cite{bousseau2019scattering}}
In \cite{bousseau2019scattering}, we defined a scattering diagram $S(\fD^\iin_{\cl^+})$
living on 
\[ U \coloneqq \{ (x,y) \in \R^2|\, y>-\frac{x^2}{2}\}\,.\]
We denote 
\[ \bar{U} \coloneqq \{ (x,y) \in \R^2|\, y\geqslant -\frac{x^2}{2}\}\,.\]

We will compare $(B,\mathscr{P},\varphi_{CPS})$ of \cite{cps} with a
polygonal approximation $\bar{C}$ of $\bar{U}$.

\begin{defn} \label{def_C_n}
For every $n \in \Z$, we denote by $\bar{C}_n$ the band in $\R^2$ delimited by the line segment $[(n-\frac{1}{2},-\frac{n(n-1)}{2}),(n+\frac{1}{2},-\frac{n(n+1)}{2})]$, and the half-lines 
\[(n-\frac{1}{2},-\frac{n(n-1)}{2})+\R_{\geqslant 0}(0,1)\]
and \[(n+\frac{1}{2},-\frac{n(n+1)}{2})+\R_{\geqslant 0}(0,1)\,.\]
For every $n \in \Z$, we denote by $C_n$ the interior of $\bar{C}_n$. Finally, we define
\[ \bar{C} \coloneq \bigcup_{n\in \Z} \bar{C_n}\,.\]
\end{defn}

We have $\bar{C} \subset \bar{U}$. Indeed, $\bar{C}$ is the subset of $\bar{U}$ delimited below by the sequence of line segments
$\bigcup_{n \in \Z}( |\fd_n^+| \cup |\fd_n^-|)$, see \cite[Definition 1.3.4]{bousseau2019scattering}. We endow
$\bar{C}$ with the integral affine structure for which the lattice of integral tangent vectors is
\[ \Z(1,0) +\Z(0,\frac{1}{3})\]
at every point. 
We denote by
$\mathscr{P}_{\bar{C}}$ the integral polyhedral decomposition of $\bar{C}$ whose two-dimensional faces are given by the
$C_n$, $n \in \Z$.
We denote by $\varphi_{\bar{C}}$
the $\mathscr{P}_{\bar{C}}$-piecewise 
affine function on $\bar{C}$ whose restriction to $\bar{C}_n$ is given by 
\[ \varphi_{\bar{C}}=-nx-y+\frac{n^2}{2}\,,\] for every $n \in \Z$.

\begin{figure}[h!]
\centering
\setlength{\unitlength}{1.2cm}
\begin{picture}(10,5)
\thicklines
\put(2,1){\line(1,1){2}}
\put(2,1){\line(0,1){4}}
\put(4,3){\line(0,1){2}}
\put(4,3){\line(1,0){2}}
\put(6,3){\line(1,-1){2}}
\put(6,3){\line(0,1){2}}
\put(8,1){\line(0,1){4}}
\put(3,4){$C_{-1}$}
\put(5,4){$C_0$}
\put(7,4){$C_1$}
\end{picture}
\caption{$(\bar{C},\mathscr{P}_{\bar{C}})$.}
\end{figure}

\begin{lem}\label{lem_universal_cover}
The universal cover of $B-P_0$ is isomorphic to $\bar{C}$ as an integral affine manifold. Under this isomorphism, the lift to the universal cover of the integral polyhedral decomposition $\mathscr{P}$ restricted to $B-P_0$ is the integral polyhedral decomposition $\mathscr{P}_{\bar{C}}$ of $\bar{C}$, and the lift to the universal cover of $\varphi_{CPS}$ restricted to $B-P_0$ is 
$\varphi_{\bar{C}}$.
\end{lem}

\begin{proof}
This follows directly by elementary computations from the explicit descriptions of $(B,\mathscr{P},\varphi_{CPS})$ and 
$(\bar{C},\mathscr{P}_{\bar{C}},\varphi_{\bar{C}})$ given above.
\end{proof}

\subsection{Comparison of the scattering diagrams}

Applying the general algorithm of
\cite{MR2846484, MR2722115},
a consistent structure $\mathcal{S}$ on 
\[(B,\mathscr{P},\varphi_{CPS})\] 
is defined in 
\cite{cps}. It is proved in 
\cite[\S 5]{gabele2019tropical} that $\mathcal{S}$ is entirely supported on $B-P_0$, and so can be lifted to its universal cover. 
By Lemma \ref{lem_universal_cover}, we obtain a consistent structure $\tilde{\mathcal{S}}$ on $(\bar{C}, \mathscr{P}_{\bar{C}},\varphi_{\bar{C}})$. 
According to \cite[Definition 5.10]{gabele2019tropical}
the initial rays of
$\tilde{\mathcal{S}}$ come out from the lifts of the singularities of the affine structure on $B$, which are exactly the points $s_n$,
$n \in \Z$, of \cite[\S 1.3]{bousseau2019scattering}. The directions of these initial rays are the lifts of the monodromy invariant directions around the singularities, which are exactly the directions $m_n^-$ and $m_n^+$, $n \in \Z$, of \cite[\S 1.3]{bousseau2019scattering}.
Finally, from 
\[ \exp \left( \sum_{\ell \geqslant 1}
\ell \frac{(-1)^{\ell-1}}{\ell^2}z^{\ell m}
\right)=1+z^m\,,\]
we obtain that the automorphisms attached to the initial rays of $\tilde{\mathcal{S}}$ coincide with the automorphisms attached to the initial rays of the scattering diagram $\fD^\iin_{\cl^+}$
defined in \cite[\S 1.4]{bousseau2019scattering}.
By uniqueness of the consistent completion, we obtain the equality
\[ \tilde{\mathcal{S}}=S(\fD^\iin_{\cl^+})\,,\]
where $S(\fD^\iin_{\cl^+})$ is defined in
\cite{bousseau2019scattering} as the consistent completion of $\fD^\iin_{\cl^+}$.

More precisely, there is additional point to discuss in order to really make sense of this equality. In the algorithm of
\cite{MR2846484} constructing $\tilde{\mathcal{S}}$ from its initial data, the local steps are local scattering diagrams in the sense of \cite[\S 1.1]{bousseau2019scattering} if, at a point $\sigma$, one uses as function $\varphi \colon M \rightarrow \R$ the function $d_\sigma \varphi_{\bar{C}}$,
differential of $\varphi_{\bar{C}}$ at the point $\sigma$, that is, $(a,b) \mapsto -an-b$ if $\sigma \in 
C_n$ (and $(a,b) \mapsto \min(-an-b,-a(n+1)-b)$ if $\sigma$ belongs to the intersection of $C_n$ and $C_{n+1}$).
On the other hand, the scattering diagrams of Section 1.2 of
\cite{bousseau2019scattering} are defined as having local scattering diagrams at a point $\sigma$ defined using the function $\varphi_\sigma \colon (a,b) \mapsto -ax-b$ if $\sigma=(x,y)$. 
Therefore, $\tilde{\mathcal{S}}$ is a priori not a scattering diagram in the sense of \cite[\S 1.2]{bousseau2019scattering} and $S(\fD_{\cl^+}^\iin)$ is a priori not a consistent structure in the sense of \cite{MR2846484}. 

Nevertheless, the only practical role of the functions 
$d_\sigma \varphi_{\bar{C}}$ and $\varphi_\sigma$ is to act as a regulator in computations of the local scattering diagrams (through the insertion of $t^{\varphi_\sigma(m)}$ and the expansion according to $t$-powers). As the collection of functions $\sigma \mapsto d_\sigma \varphi_{\bar{C}}$ is a piecewise constant approximation of the continuous family 
$\sigma \mapsto \varphi_\sigma$ (we have indeed $n-\frac{1}{2}\leqslant x \leqslant n+\frac{1}{2}$ if $n \in \bar{C}_n$), it follows that if the initial rays of $\tilde{\mathcal{S}}$ and $S(\fD^\iin_{\cl^+})$ coincide, then all the rays of 
$\tilde{\mathcal{S}}$ and $S(\fD^\iin_{\cl^+})$ coincide: exactly the same computations are done at each local scattering, only the orders in $t$ at which the computations are done are slightly different due to the small difference between $x$ and $n$ if $x \in \bar{C}_n$.
Therefore, we conclude that $\tilde{\mathcal{S}}$ and $S(\fD^\iin_{\cl^+})$ have the same rays, so $\tilde{\mathcal{S}}$ can be viewed as a scattering diagram in the sense of  \cite[\S 1.2]{bousseau2019scattering}, $S(\fD^\iin_{\cl^+})$ can be viewed as a consistent structure in the sense of \cite{MR2846484}, and the equality 
$\tilde{\mathcal{S}}=S(\fD^{\iin}_{\cl^+})$
makes sense either way.

The reason why we defined scattering diagrams in \cite[\S 1.2]{bousseau2019scattering} using the real family of functions $\sigma \mapsto \varphi_\sigma$ and not the piecewise constant family $\sigma \mapsto d_{\sigma} \varphi_{\bar{C}}$, 
is that only the former is natural from the point of view of the scattering diagram $\fD^{\PP^2}$ defined  in terms of stability conditions
in 
\cite[\S 2]{bousseau2019scattering}.
Using the notation of \cite{bousseau2019scattering}, if
$\gamma=(r,d,\chi)$ is the class of a  
$\sigma$-semistable object, then, by definition of a stability condition, we know that $\Ima Z_\gamma^\sigma>0$, which is equivalent to $-ax-b>0$ if $m_\gamma=(r,-d)=(a,b)$ and $\sigma=(x,y)$. It is not clear a priori why, for $n \in \Z$ such that $n-\frac{1}{2}\leqslant x \leqslant n+\frac{1}{2}$, one should have $-an-b>0$ (which is stronger that $-ax-b>0$ if $x<n$ and $a<0$ or if $x>n$ and $a>0$). At the end of the day, once we know by  \cite[Theorem 5.2]{bousseau2019scattering} that $\fD^{\PP^2}
=S(\fD^\iin)$, and that $S(\fD^\iin)=\tilde{\mathcal{S}}$, we obtain that this is indeed true, but it was not obvious a priori. A short way to phrase the issue is that, from the point of view of stability conditions, the lines $x=n$ or the polyhedral decomposition $\mathscr{P}_{\bar{C}}$ have no particular signification, whereas they have a clear meaning from the point of view of \cite{MR2846484} ($\mathscr{P}$ appears as the dual intersection complex of the special fiber of the degeneration of $(\PP^2,E)$ constructed in \cite[\S 1]{gabele2019tropical}).

The statement below summarizes the above discussion.

\begin{prop} \label{prop_comparing_scatterings}
The lift $\tilde{\mathcal{S}}$ to the universal cover of 
$B-P_0$ of the consistent structure $\mathcal{S}$ defined in \cite{cps}
for $(\PP^2,E)$ coincides with the scattering diagram 
$S(\fD^{\iin}_{\cl^+})$ defined in \cite{bousseau2019scattering}
and viewed as living in $\bar{C}$.
\end{prop}

\subsection{Stringy K\"ahler moduli space} \label{section_stringy_kahler}
This section is not logically necessary to follow the rest of the paper until the 
motivational heuristic explanation described in \S \ref{section_heuristic_explanation}.

The consistent structure
$\tilde{\mathcal{S}}$ being defined as lift to the universal cover of $\mathcal{S}$ on $B-P_0$, there is a natural action of $\pi_1(B-P_0) \simeq \Z$ on $\tilde{\mathcal{S}}$. On the other hand, we have a symmetry $\psi(1)$ of $S(\fD^\iin_{\cl^+})$ (see \cite[\S 1.5]{bousseau2019scattering}). Under the identification $\tilde{\mathcal{S}}
=S(\fD^\iin_{\cl^+})$, the (correctly chosen) generator of $\pi_1(B-P_0) \simeq \Z$ acts as $\psi(1)^3$. The action of $\psi(1)$ itself is related to a $\Z/3$ symmetry of $\mathcal{S}$.

By Lemma \ref{lem_universal_cover}, we can identify the universal cover of $B-\bar{P_0}$ with a subset of $U$, and so, using  
\cite[Proposition 2.14]{bousseau2019scattering}, we can view the universal cover of $B-\bar{P_0}$ as a subset of $\Stab(\PP^2)$. 
In fact, stability conditions on 
$\D^b(\PP^2)$ corresponding to points of $U$ can also be viewed as stability conditions on the derived category $\D^b_0(K_{\PP^2})$ of coherent sheaves on $K_{\PP^2}$
set-theoretically supported on the zero-section \cite{MR2852118}. Thus, we can view the universal cover of $B-\bar{P}_0$
as embedded in the space $\Stab_0(K_{\PP^2})$ of Bridgeland stability conditions on 
$\D^b_0(K_{\PP^2})$. 
It is natural to ask if the universal cover of $B_0$, and not just of 
$B_0-\bar{P}_0$, can be embedded in $\Stab_0(K_{\PP^2})$ in a similar way. 
We will see below, using \cite[\S 9]{MR2852118}, that there exists an embedding
of $U$ in $\Stab_0(K_{\PP^2})$, which is different from the one given by
\cite[Proposition 2.14]{bousseau2019scattering}, and such that the corresponding embedding 
of the universal cover of $B-\bar{P}_0$ in 
$\Stab(K_{\PP^2})$ naturally extends to an embedding of the universal cover of
$B_0$ in $\Stab(K_{\PP^2})$. 

Following \cite[\S 9]{MR2852118}, we denote
by $\cM_{\Gamma_1(3)}=(\C-\mu_3)/\mu_3$ the moduli space of elliptic curves with 
$\Gamma_1(3)$ level structure. We can view 
$\cM_{\Gamma_1(3)}$ as 
$\PP^1$ with the points $z=0$ and $z=-\frac{1}{27}$ removed, and a stacky 
$\Z/3$ point at $z=\infty$. The fundamental group of $\cM_{\Gamma_1(3)}$ is 
$\Gamma_1(3)$. We denote by
$\widetilde{\cM}_{\Gamma_1(3)}$ the universal cover of $\cM_{\Gamma_1(3)}$, and $\widetilde{\cM}^{(3)}_{\Gamma_1(3)}=\C-\mu_3$, $3$ to $1$ cover of $\cM_{\Gamma_1(3)}$. 
\cite[Equation (22)]{MR2852118} defines holomorphic functions $a(z)$ and $b(z)$ on
$\widetilde{\cM}_{\Gamma_1(3)}$, which are periods relevant for mirror symmetry for 
$K_{\PP^2}$. Using $\Rea a(z)$ and 
$3 \Rea b(z)$ as local integral affine coordinates, we view
$\widetilde{\cM}^{(3)}_{\Gamma_1(3)}$ 
as an integral affine manifold.

\begin{prop} \label{prop_integral_affine}
The multivalued map 
\begin{align*} \widetilde{\cM}^{(3)}_{\Gamma_1(3)}& \longrightarrow \R^2 \\
z &\longmapsto (x,y) :=(\Rea a(z), 3\Rea b(z))\end{align*}
induces an isomorphism of integral affine manifolds 
$\widetilde{\cM}_{\Gamma_1(3)}^{(3)} \simeq B_0$.
\end{prop}

\begin{proof}
We denote by $C \subset \R^2$ the interior of 
$\bar{C}$.
We use Lemma 
\ref{lem_universal_cover} to identify the universal cover of $B-\bar{P}_0$ with 
$C$ and we choose a fundamental domain $D$ of 
$B-\bar{P}_0$ in $C$. As reviewed in \cite[\S 9]{MR2852118}, the monodromy transformation of
$(a(z), b(z))$ around $z=0$ is given by $a(z) \mapsto a(z)+1$, $b(z) \mapsto b(z)-a(z)-\frac{1}{2}$.
On the other hand, the map $(x,y) \mapsto (x+1,y-x-\frac{1}2)$ maps the point
$(n-\frac{1}{2},-\frac{n(n-1)}{2})$
to $(n+\frac{1}{2},-\frac{n(n+1)}{2})$ for every $n$,
and in fact induces
a bijection from 
$\bar{C}_n$ to $\Bar{C}_{n+1}$ for every $n$ (see Definition \ref{def_C_n}).
It follows that
\[ \{z \in \widetilde{\cM}_{\Gamma_1(3)}|
(\Rea a(z), 3 \Rea b(z)) \in D \} \rightarrow D \]
descends to an isomorphism 
\[\{z \in \widetilde{\cM}^{(3)}_{\Gamma_1(3)}|
(\Rea a(z), 3 \Rea b(z))\in D \} \simeq  D\,.\]
Using the determinations of $a(z)$ and $b(z)$ defined by the explicit power series expansion around $z=0$ in \cite[p43]{MR2852118}, we see that $\Rea a(-\frac{1}{27})=0$ by a direct inspection of the power series, and we have $b(-\frac{1}{27})=0$ by \cite[p45]{MR2852118}. Hence, 
\[ z \mapsto (\Rea a(z), 3 \Rea b(z))\] 
maps the three points
$z=-\frac{1}{27}$ on the boundary 
$\widetilde{\cM}^{(3)}_{\Gamma_1(3)}$ to exactly the three points on the boundary of $D$ corresponding to the three singular points of the integral affine structure on $B$.

It remains to check that the monodromies of $\widetilde{\cM}^{(3)}_{\Gamma_1(3)}$ and $B_0$ match.
We have already done it around $z=0$, and so it remains to do it around the singular point $z=-\frac{1}{27}$. 
As reviewed in \cite[\S 9]{MR2852118}, the monodromy  around $z=-\frac{1}{27}$ in $\widetilde{\cM}^{(3)}_{\Gamma_1(3)}$
is categorically the spherical twist around $\mathcal{O}_{\PP^2}$, and so is given by $a(z) \mapsto a(z)+3b(z)$ and $b(z) \mapsto b(z)$. On the other hand, the monodromy of $B_0$ around the point $0 \in \bar{C}$ is $(x,y)\mapsto (x+y,y)$. 
These two monodromy coincide and this concludes the proof.
\end{proof}

\cite[Theorem 9.3]{MR2852118} gives an embedding of 
$\widetilde{M}_{\Gamma_1(3)}$, and so, by Proposition \ref{prop_integral_affine}, of the universal cover of $B_0$, in $\Stab_0(K_{\PP^2})$. 

The corresponding embedding of $U$ is different from the one of 
\cite[Proposition 2.14]{bousseau2019scattering}: the real part of the central charges is the same but the imaginary part is different.
As the scattering diagram $\fD^{\PP^2}_{u,v}$ of 
\cite{bousseau2019scattering} 
is defined in terms of the vanishing of the real part of the central charge, it is very likely that all our arguments, and in particular \cite[Theorem 5.9]{bousseau2019scattering}, should remain valid using the embedding of 
$U$ in $\Stab_0(K_{\PP^2})$ given by \cite[Theorem 9.3]{MR2852118} instead of the one given by
\cite[Proposition 2.14]{bousseau2019scattering}.

In fact, from the point of view of mirror symmetry, the embedding of
\cite[Theorem 9.3]{MR2852118} is the `right one': it is the one extending to the universal cover of the global stringy K\"ahler moduli space $\cM_{\Gamma_1(3)}$ of $K_{\PP^2}$, and so the one relevant for physics (see \cite[\S 0.4.3]{bousseau2019scattering}).
In particular, our heuristic explanation for 
\cite[Theorem 5.9]{bousseau2019scattering} given in \cref{section_heuristic_explanation} really applies to the embedding of  \cite[Theorem 9.3]{MR2852118}. The use of the embedding of 
\cite[Proposition 2.14]{bousseau2019scattering} in the \cite{bousseau2019scattering} is motivated by the fact that it is related to the standard upper half-plane already 
studied in the literature, and by the related fact that it is easier to deal with polynomials than with transcendental functions.

\section{Unrefined sheaf counting and genus-$0$ Gromov-Witten theory of $(\PP^2,E)$}
\label{section_takahashi_proof}

In this section, we prove Theorem \ref{thm_log_bps_local_bps_intro} establishing 
a connection between the signed intersection Euler characteristics \[ \Omega_{d,\chi}^{\PP^2} \]
of moduli spaces of Gieseker semistable one-dimensional sheaves on $\PP^2$ and the relative BPS invariants 
\[ \Omega_{d,k}^{\PP^2/E}\] 
defined via genus-$0$ Gromov-Witten theory of $(\PP^2,E)$
in \cref{section_takahashi_intro}. According to Proposition 
\ref{prop_comparing_scatterings}, we have the equality of scattering diagrams 
$S(\fD^{\iin}_{\cl^+})=\tilde{\mathcal{S}}$. By the main result of \cite{bousseau2019scattering}
reviewed in \cref{section_vertical_sheaves} the vertical rays of 
$S(\fD^{\iin}_{\cl^+})$ compute the sheaf invariants $\Omega_{d,\chi}^{\PP^2}$.
On the other hand, by the main result of \cite{gabele2019tropical} 
reviewed in \cref{section_vertical_gw},
the vertical rays of $\tilde{\mathcal{S}}$ compute the genus-$0$
Gromov-Witten theory of $(\PP^2,E)$. Combining this two facts, we prove 
Theorem \ref{thm_log_bps_local_bps_intro} in \cref{section_proof_thm_log_bps_local_bps}. Finally, we explain how to deduce
Theorem \ref{thm_local_relative_intro}
in \cref{section_thm_local_relative_intro}.

We use freely the notions and notation introduced in 
\cite{bousseau2019scattering}. As we know that $S(\fD^{\iin}_{\cl^+})=\tilde{\mathcal{S}}$
by Proposition 
\ref{prop_comparing_scatterings}, we will always use the notation 
$S(\fD^{\iin}_{\cl^+})$ and we will no longer use the notation 
$\tilde{S}$.

\subsection{Scattering and unrefined sheaf-counting on $\PP^2$}
\label{section_vertical_sheaves}

The scattering diagram $S(\fD^\iin_{\cl^+})$ is defined 
in \cite{bousseau2019scattering} as the consistent completion of an explicit initial 
scattering diagram $\fD^\iin_{\cl^+}$ (see \cite[Definition 1.4.10]{bousseau2019scattering}). 

On the other hand, viewing $U$ as a space of Bridgeland stability conditions
on the bounded derived category $\D^b(\PP^2)$ of coherent sheaves on $\PP^2$, we defined in \cite[Definition 2.5.8]{bousseau2019scattering} a scattering diagram
$\fD^{\PP^2}_{\cl^+}$ using intersection Euler characteristics of 
moduli spaces of $\sigma$-semistable objects in $\D^b(\PP^2)$ for $\sigma \in U$.

One of the main results of \cite{bousseau2019scattering} is that the algorithmically defined scattering diagram $S(\fD^\iin_{\cl^+})$ and the geometrically defined 
scattering diagram $\fD^{\PP^2}_{\cl^+}$ coincide: by 
\cite[Theorem 5.2.3]{bousseau2019scattering}, we have
\[ S(\fD^\iin_{\cl^+}) = \fD^{\PP^2}_{\cl^+} \,.\]

We focus on the unbounded vertical rays of $S(\fD^\iin_{\cl^+})$. We show below that they are naturally indexed by $d \in \Z_{\geq 1}$ and $\chi \in \Z$, and that the function 
$H_{\fd_{d,\chi}}$ attached to the ray $\fd_{d,\chi}$ are expressed in terms of the intersection Euler characteristics of the moduli spaces of Gieseker semistable one-dimensional sheaves on $\PP^2$, supported on curves of degree $d'$ and with
holomorphic Euler characteristic $\chi'$, for $(d',\chi')$ dividing $(d,\chi)$.

\begin{lem} \label{lem_vertical_rays}
Let $d$ be a positive integer. For every unbounded ray $\fd$
of $S(\fD^\iin_{\cl^+})$ of class $m_\fd=(0,-d)$, there exists a unique $\chi \in \Z$ such that $|\fd|$ is contained in the vertical line of equation 
\[ x=\frac{1}{d}\left(\chi-\frac{3}{2}\right)\,.\]

Moreover, for every $d \in \Z_{\geqslant 1}$ and $\chi \in \Z$, there exists a unique unbounded ray in $S(\fD^\iin_{\cl^+})$, that we denote by
$\fd_{d,\chi}$, of class $m_{\fd_{d,\chi}}
=(0,-d)$ and contained in the vertical line of equation $ x=\frac{1}{d}\left(\chi-\frac{3}{2}\right)$.
\end{lem}

\begin{proof}
By \cite[Definition 2.5.8]{bousseau2019scattering}, every ray of $\fD^{\PP^2}_{\cl^+}$
is contained in a line of equation $ry+dx+r+\frac{3}{2}d-\chi=0$
for some $\gamma =(r,d,\chi)\in \Gamma$. If $\fd$ is of class 
$m_\fd=(0,-d)$, it means that $r=0$, and so 
$|\fd|$ is contained of the line of equation $dx+\frac{3}{2}d-\chi=0$.

For an unbounded vertical ray, Bridgeland $\sigma$-stability 
for $\sigma \in U$ reduces to Gieseker stability by 
\cite[Lemma 5.1.1]{bousseau2019scattering}.
Therefore, the existence of a ray 
$\fd_{d,\chi}$ as in the statement of 
Lemma \ref{lem_vertical_rays} follows from
the definition 
\cite[Definition 2.5.8]{bousseau2019scattering} of $\fD^{\PP^2}_{\cl^+}$, and from
the fact that, for every $d \in \Z_{>0}$ and 
$\chi \in \Z$, the moduli space 
$M_{d,\chi}$ of Gieseker semistable 
degree-$d$ one-dimensional sheaves on $\PP^2$ of holomorphic Euler characteristic $\chi$ is nonempty, 
and so, by \cite[Corollary 6.1.3]{bousseau2019scattering}, $Ie(M_{d,\chi})
\neq 0$.
\end{proof}


\begin{thm}\label{thm_boussseau}
For every $d \in \Z_{>0}$ and $\chi \in \Z$, the function 
$H_{\fd_{d,\chi}}$ attached to the unbounded vertical ray 
$\fd_{d,\chi}$ in the scattering diagram $S(\fD^\iin_{\cl^+})$ is given by
\[ H_{\fd_{d,\chi}}=
(-1)^{d-1}
\left( \sum_{\substack{\ell \in \Z_{\geqslant 1} \\\ell|(d,\chi)}}
\frac{1}{\ell^2} \Omega_{\frac{d}{\ell},\frac{\chi}{\ell}}^{\PP^2}
\right) z^{(0,-d)} \,, \]
\end{thm}

\begin{proof}
For an unbounded vertical ray, Bridgeland $\sigma$-stability 
for $\sigma \in U$ reduces to Gieseker stability by 
\cite[Lemma 6.1.1]{bousseau2019scattering}.
Therefore, the result follows directly from
the definition 
\cite[Definition 2.5.8]{bousseau2019scattering} of $\fD^{\PP^2}_{\cl^+}$, 
and from the fact that $(-1)^{(\gamma,\gamma)}=(-1)^{d}$ for 
$\gamma=(0,d,\chi)$ (see \cite[Definition 2.1.1]{bousseau2019scattering}).

\end{proof}

\subsection{Scattering and genus-$0$ Gromov-Witten invariants of $(\PP^2,E)$} \label{section_vertical_gw}

In Theorem \ref{thm_gabele} below, we review the main result of 
\cite{gabele2019tropical} expressing the functions $H_{\fd_{d,\chi}}$ attached to the unbounded vertical rays 
$\fd_{d,\chi}$ in $S(\fD^\iin_{\cl^+})$ in terms of the relative Gromov-Witten invariants $N_{0,d}^{\PP^2/E,k}$ defined in \S \ref{section_takahashi_intro}.
We first introduce some notation
which are necessary to state Theorem 
\ref{thm_gabele}.

Recall from \S \ref{section_structure_proof_intro} that
for every $d \in \Z_{>0}$ and $\chi \in \Z$, we define
\[ \ell_{d,\chi} \coloneq \frac{d}{\gcd(d,\chi)} \in  \Z_{>0}\,.\]
For every abelian group $G$ and $x$ an element of $G$ of finite order divisible by $3$, we denote by $d(x)$ the smallest positive integer such that 
$(3d(x))x=0$ in $G$.
For every $\ell \in \Z_{>0}$, we denote by $r_\ell$ the number of elements $x \in \Z/(3\ell)$ such that $d(x)=\ell$.
For every $k, \ell \in \Z_{>0}$, we denote by
$s_{k,\ell}$ the number of $x=(a,b)
\in \Z/(3k) \times \Z/(3k)$
such that 
$d(x)=k$ and $d(a)=\ell$.







\begin{thm}\label{thm_gabele}
For every 
$d \in \Z_{\geqslant 1}$ and $\chi \in \Z$, we have 
\[ H_{\fd_{d,\chi}}
=\left( \sum_{\substack{k \in \Z_{\geqslant 1} \\ \ell_{d,\chi}|k|d}}
\frac{s_{k,\ell_{d,\chi}}}{r_{\ell_{d,\chi}}} N_{0,d}^{\PP^2/E,k}
\right)
z^{(0,-d)} \,.\]
\end{thm}

\begin{proof}
This is the main result of 
\cite{gabele2019tropical}.
More precisely, the main result of 
\cite{gabele2019tropical} is a tropical correspondence theorem computing the Gromov-Witten invariants $N_{0,d}^{\PP^2/E,k}$ in terms of the consistent structure 
$\mathcal{S}$ of
\cite[Example 2.4]{cps}. It remains to use the equality 
$S(\fD^\iin_{\cl^+})=\tilde{\mathcal{S}}$ of 
Proposition \ref{prop_comparing_scatterings}
to conclude.

The precise form of the formula proved in 
\cite{gabele2019tropical} comes from a study of the tropicalization of the torsion points of $E$ into torsion points of the 
`circle at infinity' of the integral affine manifold with singularities $B$ on which $\mathcal{S}$ lives. We refer to \cite[\S 6]{gabele2019tropical} for details. 
Under the equality $S(\fD^\iin_{\cl^+})=\tilde{\mathcal{S}}$, these torsion points on the circle at infinity in $B$ become the intersection of the vertical lines $\fd_{d,\chi}$ with the horizontal line at infinity in $U$. 
\end{proof}

According to Theorem \ref{thm_gabele}, the unbounded vertical rays of the scattering diagram $S(\fD^\iin_{\cl^+})$ compute genus-$0$ Gromov-Witten invariants of $(\PP^2,E)$. 
By Theorem \ref{thm_boussseau}, they also compute intersection Euler characteristics of moduli spaces of Gieseker semistable rank $0$ sheaves on $\PP^2$.

One could ask what happens for unbounded non-vertical rays of the scattering diagram $S(\fD^\iin_{\cl^+})$. According to \cite{bousseau2019scattering}, they compute intersection Euler characteristics of moduli spaces of Gieseker semistable sheaves on $\PP^2$ of positive rank. To find a Gromov-Witten description seems more subtle. Intuitively, unbounded rays of $S(\fD^\iin_{\cl^+})$ should count holomorphic disks in $\PP^2-E$ with boundary on fibers of a Lagrangian torus fibration of 
$\PP^2-E$. Vertical rays correspond to disks with boundary which can be closed  `at infinity' to produce closed holomorphic curves meeting $E$ in a single point, whose counts produce the invariants $N_{0,d}^{\PP^2/E}$. On the other hand, non-vertical rays correspond to disks with boundary which cannot be closed to `infinity' (because not monodromy invariant).
Given a symplectic definition of the counts of such disks (for example following the techniques of \cite{lin2017correspondence,
lin2019correspondence} and using the hyperkäher metric of \cite{collins2019special}), we expect these counts to be computed by $S(\fD^\iin_{\cl^+})$.
The upshot would be a correspondence between counts of open holomorphic curves in 
$\PP^2 \setminus E$ and counts of higher-rank sheaves on $\PP^2$. Staying in algebraic geometry, these counts should be punctured log Gromov--Witten invariants \cite{abramovich2020punctured}
of the special fiber of the degeneration considered in \cite{gabele2019tropical}.

\subsection{Proof of Theorem \ref{thm_log_bps_local_bps_intro}}
\label{section_proof_thm_log_bps_local_bps}

The functions $H_{\fd_{d,\chi}}$ attached to the unbounded vertical rays
$\fd_{d,\chi}$ of the scattering diagram  $S(\fD^\iin_{\cl^+})$ are expressed in terms of the sheaf counting invariants $\Omega_{d',\chi'}^{\PP^2}$ by Theorem 
\ref{thm_boussseau} and in terms of the genus-$0$ Gromov-Witten invariants 
$N_{0,d}^{\PP^2/E,k}$ by Theorem \ref{thm_gabele}. Comparing these two expressions, we obtain the following result.

\begin{thm} \label{thm_sheaf_gw}
For every $d \in \Z_{\geqslant 1}$ and $\chi \in \Z$, we have 
\[ (-1)^{d-1}
 \sum_{\substack{\ell \in \Z_{\geqslant 1} \\\ell|(d,\chi)}}
\frac{1}{\ell^2} \Omega_{\frac{d}{\ell},\frac{\chi}{\ell}}^{\PP^2}
=
 \sum_{\substack{k \in \Z_{\geqslant 1} \\ \ell_{d,\chi}|k|d}}
\frac{s_{k,\ell_{d,\chi}}}{r_{\ell_{d,\chi}}} N_{0,d}^{\PP^2/E,k} \,.\]
\end{thm}

We will show that a simple reformulation of Theorem \ref{thm_sheaf_gw} 
gives Theorem \ref{thm_log_bps_local_bps_intro}.

\begin{lem} \label{lem_rewriting}
For every $d \in \Z_{\geqslant 1}$ and $\chi \in \Z$, we have 
\[ \sum_{\substack{\ell \in \Z_{\geqslant 1}\\
\ell|(d,\chi)}} \frac{1}{\ell^2} \Omega_{\frac{d}{\ell},\frac{\chi}{\ell}}^{\PP^2}
=\sum_{\substack{d'\in \Z_{\geqslant 1}\\
\ell_{d,\chi}|d'|d}} \frac{1}{(d/d')^2}
\Omega^{\PP^2}_{d',\chi \frac{d'}{d}}\,.\]
\end{lem}

\begin{proof}
It is a simple rewriting.
Indeed, we have by definition 
$d=\ell_{d,\chi} \gcd(d,\chi)$, so if $\ell|(d,\chi)$, then 
$\ell|\gcd(d,\chi)$ and so $d' \coloneq \frac{d}{\ell}$ satisfies $\ell_{d,\chi}|d'|d$, and conversely.
\end{proof}

Using Lemma \ref{lem_rewriting}, we rewrite Theorem \ref{thm_sheaf_gw} as
\[ 
 \sum_{\substack{d'\in \Z_{\geqslant 1}\\
\ell_{d,\chi}|d'|d}} \frac{1}{(d/d')^2}
\Omega^{\PP^2}_{d',\chi \frac{d'}{d}}
=
 \sum_{\substack{k \in \Z_{\geqslant 1} \\ \ell_{d,\chi}|k|d}}
\frac{s_{k,\ell_{d,\chi}}}{r_{\ell_{d,\chi}}} (-1)^{d-1} N_{0,d}^{\PP^2/E,k} \,.\]

Recall from \cref{section_takahashi_intro} that the relative BPS invariants 
$\Omega_{d,k}^{\PP^2/E}$ are defined in terms of the genus-$0$
Gromov-Witten invariants 
$N_{0,d}^{\PP^2/E,k}$ by the multicover formula
\[ (-1)^{d-1} N_{0,d}^{\PP^2/E,k}=\sum_{\substack{d' \in \Z_{\geqslant 1}\\ k|d'|d}} \frac{1}{(d/d')^2} \Omega^{\PP^2/E}_{d', k}\,.\]

Therefore, the above rewriting of Theorem \ref{thm_sheaf_gw} becomes
\[ 
 \sum_{\substack{d'\in \Z_{\geqslant 1}\\
\ell_{d,\chi}|d'|d}} \frac{1}{(d/d')^2}
\Omega^{\PP^2}_{d',\chi \frac{d'}{d}}
=\sum_{\substack{k \in \Z_{\geqslant 1} \\ \ell_{d,\chi}|k|d}}
\frac{s_{k,\ell_{d,\chi}}}{r_{\ell_{d,\chi}}} \sum_{\substack{d' \in \Z_{\geqslant 1}\\ k|d'|d}} \frac{1}{(d/d')^2} \Omega^{\PP^2/E}_{d', k} \,.\]
Changing the order of the two sums in the right-hand side, we obtain
\begin{equation} \label{eq_sums}
 \sum_{\substack{d'\in \Z_{\geqslant 1}\\
\ell_{d,\chi}|d'|d}} \frac{1}{(d/d')^2}
\Omega^{\PP^2}_{d',\chi \frac{d'}{d}}
= \sum_{\substack{d'\in \Z_{\geqslant 1}\\
\ell_{d,\chi}|d'|d}} \frac{1}{(d/d')^2}
\sum_{\substack{k \in \Z_{\geqslant 1} \\ \ell_{d,\chi}|k|d'}}
\frac{s_{k,\ell_{d,\chi}}}{r_{\ell_{d,\chi}}} 
\Omega^{\PP^2/E}_{d', k} \,.
\end{equation}
This equality holds for every $d \in \Z_{\geqslant 1}$ and $\chi \in \Z$. 

We prove Theorem \ref{thm_log_bps_local_bps_intro}, that is, 
\begin{equation} \label{eq_thm}
\Omega_{d,\chi}^{\PP^2}=\sum_{\ell_{d,\chi}|k|d}
\frac{s_{k,\ell_{d,\chi}}}{r_{\ell_{d,\chi}}} \Omega_{d,k}^{\PP^2/E}
\end{equation}
for every $d \in \Z_{\geq 1}$, $\chi \in \Z$, by
induction on $\gcd(d,\chi)$. 

If $\gcd(d,\chi)=1$, then $\ell_{d,\chi}=d$ and \eqref{eq_sums} reduces to 
\eqref{eq_thm}.

Let us assume that \eqref{eq_thm} holds for every $(d',\chi')$ with 
$\gcd(d',\chi') < \gcd(d,\chi)$. We can rewrite 
Theorem \ref{thm_log_bps_local_bps_intro} as
\begin{equation} \label{eq_two_sums}
\Omega_{d,\chi}^{\PP^2}+
 \sum_{\substack{d'\in \Z_{\geqslant 1}\\
\ell_{d,\chi}|d'|d \\ d'<d}} \frac{1}{(d/d')^2}
\Omega^{\PP^2}_{d',\chi \frac{d'}{d}}
= 
\sum_{\substack{k \in \Z_{\geqslant 1} \\ \ell_{d,\chi}|k|d}}
\frac{s_{k,\ell_{d,\chi}}}{r_{\ell_{d,\chi}}} 
\Omega^{\PP^2/E}_{d, k} 
+
\sum_{\substack{d'\in \Z_{\geqslant 1}\\
\ell_{d,\chi}|d'|d \\ d'<d}} \frac{1}{(d/d')^2}
\sum_{\substack{k \in \Z_{\geqslant 1} \\ \ell_{d,\chi}|k|d'}}
\frac{s_{k,\ell_{d,\chi}}}{r_{\ell_{d,\chi}}} 
\Omega^{\PP^2/E}_{d', k} \,.
\end{equation}
For every $d'$ with $\ell_{d,\chi}|d'|d$ and $d'<d$, we have $\gcd(d',\chi \frac{d'}{d})=\frac{d'}{d} \gcd(d,\chi) < \gcd(d,\chi)$, and so by the induction
hypothesis, we can apply \eqref{eq_thm} to get 
\[ \Omega_{d',\chi d'/d}^{\PP^2}=\sum_{\ell_{d',\chi d'/d}|k|d}
\frac{s_{k,\ell_{d',\chi d'/d}}}{r_{\ell_{d',\chi d'/d}}} \Omega_{d',k}^{\PP^2/E}\,.\]
But we have 
\[ \ell_{d',\chi \frac{d'}{d}}=\frac{d'}{\gcd(d',\chi \frac{d'}{d})}
= \frac{d}{\gcd(d,\chi)}=\ell_{d,\chi}\,,\]
and so 
\[ \Omega_{d',\chi d'/d}^{\PP^2}=\sum_{\ell_{d,\chi}|k|d}
\frac{s_{k,\ell_{d,\chi}}}{r_{\ell_{d,\chi}}} \Omega_{d',k}^{\PP^2/E}\,.\]
Therefore, the sums over $d'<d$ on each side of \eqref{eq_two_sums} cancel each other, and we obtain \eqref{eq_thm}. This ends the proof of Theorem \ref{thm_log_bps_local_bps_intro}.

\subsection{Proof of Theorem \ref{thm_local_relative_intro}}
\label{section_thm_local_relative_intro}

We explain how Theorem \ref{thm_local_relative_intro} follows 
from Theorem \ref{thm_log_bps_local_bps_intro}
and Theorem \ref{thm_joyce_conj_intro}.

By Theorem \ref{thm_joyce_conj_intro}, $\Omega_{d,\chi}^{\PP^2}$ does not depend on $\chi$. We denote $\Omega_d^{\PP^2}$ the common value of the invariants 
$\Omega_{d,\chi}^{\PP^2}$.
Therefore, we can rewrite 
\ref{thm_log_bps_local_bps_refined_intro}
as 
\[ \Omega_d^{\PP^2}=\sum_{\ell_{d,\chi}|k|d}
\frac{s_{k,\ell_{d,\chi}}}{r_{\ell_{d,\chi}}} \Omega_{d,k}^{\PP^2/E}\,.\]

This formula gives a recursive way to uniquely determine the invariants 
$\Omega_{d,k}^{\PP^2/E}$. 
Therefore, in order to show that 
$\Omega_{d,k}^{\PP^2/E}
=\frac{1}{3d} \Omega_d^{\PP^2}$, it is enough to show that this formula is the solution to the above set of recursive equations. So it is enough to show that 
\[\sum_{\substack{k\in \Z_{\leqslant 1}\\\ell_{d,\chi}|k|d}} 
\frac{s_{k,\ell_{d,\chi}}}{r_{\ell_{d,\chi}}}=3d\,.\]
The sum
\[\sum_{\substack{k\in \Z_{\leqslant 1}\\\ell_{d,\chi}|k|d'}} 
s_{k,\ell_{d,\chi}}\]
is the number of $x=(a,b)\in \Z/(3d)\times \Z/(3d)$
with $d(a)=\ell_{d,\chi}$. 
As
$r_{\ell_{d,\chi}}$ is the number of 
$a\in \Z/(3d)$ such that $d(a)=\ell_{d,\chi}$,
\[
\frac{1}{r_{\ell_{d,\chi}}}\sum_{\substack{k\in \Z_{\leqslant 1}\\\ell_{d,\chi}|k|d}} 
s_{k,\ell_{d,\chi}}\]
is the number of $x=(a,b) \in \Z/(3d)\times
\Z/(3d)$ with a given $a$, which is indeed $3d$. This ends the proof of
Theorem \ref{thm_local_relative_intro}.

\section{Refined sheaf counting and higher-genus Gromov-Witten invariants of $(\PP^2,E)$}
\label{section_higher_genus_refinement}

In this section, we prove Theorem \ref{thm_log_bps_local_bps_refined_intro} establishing 
a connection between the signed 
symmetrized Poincar\'e polynomials 
\[ \Omega_{d,\chi}^{\PP^2}(y^{\frac{1}{2}})\]
of moduli spaces of Gieseker semistable one-dimensional sheaves on $\PP^2$ and the relative BPS invariants 
\[ \Omega_{d,k}^{\PP^2/E}(\hbar)\] 
defined via 
higher-genus Gromov-Witten theory of $(\PP^2,E)$
in \cref{section_higher_genus_intro}. 
We follow the logic of the proof of Theorem 
\ref{thm_log_bps_local_bps_intro} given in 
\cref{section_takahashi_proof}, the scattering diagram $S(\fD^{\iin}_{\cl^+})$
being replaced by its quantum deformation
$S(\fD^{\iin}_{y^+})$.

By the main result of \cite{bousseau2019scattering} 
reviewed in \cref{section_vertical_sheaves_refined} the vertical rays of 
$S(\fD^{\iin}_{y^+})$ compute the sheaf invariants $\Omega_{d,\chi}^{\PP^2}(y^{\frac{1}{2}})$.
In \cref{section_vertical_gw_refined}, we show that the vertical rays of 
$S(\fD^{\iin}_{y^+})$ also compute the maximal contact higher-genus
Gromov-Witten invariants of $(\PP^2,E)$.
It is a higher-genus version of the genus-$0$ result of \cite{gabele2019tropical}
whose proof combines the degeneration argument of \cite{gabele2019tropical} with 
\cite{MR3904449, bousseau2018quantum_tropical}.  
Putting together these two geometric interpretations of the quantum scattering diagram $S(\fD^{\iin}_{y^+})$, we prove 
Theorem \ref{thm_log_bps_local_bps_refined_intro} in \cref{section_proof_thm_log_bps_local_bps_refined_intro}. Finally, we explain how to deduce
Theorem \ref{thm_local_relative_intro}
and Theorem \ref{thm_NS_intro}
in \cref{section_proof_section_proof_thm_log_bps_local_bps_refined_intro}
and \cref{section_proof_thm_NS_intro}.

\subsection{Scattering and  refined sheaf-counting on $\PP^2$}
\label{section_vertical_sheaves_refined}

This section is entirely parallel to 
\cref{section_vertical_sheaves}
(and more precisely recovers \cref{section_vertical_sheaves} in the limit 
$y^{\frac{1}{2}} \rightarrow 1$).

The quantum scattering diagram $S(\fD^\iin_{y^+})$ is defined 
in \cite{bousseau2019scattering} as the consistent completion of an explicit initial 
scattering diagram $\fD^\iin_{y^+}$ (see \cite[Definition 1.4.8]{bousseau2019scattering}). The scattering diagram $S(\fD^\iin_{\cl^+})$ used
in \cref{section_takahashi_proof} is the classical limit $y^{\frac{1}{2}}
\rightarrow 1$ of $S(\fD^\iin_{y^+})$.

On the other hand, viewing $U$ as a space of Bridgeland stability conditions
on the bounded derived category $\D^b(\PP^2)$ of coherent sheaves on $\PP^2$, we defined in \cite[Definition 2.5.7]{bousseau2019scattering} a quantum scattering diagram
$\fD^{\PP^2}_{y^+}$ using intersection Poincar\'e polynomials of 
moduli spaces of $\sigma$-semistable objects in $\D^b(\PP^2)$ for $\sigma \in U$.

One of the main results of \cite{bousseau2019scattering} is that the algorithmically defined scattering diagram $S(\fD^\iin_{y^+})$ and the geometrically defined 
scattering diagram $\fD^{\PP^2}_{y^+}$ coincide: by 
\cite[Theorem 5.2.3]{bousseau2019scattering}, we have
\[ S(\fD^\iin_{y^+}) = \fD^{\PP^2}_{y^+} \,.\]

We focus on the unbounded vertical rays of $S(\fD^\iin_{y^+})$. We show below that there are naturally indexed by $d \in \Z_{\geq 1}$ and $\chi \in \Z$, and that the function 
$H_{\fd_{d,\chi}^y}$ attached to the ray $\fd_{d,\chi}^y$ are expressed in terms of the intersection 
Poincar\'e polynomial of the moduli spaces of Gieseker semistable one-dimensional sheaves on $\PP^2$, supported on curves of degree $d'$ and with
holomorphic Euler characteristic $\chi'$, for $(d',\chi')$ dividing $(d,\chi)$.

\begin{lem} \label{lem_vertical_rays_refined}
Let $d$ be a positive integer. For every unbounded ray $\fd$
of $S(\fD^\iin_{y^+})$ of class $m_\fd=(0,-d)$, there exists a unique $\chi \in \Z$ such that $|\fd|$ is contained in the vertical line of equation 
\[ x=\frac{1}{d}\left(\chi-\frac{3}{2}\right)\,.\]

Moreover, for every $d \in \Z_{\geqslant 1}$ and $\chi \in \Z$, there exists a unique unbounded ray in $S(\fD^\iin_{y^+})$, that we denote by
$\fd_{d,\chi}^y$, of class $m_{\fd_{d,\chi}}
=(0,-d)$ and contained in the vertical line of equation $ x=\frac{1}{d}\left(\chi-\frac{3}{2}\right)$.
\end{lem}

\begin{proof}
By \cite[Definition 2.5.7]{bousseau2019scattering}, every ray of $\fD^{\PP^2}_{y^+}$
is contained in a line of equation $ry+dx+r+\frac{3}{2}d-\chi=0$
for some $\gamma=(r,d,\chi)\in \Gamma$. 
If $\fd$ is of class 
$m_\fd=(0,-d)$, it means that $r=0$, and so 
$|\fd|$ is contained of the line of equation $dx+\frac{3}{2}d-\chi=0$.

For an unbounded vertical ray, Bridgeland $\sigma$-stability 
for $\sigma \in U$ reduces to Gieseker stability by 
\cite[Lemma 6.1.1]{bousseau2019scattering}.
Therefore, the existence of a ray 
$\fd_{d,\chi}$ as in the statement of 
Lemma \ref{lem_vertical_rays} follows from
the definition 
\cite[Definition 2.5.7]{bousseau2019scattering} of $\fD^{\PP^2}_{y^+}$, and from
the fact that, for every $d \in \Z_{>0}$ and 
$\chi \in \Z$, the moduli space 
$M_{d,\chi}$ of Gieseker semistable 
degree-$d$ one-dimensional sheaves on $\PP^2$ of holomorphic Euler characteristic $\chi$ is nonempty, 
and so, by \cite[Corollary 6.1.3]{bousseau2019scattering}, $Ie(M_{d,\chi})
\neq 0$ and in particular $\Omega_{d,\chi}^{\PP^2}(y^{\frac{1}{2}}) \neq 0$.
\end{proof}

\begin{thm}\label{thm_boussseau_refined}
For every $d \in \Z_{>0}$ and $\chi \in \Z$, the function 
$H_{\fd_{d,\chi}^y}$ attached to the unbounded vertical ray 
$\fd_{d,\chi}^y$ in the quantum scattering diagram $S(\fD^\iin_{y^+})$ is given by
\[ H_{\fd_{d,\chi}^y}=
(-1)^{d-1}
\left( \sum_{\substack{\ell \in \Z_{\geqslant 1} \\\ell|(d,\chi)}}
\frac{1}{\ell^2} \Omega_{\frac{d}{\ell},\frac{\chi}{\ell}}^{\PP^2}(y^{\frac{\ell}{2}})
\right) z^{(0,-d)} \,, \]
\end{thm}

\begin{proof}
For an unbounded vertical ray, Bridgeland $\sigma$-stability 
for $\sigma \in U$ reduces to Gieseker stability by 
\cite[Lemma 6.1.1]{bousseau2019scattering}.
Therefore, the result follows directly from
the definition 
\cite[Definition 2.5.7]{bousseau2019scattering} of $\fD^{\PP^2}_{y^+}$, 
and from the fact that $(-1)^{(\gamma,\gamma)}=(-1)^{d}$ for 
$\gamma=(0,d,\chi)$ (see \cite[Definition 2.1.1]{bousseau2019scattering}).

\end{proof}

\subsection{Scattering and higher-genus Gromov-Witten invariants}
\label{section_vertical_gw_refined}

In Theorem \ref{thm_gabele_refined} below, we prove a higher-genus version of the main result of  
\cite{gabele2019tropical}, that is of
Theorem \ref{thm_gabele}. 
More precisely, we express the functions $H_{\fd_{d,\chi}^y}$ attached to the unbounded vertical rays 
$\fd_{d,\chi}^y$ in the quantum scattering diagram $S(\fD^\iin_{y^+})$ in terms of the
higher-genus relative Gromov-Witten invariants $N_{g,d}^{\PP^2/E,k}$ defined in \S \ref{section_higher_genus_intro}.

\begin{thm}\label{thm_gabele_refined}
For every 
$d \in \Z_{\geqslant 1}$ and $\chi \in \Z$, we have 
\[ H_{\fd_{d,\chi}^y}
=(-i)\left( 
\sum_{\substack{k\in \Z_{\geqslant 1}\\
\ell_{d,\chi}|k|d}}
\frac{s_{k,\ell_{d,\chi}}}{r_{\ell_{d,\chi}}}
\left(
\sum_{g \geqslant 0} N_{g,d}^{\PP^2/E,k}
\hbar^{g-1}
\right)
\right)z^{(0,-d)} \]
after the change of variables $y=e^{i\hbar}$.
\end{thm}

\begin{proof}
We describe briefly how to modify the genus-$0$ argument of
\cite{gabele2019tropical}.

As in \cite{gabele2019tropical}, we consider 
\[ \cX \coloneq \{ ([x:y:z:w], t) \in \PP(1,1,1,3) \times 
\A^1 |\, xyz-t^3(x^3+y^3+z^3+w)=0\}\,,\]
where $\PP(1,1,1,3)$ is the weighted projective space with weights $(1,1,1,3)$,
and 
\[ \pi \colon \cX \rightarrow \A^1 \]
\[ ([x:y:z:w],t) \mapsto t \,.\]
We view $\pi \colon \cX \rightarrow \A^1$ as a degeneration of $(\PP^2,E)$.
Indeed, for $t \neq 0$, we have $\cX_t \coloneq \pi^{-1}(t) \simeq \PP^2$ and $\cX_t$ intersects the toric boundary divisor 
$\PP^2=\{ w=0\}$ of $\PP(1,1,1,3)$
along the cubic curve 
\[ E_t: xyz-t^3(x^3+y^3+z^3)=0\,.\] 
The special fiber $\cX_0 \coloneq \pi^{-1}(0)$ breaks into the union of the three other toric divisors 
$\{x=0\}$, $\{y=0\}$, $\{z=0\}$, each one being isomorphic to the weighted projective plane 
$\PP(1,1,3)$. The cubic $E_0$
breaks into a triangle of lines.

It is shown in \cite{gabele2019tropical}
that the dual intersection complex of 
$\cX_0$ coincides with the polyhedral decomposition  $(B, \mathscr{P})$ of 
\cite[Example 2.4]{cps} and reviewed in \cref{section_cps}. The polyhedral decomposition 
$\cP$ contains three vertices, dual to the 
three components of $\cX_0$, defining a triangle $T$ dual to the triple intersection point of the three components of $\cX_0$. 
As each irreducible component of $\cX_0$ is toric, there is a natural way to define an integral affine structure on the complement $B_0$ in $B$ of three focus-focus singularities,
such that the local picture of $\cP$ near each of the three vertices is the fan of the corresponding toric component of 
$\cX_0$. 

The presence of the three singularities of the affine structure is related to the fact that the total space of $\cX$ has three ordinary double points and that the family $\cX \rightarrow \A^1$
is not log smooth at these points. Here, we view $\cX$ as endowed with the divisorial log structure defined by the vertical divisor $\cX_0$ and the horizontal divisor $(E_t)_{t \in \A^1}$. We view $\A^1$ as endowed with the divisorial log structure defined by $0 \in \A^1$.
Taking small resolutions of the three ordinary double points, we obtain a log smooth degeneration $\tilde{\cX} \rightarrow \A^1$.

For fixed $d \geqslant 1$, it is shown in \cite{gabele2019tropical} that there are finitely many genus-$0$ tropical curves in $(B,\mathscr{P})$ of type relevant for a tropical description of the invariants $N_{0,d}^{\PP^2,E}$. Picking a polyhedral decomposition containing all these tropical curves, we construct a log smooth degeneration 
$\tilde{\cX}_d \rightarrow \A^1$, obtained by blow-ups and base change from $\tilde{\cX} \rightarrow \A^1$.

We apply the decomposition formula of 
\cite{abramovich2017decomposition} to the degeneration $\tilde{\cX}_d \rightarrow \A^1$. The decomposition formula expresses the invariants 
$N_{g,d}^{\PP^2/E}$ as a sum over genus-$g$ rigid tropical curves.
As $N_{g,d}^{\PP^2/E}$ is defined in 
\cref{section_higher_genus_intro} by integration of the class $\lambda_g$, which vanishes on families of stable curves whose dual graph has positive genus (see
\cite[\S 3]{MR3904449} for a review), the only contributing rigid genus-$g$ tropical curves are obtained by distributing the genus on the vertices of the rigid genus-$0$ tropical curves.

The contribution of each rigid tropical curve can be expressed as gluing of contributions attached to the vertices of the tropical curve. Because $\tilde{\cX}_d$ is defined by a polyhedral decomposition
of $(B, \mathscr{P})$  
containing all the relevant 
tropical curves, all the relevant genus-$0$
stable log maps to the special fiber are torically transverse (that is not meeting the codimension $2$ strata of the special fiber) and this is used in
\cite{gabele2019tropical} to prove the gluing statement: gluing of torically transverse stable log maps can be reduced to the degeneration formula along a smooth divisor \cite{kim2018degeneration}.
In higher genus, things are more complicated: the moduli space of genus-$g$ stable log maps to the special fiber attached to a vertex of the tropical curve can contain non-torically tansverse stable maps. Nevertheless, it is possible to prove the gluing using the vanishing properties of $\lambda_g$ as in 
\cite[\S 6-7]{MR3904449}.

It remains to evaluate the contribution of each vertex and to show that the gluing procedure recovers the combinatorial description of the scattering diagram 
$S(\fD^{\iin}_{\cl^+})$ in genus-$0$ and of the quantum scattering diagram 
$S(\fD^{\iin}_{y^+})$
in higher-genus. In genus-$0$, it is shown in \cite{gabele2019tropical} that this follows from the main result of 
\cite{MR2667135} which establishes the relation between genus-$0$ log Gromov-Witten invariants of toric surfaces and scattering diagrams. Similarly, the higher-genus result follows from the main result of 
\cite{bousseau2018quantum_tropical} which establishes the relation between higher-genus log Gromov-Witten invariants with 
$\lambda_g$-insertion of toric surfaces and quantum scattering diagrams.
\end{proof}

\subsection{Proof of Theorem \ref{thm_log_bps_local_bps_refined_intro}}
\label{section_proof_thm_log_bps_local_bps_refined_intro}
The functions $H_{\fd_{d,\chi}^y}$ attached to the unbounded vertical rays
$\fd_{d,\chi}^y$ of the quantum scattering diagram  $S(\fD^\iin_{y^+})$ are expressed in terms of the sheaf counting invariants $\Omega_{d',\chi'}^{\PP^2}(y^{\frac{1}{2}})$ by Theorem 
\ref{thm_boussseau_refined} and in terms of the Gromov-Witten invariants 
$N_{g,d}^{\PP^2/E,k}$ by Theorem \ref{thm_gabele_refined}. Comparing these two expressions, we obtain the following result.

\begin{thm} \label{thm_sheaf_gw_refined}
For every $d \in \Z_{\geqslant 1}$ and $\chi \in \Z$, we have 
\[ (-1)^{d-1}
\sum_{\substack{\ell \in \Z_{\geqslant 1} \\\ell|(d,\chi)}}
\frac{1}{\ell^2} \Omega_{\frac{d}{\ell},\frac{\chi}{\ell}}^{\PP^2}(y^{\frac{\ell}{2}})
=
(-i)
\sum_{\substack{k\in \Z_{\geqslant 1}\\
\ell_{d,\chi}|k|d}}
\frac{s_{k,\ell_{d,\chi}}}{r_{\ell_{d,\chi}}}
\left(
\sum_{g \geqslant 0} N_{g,d}^{\PP^2/E,k}
\hbar^{g-1}
\right)\]
after the change of variables $y=e^{i\hbar}$.
\end{thm}

Theorem \ref{thm_sheaf_gw_refined} implies Theorem \ref{thm_log_bps_local_bps_refined_intro}
in the same way that Theorem \ref{thm_sheaf_gw} implies Theorem \ref{thm_log_bps_local_bps_intro}
(see \cref{section_proof_thm_log_bps_local_bps}
).

\subsection{Proof of Theorem \ref{thm_equiv_conj_intro}}
\label{section_proof_section_proof_thm_log_bps_local_bps_refined_intro}
In this section, we explain how Theorem \ref{thm_equiv_conj_intro}
follows from Theorem \ref{thm_log_bps_local_bps_refined_intro}.

According to Theorem \ref{thm_log_bps_local_bps_refined_intro}, we have 
\[ 
\Omega_{d,\chi}^{\PP^2}(y^{\frac{1}{2}})=\sum_{\ell_{d,\chi}|k|d}
\frac{s_{k,\ell_{d,\chi}}}{r_{\ell_{d,\chi}}} \Omega_{d,k}^{\PP^2/E}(\hbar)\,\]
after the change of variables 
$y=e^{i\hbar}$.

If Conjecture \ref{conj_tak_higher_genus_bps_intro} holds, then $\Omega_{d,k}^{\PP^2/E}(\hbar)$ does not depend on $k$, and so, denoting by 
$\Omega_d^{\PP^2/E}(\hbar)$ the common value of the invariants $\Omega_{d,k}^{\PP^2/E}(\hbar)$, we have 
\[ \Omega_{d,\chi}^{\PP^2}(y^{\frac{1}{2}})=
\Omega_d^{\PP^2/E}(\hbar)
\sum_{\ell_{d,\chi}|k|d}
\frac{s_{k,\ell_{d,\chi}}}{r_{\ell_{d,\chi}}}\,.\]
We saw in \cref{section_thm_local_relative_intro} that 
\[ \sum_{\ell_{d,\chi}|k|d}
\frac{s_{k,\ell_{d,\chi}}}{r_{\ell_{d,\chi}}}=3d\,.\]
Therefore, we have 
\[ 
\Omega_{d,\chi}^{\PP^2}(y^{\frac{1}{2}})=3d\,
\Omega_d^{\PP^2/E}(\hbar)\,.\]
In particular, $\Omega_{d,\chi}^{\PP^2}$ does not depend on $\chi$ and Conjecture 
\ref{conj_refined_joyce_conj_intro} holds.

Conversely, let us assume that Conjecture 
\ref{conj_refined_joyce_conj_intro} holds. Then $\Omega_{d,\chi}^{\PP^2}$ does not depend on $\chi$ and we denote by 
$\Omega_d^{\PP^2}$ the common value of the invariants $\Omega_{d,\chi}^{\PP^2}$.
We can argue as in \cref{section_thm_local_relative_intro}.
Theorem \ref{thm_log_bps_local_bps_refined_intro} gives a recursive way to determine the invariants $\Omega_{d,k}^{\PP^2}$ in terms of the invariants $\Omega_d^{\PP^2}$. 
Therefore, in order to show that 
$\Omega_{d,k}^{\PP^2/E}(\hbar)=
\frac{1}{3d}
\Omega_d^{\PP^2}(y^{\frac{1}{2}})$, and in particular that Conjecture \ref{conj_tak_higher_genus_bps_intro} holds, it is enough to show that this formula is the solution of the above set of recursive equations. As in \cref{section_thm_local_relative_intro}, this result follows from 
\[ \sum_{\ell_{d,\chi}|k|d}
\frac{s_{k,\ell_{d,\chi}}}{r_{\ell_{d,\chi}}}=3d\,.\]

\subsection{Proof of Theorem \ref{thm_NS_intro}}
\label{section_proof_thm_NS_intro}

We explain how Theorem \ref{thm_NS_intro}
follows from Theorem \ref{thm_log_bps_local_bps_refined_intro}.

By Theorem \ref{thm_log_bps_local_bps_refined_intro}, we have
\[ 
\Omega_{d,\chi}^{\PP^2}(y^{\frac{1}{2}})=\sum_{\ell_{d,\chi}|k|d}
\frac{s_{k,\ell_{d,\chi}}}{r_{\ell_{d,\chi}}} \Omega_{d,k}^{\PP^2/E}(\hbar)\,,\]
so 
\[ \Omega_d^{\PP^2}(y^{\frac{1}{2}})
\coloneq \frac{1}{d} \sum_{\chi\mod d} \Omega_{d,\chi}^{\PP^2}(y^{\frac{1}{2}})
=\frac{1}{3d} \sum_{k|d} a_{d,k}
 \Omega_{d,k}^{\PP^2/E}(\hbar) \,,\]
where 
\[ a_{d,k} \coloneq 3
\sum_{\substack{ \chi \mod d \\
\ell_{d,\chi}|k}}
\frac{s_{k,\ell_{d,\chi}}}{r_{\ell_{d,\chi}}}\,\]
is the number of $x \in \Z/(3d)
\times \Z/(3d)$ such that $d(x)=k$.
The factor $3$ comes from the fact that we sum over $\chi \mod d$ and not over $\chi \mod 3d$.

Therefore, we have
\[ \bar{F}^{NS}(y^{\frac{1}{2}},Q) 
\coloneq i \sum_{d \in \Z_{>0}}
\sum_{\ell \in \Z_{>0}} \frac{1}{\ell} \frac{\Omega_{d}(y^{\frac{\ell}{2}})}{
y^{\frac{\ell}{2}}-y^{-\frac{\ell}{2}}} Q^{\ell d} 
=
i \sum_{d \in \Z_{>0}} \frac{1}{3d}
\sum_{k|d} a_{d,k} 
\sum_{\ell \in \Z_{>0}}
\frac{1}{\ell} \frac{
\Omega_{d,k}^{\PP^2/E}(\ell \hbar)}{
y^{\frac{\ell}{2}}-y^{-\frac{\ell}{2}}} Q^{\ell d} \,.\]
Using the definition of $\Omega_{d,k}^{\PP^2/E}(\hbar)$ given in \cref{section_higher_genus_intro}, we obtain
\[ \bar{F}^{NS}(y^{\frac{1}{2}},Q) 
= i \sum_{g \geqslant 0}
\sum_{d \in \Z_{>0}} \frac{1}{3d}
(-1)^{d-1} \sum_{k|d} a_{d,k} N_{g,d}^{\PP^2/E,k} \hbar^{2g-1} Q^{d}\,.\]
As 
\[ N_{g,d}^{\PP^2/E}=\sum_{p \in P_d}
N_{g,d}^{\PP^2/E,p} 
=\sum_{k|d} a_{d,k} N_{g,d}^{\PP^2/E,k}\,, \]
we obtain
\[ \bar{F}^{NS}(y^{\frac{1}{2}},Q) 
= i \sum_{g \geqslant 0}
\sum_{d \in \Z_{>0}} \frac{1}{3d}
(-1)^{d-1} N_{g,d}^{\PP^2,E}
\hbar^{2g-1} Q^{d}\,,\]
and this ends the proof of Theorem 
\ref{thm_NS_intro}.

\section{Further results}
\label{section_more_gw}

\subsection{Compatibility with the local-relative correspondence}
\label{section_local_relative}
Recall that once we know Theorem 
\ref{thm_joyce_conj_intro}, we write 
$\Omega_d^{\PP^2}$ for 
$\Omega_{d,\chi}^{\PP^2}$, and that 
once we know Theorem \ref{thm_takahashi_bps}, we write 
$\Omega_d^{\PP^2/E}$ for 
$\Omega_{d,k}^{\PP^2/E}$.
According to Theorem \ref{thm_local_relative_intro}, we have
\[ \Omega_d^{\PP^2/E} = \frac{1}{3d}
\Omega_{d}^{\PP^2}\,.\]
We can use this result to give a new proof of the local-relative correspondence of \cite{MR3948687} relating 
genus-$0$ Gromov-Witten invariants 
$N_{0,d}^{\PP^2/E}$ of $(\PP^2,E)$ and 
genus-$0$ Gromov-Witten invariants 
$N_{0,d}^{K_{\PP^2}}$ of $K_{\PP^2}$.

\begin{thm}\label{thm_local_relative}
For every $d \in \Z_{>0}$, we have 
\[N_{0,d}^{\PP^2/E}=(-1)^{d-1}3d N_{0,d}^{K_{\PP^2}}\,.\]
\end{thm}

\begin{proof}
Using the notation introduced in 
\cref{section_takahashi_intro}, denote by 
$P_{d,k}$ the set of $p \in P_d$ 
with $d(p)=k$, and $|P_{d,k}|$
the cardinal of $P_{d,k}$.
For every positive integer $d$, we have 
\[ N_{0,d}^{\PP^2/E}
=\sum_{\substack{k \in \Z_{\geqslant 1}\\
k|d}} |P_{d,k}|N_{0,d}^{\PP^2/E,k}
=(-1)^{d-1}\sum_{\substack{k \in \Z_{\geqslant 1}\\
k|d}} |P_{d,k}|
\sum_{\substack{d'\in \Z_{\geqslant 1}\\
k|d'|d}} \frac{1}{(d/d')^2}\Omega_{d'}^{\PP^2/E}\]
\[=(-1)^{d-1} \sum_{\substack{d'\in \Z_{\geqslant 1}\\
d'|d}} \frac{1}{(d/d')^2}\Omega_{d'}^{\PP^2/E}\sum_{\substack{k \in \Z_{\geqslant 1}\\
k|d'}} |P_{d,k}|\,.\]
Using \[\sum_{\substack{k \in \Z_{\geqslant 1}\\
k|d'}} |P_{d,k}|=(3d')^2\,,\]
we obtain 
\[ N_{0,d}^{\PP^2/E}=(-1)^{d-1} \sum_{\substack{d'\in \Z_{\geqslant 1}\\
d'|d}} \frac{(3d')^2}{(d/d')^2}\Omega_{d'}^{\PP^2/E}\,.\]
 
By Theorem 
\ref{thm_local_relative_intro},
we have $(3d')\Omega_{d'}^{\PP^2/E}
=\Omega_{d'}^{\PP^2}$ and so
\[ N_{0,d}^{\PP^2/E}=(-1)^{d-1} \sum_{\substack{d'\in \Z_{\geqslant 1}\\
d'|d}} \frac{3d'}{(d/d')^2}\Omega_{d'}^{\PP^2}=(-1)^{d-1}3d\sum_{\substack{d'\in \Z_{\geqslant 1}\\
d'|d}} \frac{1}{(d/d')^3}\Omega_{d'}^{\PP^2} \,.\]
On the other hand, we have 
\[ N_{0,d}^{K_{\PP^2}}=\sum_{\substack{d'\in \Z_{\geqslant 1}\\
d'|d}} \frac{1}{(d/d')^3} n_{0,d'}^{K_{\PP^2}}\,,\]
where  $n_{0,d}^{K_{\PP^2}}$ is the genus $0$ degree $d$ Gopakumar-Vafa invariant of 
$K_{\PP^2}$ defined through Gromov-Witten theory. Finally, by Proposition
\ref{prop_katz_conj}, we have that Katz's conjecture
(\cite[Conjecture 2.3]{MR2420017}) is true for $K_{\PP^2}$, that is, for every $d \in \Z_{>0}$, we have
\[ n_{0,d}^{K_{\PP^2}}=\Omega_{d}^{\PP^2}\,.\]
This concludes the proof.
\end{proof}

 Our proof of Theorem \ref{thm_local_relative} is quite complicated, using essentially the full content of the present paper. 
 In particular, this proof contains a long detour through sheaf counting, whereas the simple degeneration argument of \cite{MR3948687} gives directly a proof staying in Gromov-Witten theory. 
 Nevertheless, we stress that the two proofs are logically independent, and combining them, we obtain a full circle of relations whose consistency can be appreciated:

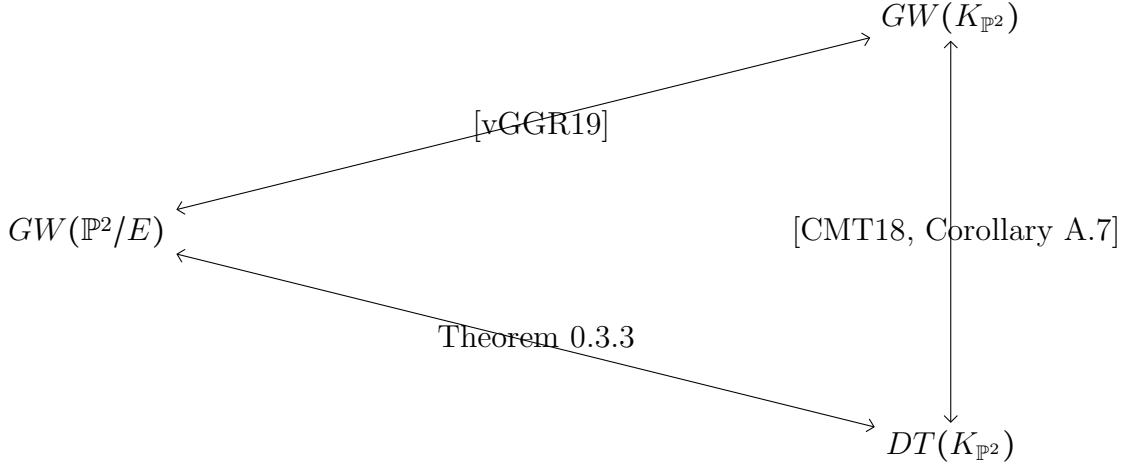
\begin{center}
\begin{tikzpicture}[>=angle 90]
\matrix(a)[matrix of math nodes,
row sep=2em, column sep=4em,
text height=1.5ex, text depth=0.25ex]
{&&&GW(K_{\PP^2})\\
&&\text{\cite{MR3948687}}& \\
GW(\PP^2/E)&&
&\text{  \cite[Corollary A.7]{MR3861701}}\\
&&\text{Theorem \ref{thm_local_relative_intro}}&\\
&&&DT(K_{\PP^2})\\};
\path[<-](a-3-1) edge (a-2-3);
\path[->](a-2-3) edge (a-1-4);
\path[<-](a-1-4) edge  (a-3-4);
\path[->](a-3-4) edge  (a-5-4);
\path[<-](a-3-1) edge (a-4-3);
\path[->](a-4-3) edge (a-5-4);
\end{tikzpicture}
\end{center}

\subsection{A conjectural real sheaves/Gromov-Witten correspondence}
\label{section_real_gw}

As $K_{\PP^2}$ is toric, $K_{\PP^2}$ has a natural real structure, whose real locus 
$K_{\PP^2}^{\R}$ is the real projective plane 
$\R \PP^2$ in restriction to $\PP^2$.
Arguing as in \cite{MR2425184}, we can define genus-$0$ degree-$d$ real Gromov-Witten invariants of $K_{\PP^2}$:
\[ N_{0,d}^{K_{\PP^2},\R} \in \Q\,.\]
We have $N_{0,d}^{K_{\PP^2,\R}}=0$ if $d$ is even, and we choose the Spin structure on $\R \PP^2$ involved in the definition of 
$N_{0,d}^{K_{\PP^2},\R}$ such that 
$N_{0,1}^{K_{\PP^2},\R}=1$
(a different choice changes the invariants 
$N_{0,d}^{K_{\PP^2},\R}$ by a global sign independent of $d$).
We define genus-$0$ degree-$d$ real Gopakumar-Vafa invariants 
$n_{0,d}^{K_{\PP^2},\R}$ by:
$n_{0,d}^{K_{\PP^2},\R}=0$ if $d$ is even, and by
\[ N_{0,d}^{K_{\PP^2},\R}
=\sum_{\substack{\ell \geqslant 1\\\ell|d}} \frac{1}{\ell^2} 
n_{0,\frac{d}{\ell}}^{K_{\PP^2},\R}\,,\]
for every $d$ odd.

We recall from  \cite[\S 6.2]{bousseau2019scattering} that, for every $d\in \Z_{>0}$,  the moduli space $M_{d,1}$ has a natural real structure and we denote by $M_{d,1}(\R)$ its real locus.
In particular, 
$M_{d,1}(\R)$ is a smooth compact manifold, and we denote by 
$e(M_{d,1}(\R)) \in \Z$ its topological Euler characteristic.

\begin{conj} \label{conj_real}
For every positive integer $d$, we have 
\[n_{0,d}^{K_{\PP^2},\R}=(-1)^{\frac{d-1}{2}} e(M_{d,1}(\R))\,.\]
\end{conj}

As reviewed in Proposition
\ref{prop_katz_conj}, Katz's conjecture
(\cite[Conjecture 2.3]{MR2420017}) is true for $K_{\PP^2}$,
that is, we have $n_{0,d}^{K_{\PP^2}}
=(-1)^{d-1} e(M_{d,1})$, where 
$n_{0,d}^{K_{\PP^2}}$ is the genus $0$
degree $d$ (complex) Gopakumar-Vafa invariant defined in terms of 
genus $0$ Gromov-Witten theory of 
$K_{\PP^2}$, and where $e(M_{d,1})$
is the (complex) topological Euler characteristic of $M_{d,1}$.
Thus, Conjecture \ref{conj_real} is a natural real analogue of Katz's conjecture for $K_{\PP^2}$. Despite this clear analogy, Conjecture \ref{conj_real} seems to be new.

Conjecture \ref{conj_real} is related to the main theme of the present paper because, by \cite[Theorem 6.4]{bousseau2019scattering},
the real Euler characteristic  
$(-1)^{\frac{d-1}{2}} e(M_{d,1}(\R))$ is the specialization at
$y=-1$ ($y^{\frac{1}{2}}=i$) of the refined invariant $\Omega_{d,1}^{\PP^2}(y^{\frac{1}{2}})$, 
which can be computed by the scattering diagram 
$S(\fD^{\iin}_{y^+})$.
Therefore, we can rephrase Conjecture 
\ref{conj_real} as the statement 
that the real Gopakumar-Vafa invariants 
$n_{0,d}^{K_{\PP^2},\R}$ can be computed 
by the specialization at $y=-1$ of the invariants computed by the scattering diagram $S(\fD^{\iin}_{y^+})$, which are essentially $y$-refined tropical invariants in the tropicalization of 
$(\PP^2,E)$ (see  \cref{section_comparison_cps_wall_structure}).
This statement is analogous to the fact that, in the setting of real toric surfaces
with point insertions, Welschinger's counts 
of genus $0$ real curves can be computed by the 
specialization $y=-1$ of the Block-Göttsche 
$y$-refined count of 
tropical curves in 
$\R^2$ \cite{MR2137980, MR3453390}.
Thus, a natural strategy to prove Conjecture 
\ref{conj_real} would be to relate 
$n_{0,d}^{K_{\PP^2},\R}$ to some kind of real version of the Gromov-Witten theory of the pair $(\PP^2,E)$, and then to prove 
a real version of the tropical correspondence theorem of \cite{gabele2019tropical}. 
The relevant tools in real Gromov-Witten theory do not seem to be developed enough yet, 
and so we leave this question open.

Conjecture \ref{conj_real} is true for $d$ even for almost trivial reasons:
we have $n_{0,d}^{K_{\PP^2},\R}=0$, and as 
$\dim M_{d,1}=d^2+1$ is odd, we also have $(-1)^{\frac{d-1}{2}} e(M_{d,1}(\R))=\Omega_{d,1}^{\PP^2}(y^{\frac{1}{2}}=i)=0$ by Poincar\'e duality.

In low degrees, the real Gopakumar-Vafa invariants $n_{0,d}^{K_{\PP^2},\R}$ 
have been computed by some real topological vertex arguments in the physics paper \cite{krefl2009real}. 
Assuming that these arguments can be made completely rigorous, which is likely but apparently not yet done\footnote{I thank Penka Georgieva for a discussion of this point.}, we can compare the explicit results of 
\cite[Table 2, p33]{krefl2009real} for $n_{0,d}^{K_{\PP^2},\R}$ in low degrees, with
explicit computations of
$(-1)^{\frac{d-1}{2}} e(M_{d,1}(\R))=\Omega_{d,1}^{\PP^2}(y^{\frac{1}{2}}=i)$, which can be done using either the known computations of 
$\Omega_{d,1}^{\PP^2}(y^{\frac{1}{2}})$ 
in the literature
(\cite{MR3217411} for
$(d,\chi)=(4,1)$,
\cite{MR3488141, MR3324766,
MR3152206} for $(d,\chi)=(5,1)$,
\cite{MR3319919} for $(d,\chi)=(6,1)$, and \cite[Table 2]{MR3035638} for physics predictions for $d \leqslant 7$), or using our scattering algorithm (see \cite[\S 6.4]{bousseau2019scattering} for examples). We have checked this way that Conjecture \ref{conj_real} holds for $d \leqslant 7$.

\section{Some heuristic explanation}
\label{section_heuristic_explanation}

As in the rest of the paper, we consider a smooth cubic curve $E$
in the complex projective plane
$\PP^2$.
We denote by $V$ the complement of $E$
in $\PP^2$. As $E$ is an anticanonical divisor of $\PP^2$, the noncompact surface $V$ is naturally
holomorphic symplectic.

In this section, we give a heuristic explanation for the main result of the paper: the link between Gromov-Witten invariants of $(\PP^2,E)$ and Donaldson-Thomas invariants of $K_{\PP^2}$ through the scattering diagram $S(\fD^{\iin}_{y^+})$.

From the Gromov-Witten point of view,  $S(\fD^{\iin}_{y^+})$ is contained in
 a tropical version of 
$(\PP^2,E)$, and so morally (in some appropriate tropical limit) in the base of a Strominger-Yau-Zaslow (SYZ) torus fibration on $V$. On the other hand, from the Donaldson-Thomas point of view, $S(\fD^{\iin}_{y^+})$ is contained in a space of stability conditions on the derived category $\D^b_0(K_{\PP^2})$ of coherent sheaves on $K_{\PP^2}$
set-theoretically supported on the zero-section.
More precisely, $S(\fD^{\iin}_{y^+})$ is contained in (a triple cover of) the stringy Kähler moduli space of $K_{\PP^2}$ (see \cref{section_stringy_kahler}).
Therefore, we need to understand why the base of the SYZ fibration on $(\PP^2,E)$ should be indentified with (a triple cover of) the stringy Kähler moduli space of $K_{\PP^2}$.
Furthermore, we would like to understand why the holomorphic curves in the SYZ fibration of $V$, described tropically by 
$S(\fD^{\iin}_{y^+})$, should be related to 
Bridgeland semistable objects in 
$\D^b_0(K_{\PP^2})$, whose wall-crossing behavior is described by $S(\fD^{\iin}_{y^+})$.
Our heuristic explanation will have three steps:
\begin{enumerate}
    \item Hyperkähler rotation for $V$.
    \item Suspension from dimension $2$ to dimension $3$.
    \item Mirror symmetry for $K_{\PP^2}$.
\end{enumerate}

\subsection{Hyperkähler rotation for $V$}
It was recently shown \cite{collins2019special}
that $V$ admits a hyperkähler metric and a
special Lagrangian torus fibration
\[ \pi \colon V \rightarrow B \simeq \R^2\,\] with three singular fibers. 
Tropicalization of holomorphic curves in $V$ naturally live in $B$ (after tropical limit). 
It is also shown in \cite{collins2019special} that, after hyperkähler rotation, this torus fibration becomes
an elliptic fibration 
\[ \pi \colon M \rightarrow B \simeq \C\,,\] 
which is the fiberwise compactified mirror of $\PP^2$
(see \cite[\S 3.1]{MR2257391}).
In other words, despite being very different as algebraic complex manifolds 
($V$ is affine whereas $M$ admits a stucture of elliptic fibration), 
$V$ and $M$ are diffeomorphic and there exists a hyperkähler metric which is compatible with both complex structures.
In order to keep the exposition short, 
we suppress the discussion of the exact match of parameters 
(there is in fact a family of hyperkähler metrics, and the mirror of 
$\PP^2$ has complex and K\"ahler moduli).

Under hyperkähler rotation, holomorphic curves in $V$
with boundary on torus fibers of the Lagrangian torus fibration become open special Lagrangian submanifolds 
in $M$ with boundary on fibers of the elliptic fibration. More precisely, the special Lagrangians submanifolds obtained in that way are special Lagrangian submanifolds of a given specific phase.

\subsection{Suspension from dimension $2$ to dimension $3$}

We refer to \cite{MR2651908} for the topic of suspension in symplectic geometry.
For every $t \in B \simeq \C$
away from critical values of $\pi$, we define a 
noncompact Calabi-Yau 3-fold 
$Y_t$ by the equation $uv=\pi-t$.
The 3-fold $Y_t$ is 
a fibration in affine quadrics over the surface 
$M$, degenerate over the fiber $\pi^{-1}(t)$
in $M$.
By suspension, the open special Lagrangians
in $M$ with boundary on the fiber 
$\pi^{-1}(t)$ of the elliptic fibration
lift to closed special Lagrangians in $Y_t$.

The parameter $t \in B$ is a complex moduli of the Calabi-Yau 3-fold $Y_t$. From a symplectic point of view, $Y_t$ is independent of $t$, and we denote by $F(Y)$ the corresponding compact Fukaya category.
It is a general expectation
\cite{MR1882337, MR1957663, MR3354954} that a choice of
complex structure should define a Bridgeland stability condition on $F(Y)$, with stable objects related to special Lagrangian submanifolds. Therefore, we should be able to view $B$ as a space of Bridgeland stability conditions on $F(Y)$. For $t \in B$, 
the closed special Lagrangian submanifolds obtained by suspension of open special Lagrangian submanifolds in $M$
with boundary on the fiber $\pi^{-1}(t)$
should be $t$-stable objects of $F(Y)$.
Furthermore, the phase of the central charge should be related to the phase of the special Lagrangian.

\subsection{Mirror symmetry for $K_{\PP^2}$}
The last point is to remark that the noncompact Calabi-Yau 3-fold $Y_t$ is mirror to the noncompact Calabi-Yau 3-fold $K_{\PP^2}$, see \cite{MR2651908} \cite{chan2016lagrangian}.
In particular, we have $F(Y)=\D^b_c(K_{\PP^2})$, where $\D^b_c(K_{\PP^2})$ is the derived category of compactly supported sheaves on $K_{\PP^2}$. We have 
$\D^b_0(K_{\PP^2}) \subset \D^b_c(K_{\PP^2})$. Thus, $B$ should be identified with a space of stability conditions on $\D^b_c(K_{\PP^2})$.
Furthermore, for every $t \in B$, counts of special Lagrangians in 
$Y_t$ should coincide with counts of $t$-stable objects in $\D^b_c(K_{\PP^2})$.

\subsection{Conclusion}
Combining the previous steps, starting with a
holomorphic curve in $V=\PP^2-E$ with boundary on 
$\pi^{-1}(t)$, we obtain an open special Lagrangian submanifold in $M$ with boundary on $\pi^{-1}(t)$, then a closed special Lagrangian submanifold in $Y_t$, and finally a
$t$-stable object in $\D^b_c(K_{\PP^2})$.
The special Lagrangians submanifolds obtained from holomorphic curves by hyperkähler rotation have a specific phase, 
and so the resulting $t$-stable objects in 
$\D^b_c(K_{\PP^2})$ have a central charge of specific phase. This explains why 
$S(\fD^{\iin}_{y^+})$ describes semistable objects with a central charge of specific phase (more precisely, purely imaginary).

It is a heuristic explanation and not a proof for at least two reasons.
First, directly constructing Bridgeland stability conditions on Fukaya categories such as $F(Y)$ is quite non-trivial.
Second, even assuming that we can directly prove the existence of the expected Bridgeland stability conditions on $F(Y)$, one should prove that in the hyperkähler rotation step, the virtual counts of holomorphic curves given by Gromov-Witten theory coincide with the virtual counts of special Lagrangians given by applying Donaldson-Thomas theory to $F(Y)$. At the higher-genus/refined level, as in  \cref{section_higher_genus_refinement}, one should understand why counts of higher-genus open curves coincide after the change of variables $y=e^{i\hbar}$ with the refined Donaldson-Thomas invariants extracted from the moduli spaces of special Lagrangians. 
The difficulty is in the virtual aspect of both sides (e.g.\ curves in Gromov-Witten theory are far from embedded in general). A non-trivial indication\footnote{I thank Vivek Shende for a discussion related to this point.} for such correspondence is the known result
\cite{MR2350052, MR2453601, MR3221294} that higher-genus Gromov-Witten invariants of a local curve $\Sigma$ are related through the change of variables 
$y=e^{i\hbar}$ to the weight polynomials of the $GL_n$ character varieties of $\Sigma$ (moduli spaces of $GL_n$ local systems on $\Sigma$ should be thought as local contributions of 
$\Sigma$ to the moduli space of objects 
of the Fukaya category), but a general argument seems to be lacking.

\vspace{+8 pt}
\noindent
University of Georgia \\
Department of Mathematics \\
Athens, GA 30605 \\
Pierrick.Bousseau@uga.edu

\end{document}